\documentclass[12pt,reqno]{amsart}
\usepackage{amsmath,amssymb,amsthm,mathrsfs}

\usepackage{graphicx, cite, times}
\usepackage{cases}
\usepackage{xcolor}
\usepackage{appendix}
\usepackage{enumerate}
\usepackage{soul}
\usepackage[colorlinks, linkcolor=blue, citecolor=blue]{hyperref}
\usepackage{xpatch}
\usepackage{bm}
\soulregister{\cite}7
\usepackage{enumitem} 
\usepackage{appendix}
\setlist[enumerate,1]{label=\qquad Case \arabic*:}

\theoremstyle{definition}
\newtheorem{theo}{Theorem}[section]
\newtheorem{coro}[theo]{Corollary}
\newtheorem{lemm}[theo]{Lemma}
\newtheorem{prop}[theo]{Proposition}
\newtheorem{rema}[theo]{ Remark}

\newtheorem{assu}{Assumption}
\numberwithin{equation}{section}

\begin{document}
\title[Shape Derivatives for  Maxwell's Equations] {Shape Derivatives for  Maxwell's Equations with Nonlinear Boundary Conditions
}

\author{Chao Deng}
\address{School of
Mathematics and Statistics, Center for Mathematics and
Interdisciplinary Sciences, Northeast Normal University, Changchun, Jilin 130024, P.R.China. }
\email{dengchao101@nenu.edu.cn}

\author{Yixian Gao}
\address{School of
Mathematics and Statistics, Center for Mathematics and
Interdisciplinary Sciences, Northeast Normal University, Changchun, Jilin 130024, P.R.China. }
\email{gaoyx643@nenu.edu.cn}

\thanks{The research of Y. Gao was supported by the NSFC
(project number, 12371187) and Science and Technology Development Plan Project of Jilin Province (project number, 20240101006JJ)}


\begin{abstract}
This paper develops a trace-regular variational framework for time-harmonic Maxwell scattering problems involving pointwise nonlinear boundary and interface responses. We investigate three canonical classes of models: nonlinear impedance, nonlinear perfect electric conductor, and nonlinear transmission conditions. Since the standard Maxwell tangential trace belongs to a space of negative order, the nonlinearities are formulated in refined functional spaces where the tangential electric field admits an $L^2(\Gamma)$-trace. Under the assumption of a sufficiently small Lipschitz constant for the nonlinear response, we establish the well-posedness of the direct problems via fixed-point arguments leveraging the mapping properties of the associated linear Maxwell operators.
Within this framework, we perform a rigorous sensitivity analysis of the electromagnetic fields with respect to perturbations of the scattering interface. By employing the covariant Piola transform, we prove the continuity and Fr\'echet differentiability of the pulled-back solutions with respect to domain variations. The material derivative is characterized as the unique solution to a corresponding $\mathbb{R}$-linearized Maxwell system, and the shape derivative is shown to satisfy explicit boundary or interface conditions for each of the three nonlinear models. We further demonstrate that the resulting sensitivity expressions possess the Hadamard structure, depending exclusively on the normal component of the boundary deformation. The resulting derivative characterizations provide a mathematical
basis for subsequent adjoint-based sensitivity analysis, shape optimization,
and gradient-driven inverse reconstruction.
\end{abstract}

\maketitle

\section{Introduction}

The inverse scattering problem, which entails reconstructing the geometry or boundary physical properties of an obstacle from measurements of the scattered field, represents a fundamental challenge in wave propagation. In the linear regime, a comprehensive mathematical theory has been established \cite{Colton2013, Kirsch2008}, underpinned by the variational analysis of Maxwell's equations in Hilbert spaces \cite{BuffaCostabelSheen2002} and the systematic framework of shape calculus \cite{SokolowskiZolesio1992, Delfour2011}. However, in numerous emerging applications, the interaction between the electromagnetic wave and the scatterer transcends linear descriptions, necessitating a rigorous treatment of Maxwell scattering problems governed by nonlinear boundary and interface responses.

Nonlinear boundary responses arise naturally in the modeling of thin coatings, nonlinear dielectric layers, and engineered media. For thin coatings, effective boundary conditions are typically derived through the asymptotic reduction of volumetric coating models \cite{Senior1995, Haddar2002}. While linear thin-layer approximations are well-established \cite{Ammari1998, Bendali1996, Durufle2006}, the inclusion of nonlinear effects in coatings and dielectric layers leads to more intricate boundary or interface laws \cite{Norgren1999, Shestopalov2010, Smirnov2004}. Further physical impetus for these models is provided by recent advancements in metamaterials and engineered surfaces \cite{engheta2006metamaterials, Kadic2019}. In these configurations, the effective response often exhibits a non-trivial dependence on the field strength, characterized by pointwise nonlinearities that significantly complicate the analytical and numerical treatment of the underlying vector wave equations \cite{Norgren1999, Shestopalov2010, Smirnov2004}.

The preponderance of existing mathematical research on nonlinear scattering has focused on scalar models, exemplified by nonlinear Helmholtz equations and nonlinear inclusions \cite{Griesmaier2022, Lechleiter2011}. Nonlinear boundary conditions have also been investigated within the context of acoustic wave scattering, particularly through the use of boundary integral and convolution quadrature techniques \cite{BanjaiLubich2019, Banjai2016}. The vector electromagnetic setting, however, is considerably more involved; the presence of the curl operator and the fact that natural boundary quantities are tangential traces, often residing in spaces of negative order necessitate a more sophisticated functional analytic approach.

Shape sensitivity for linear scattering problems has been extensively studied using both boundary integral and variational methods \cite{Hettlich1995, Kirsch1989, Kress2018}. In the specific case of electromagnetic obstacle scattering, differentiability results have been established for obstacles subject to linear impedance conditions via boundary integral equations \cite{HaddarKress2004}. Within the variational framework for time-harmonic Maxwell equations, the curl-preserving covariant Piola transformation serves as a fundamental tool for comparing fields across perturbed domains and for differentiating curl-related bilinear forms on a fixed reference configuration \cite{Hettlich2012}. While recent advances have led to the development of high-order shape Taylor expansions and robust computational procedures for scattering problems under classical (linear) boundary conditions \cite{bao2025computation, bao2026shape}, the extension of these sensitivity analysis techniques to Maxwell systems featuring pointwise nonlinear boundary and interface responses remains a largely unexplored territory.

The principal technical challenge in the nonlinear electromagnetic setting lies in the intricate coupling between geometric variations and nonlinear tangential trace operators. A representative nonlinear impedance condition involves a boundary response of the form
\begin{equation*}
    \mathbf{g}(\mathbf{x}, \gamma_T \mathbf{E}),
\end{equation*}
where $\gamma_T \mathbf{E}$ denotes the tangential component of the electric field. Under a domain perturbation, this term is affected by a multifaceted variation involving the field itself, the unit normal vector, the tangential projection operator, the surface Jacobian, and the $\mathbb{R}$-linear derivative of the nonlinear response map. Critically, for a general field in the standard energy space $H(\operatorname{curl}; \Omega)$, the tangential trace resides in a space of negative order \cite{BuffaCostabelSheen2002}. Such distributional regularity is insufficient to characterize a pointwise nonlinear response, necessitating the development of a trace-regular Maxwell framework wherein the tangential traces possess $L^2$-regularity.

This paper aims to establish a trace-regular variational setting for Maxwell scattering problems featuring pointwise nonlinear boundary and interface responses, and to rigorously analyze their sensitivity with respect to domain perturbations. While the three nonlinear models considered herein impedance, transmission, and perfect conductor are unified under the same trace-regularity principle, their variational realizations diverge according to the mathematical structure of the boundary laws. Specifically, the nonlinear impedance and transmission problems are addressed through primal curl--curl formulations, whereas the nonlinear perfect conductor case, being an essential-type condition, is incorporated via a mixed dual formulation. Building upon these fixed-domain formulations, we utilize a unified perturbation framework based on the covariant Piola transform to derive the continuity, material differentiability, and Eulerian shape differentiability of the electromagnetic solutions.

The primary contributions of this work are summarized as follows:

\begin{itemize}
    \item \textbf{Trace-regular nonlinear Maxwell formulations.} We establish a rigorous functional analytic framework for nonlinear Maxwell scattering problems by employing specialized subspaces where the tangential electric field admits an $L^2(\Gamma)$-trace. This regularity is essential for the pointwise characterization of nonlinear boundary and interface responses. Our framework unifies three canonical models: nonlinear impedance, nonlinear transmission, and nonlinear perfect electric conductors. While the first two models are addressed via primal curl--curl variational formulations, the latter is treated through a mixed dual formulation to accommodate its essential-type boundary condition. Existence and uniqueness of solutions are established via fixed-point arguments, leveraging the well-posedness of the associated linear Maxwell operators under suitable smallness assumptions on the nonlinearities.

    \item \textbf{Unified domain-dependence analysis.} We provide a comprehensive analysis of the solution's dependence on domain variations. By utilizing the curl-preserving covariant Piola transformation, we map the perturbed problems back to a fixed reference configuration. This leads to a unified perturbative structure for the volume integrals, the transported tangential traces, and the nonlinear boundary terms. Within this setting, we prove the continuity of the pulled-back solutions with respect to admissible domain perturbations in the $C^1$-topology.

    \item \textbf{Characterization of material and Eulerian shape derivatives.} We establish the Fr\'echet differentiability of the electromagnetic fields with respect to shape variations. The material derivatives are characterized as unique solutions to corresponding $\mathbb{R}$-linearized Maxwell systems. By appropriately accounting for the convective and geometric contributions of the domain transformation, we derive the Eulerian shape derivatives and provide explicit linearized boundary or interface conditions for each nonlinear model. Finally, we rigorously verify that the resulting sensitivity expressions satisfy the Hadamard structure, confirming that the shape derivative depends exclusively on the normal component of the boundary deformation.
\end{itemize}

The resulting linearized systems provide a rigorous mathematical foundation for future sensitivity analysis and shape optimization, although adjoint formulations and numerical reconstruction algorithms are beyond the scope of the present paper.

The remainder of this paper is organized as follows. In Section~\ref{sec:2}, we introduce the nonlinear Maxwell models and their trace-regular variational formulations, including the mixed formulation for the nonlinear perfect conductor case, and establish the corresponding well-posedness results. Section~\ref{sec:3} investigates the domain dependence of the pulled-back solutions and derives the material derivatives using the covariant Piola transformation. In Section~\ref{sec:4}, we transition from material derivatives to Eulerian shape derivatives, deriving the explicit boundary and interface conditions satisfied by the linearized fields and verifying their Hadamard structure.

\section{The direct scattering problem} \label{sec:2}

Let \(D\subset\mathbb R^3\) denote a bounded domain with a Lipschitz boundary
\(\Gamma:=\partial D\). Throughout the paper, \(\mathbf n\) denotes the unit
normal on \(\Gamma\) pointing outward from \(D\) into the exterior medium. We assume that $D$ is composed of an isotropic and homogeneous material, situated within an ambient vacuum characterized by electric permittivity $\varepsilon_0 > 0$, magnetic permeability $\mu_0 > 0$, and vanishing conductivity. The constitutive properties of the obstacle are specified by its magnetic permeability $\mu_D > 0$, permittivity $\varepsilon_D > 0$, and conductivity $\sigma_D \geq 0$. Throughout this work, we adopt the time-harmonic convention with time dependence of the form $e^{-\mathrm{i}\omega t}$, where $\omega > 0$ denotes the angular frequency.

The material contrast across the interface $\Gamma$ is defined by the relative permeability $\mu_r$ and the complex-valued relative permittivity $\varepsilon_r$. These constitutive parameters are given by the piecewise constant functions
\[
    \mu_r(\boldsymbol{x}) =
    \begin{cases}
        \mu_D/\mu_0, & \boldsymbol{x} \in D, \\
        1, & \boldsymbol{x} \in \mathbb{R}^3 \setminus \overline{D},
    \end{cases}
    \qquad
    \varepsilon_r(\boldsymbol{x}) =
    \begin{cases}
        \dfrac{1}{\varepsilon_0}
        \left( \varepsilon_D + \dfrac{\mathrm{i}\sigma_D}{\omega} \right),
        & \boldsymbol{x} \in D, \\
        1, & \boldsymbol{x} \in \mathbb{R}^3 \setminus \overline{D}.
    \end{cases}
\]
This general formulation encompasses both purely dielectric ($\sigma_D = 0$) and lossy conductive ($\sigma_D > 0$) penetrable obstacles. 

Let $k = \omega\sqrt{\varepsilon_0\mu_0} > 0$ denote the wavenumber in the ambient vacuum. The obstacle is excited by an incident field $\mathbf{E}^i$ in the form of a plane wave
\[
\mathbf{E}^i(\boldsymbol{x}) = \boldsymbol{p} \, e^{\mathrm{i}k\boldsymbol{d}\cdot\boldsymbol{x}},
\]
where the propagation direction $\boldsymbol{d} \in \mathbb{S}^2$ and the polarization vector $\boldsymbol{p} \in \mathbb{C}^3$ satisfy the transversality condition $\boldsymbol{p} \cdot \boldsymbol{d} = 0$.
The interaction of the incident field $\mathbf{E}^i$ with the obstacle generates a scattered electric field $\mathbf{E}^s$, such that the total electric field is given by $\mathbf{E} := \mathbf{E}^i + \mathbf{E}^s$. For a penetrable obstacle, the total field satisfies the second-order time-harmonic Maxwell system
\begin{equation}\label{eq:general_transmission}
\nabla \times
\left(
\frac{1}{\mu_r}\nabla \times \mathbf{E}
\right)
-
k^2\varepsilon_r\mathbf{E}
=
0
\qquad
\text{in } \mathbb{R}^3 \setminus \Gamma.
\end{equation}

In the case of impenetrable obstacles---exemplified by those subject to perfectly conducting or impedance-type boundary conditions---the electromagnetic field is confined to the exterior region $\mathbb{R}^3 \setminus \overline{D}$. Under this configuration, the interior $D$ does not support wave propagation. Furthermore, the exterior medium is identified with the vacuum background, yielding the constitutive relations $\mu_r \equiv \varepsilon_r \equiv 1$ in $\mathbb{R}^3 \setminus \overline{D}$. Consequently, the governing equation \eqref{eq:general_transmission} reduces to the homogeneous vector wave equation
\begin{equation}\label{eq:exterior_equation}
\nabla \times (\nabla \times \mathbf{E}) - k^2 \mathbf{E} = 0 \qquad \text{in } \mathbb{R}^3 \setminus \overline{D}.
\end{equation}

We investigate three canonical classes of nonlinear boundary and transmission conditions on the interface $\Gamma$.

\begin{itemize}
    \item \textbf{Nonlinear Impedance Boundary Condition (NIBC).} 
    The electric field on $\Gamma$ is governed by the relation
    \begin{equation}\label{eq:NIBC}
        \mathbf{n} \times (\nabla \times \mathbf{E}) + \mathrm{i}k\lambda\mathbf{E}_T = \mathbf{g}(\boldsymbol{x}, \mathbf{E}_T),
    \end{equation}
    where $\lambda \in \mathbb{C}$ is a constant impedance parameter with $\operatorname{Re}\lambda > 0$. The positivity of the real part of $\lambda$ accounts for physical energy dissipation at the boundary and ensures the well-posedness of the associated linear problem. Here, $\mathbf{E}_T:= \mathbf{n} \times (\mathbf{E} \times \mathbf{n})$ denotes the tangential component of the electric field, and $\mathbf{g} \colon \Gamma \times \mathbb{C}^3 \to \mathbb{C}^3$ represents a nonlinear response acting pointwise on $\mathbf{E}_T$.

    \item \textbf{Nonlinear Perfect Electric Conductor Condition (NPEC).} 
    In this configuration, the boundary condition is prescribed as
    \begin{equation}\label{eq:npc}
        \mathbf{n} \times \mathbf{E} = \mathbf{g}(\boldsymbol{x}, \mathbf{E}_T) \qquad \text{on } \Gamma.
    \end{equation}
    In the limit where $\mathbf{g} \equiv 0$, the expression \eqref{eq:npc} reduces to the classical perfect electric conductor (PEC) condition, $\mathbf{n} \times \mathbf{E} = 0$.

    \item \textbf{Nonlinear Transmission Condition (NTC).} 
    For penetrable obstacles, the electromagnetic field satisfies the following transmission system across the interface:
    \begin{equation}\label{eq:transmission}
    \begin{aligned}
        \nabla \times \left( \frac{1}{\mu_r} \nabla \times \mathbf{E} \right) - k^2 \varepsilon_r \mathbf{E} &= 0 && \text{in } D \cup (\mathbb{R}^3 \setminus \overline{D}), \\
        [\mathbf{n} \times \mathbf{E}] &= 0 && \text{on } \Gamma, \\
        \left[ \mathbf{n} \times \left( \frac{1}{\mu_r} \nabla \times \mathbf{E} \right) \right] &= \mathbf{g}(\boldsymbol{x}, \mathbf{E}_T) && \text{on } \Gamma.
    \end{aligned}
    \end{equation}
    Here, the operator $[\cdot]$ denotes the jump across $\Gamma$, defined by $[f]:= f^+ - f^-$, where $f^+$ and $f^-$ designate the traces from the exterior $\mathbb{R}^3 \setminus \overline{D}$ and the interior $D$, respectively. Under the continuity of the tangential trace $[\mathbf{n} \times \mathbf{E}] = 0$, the tangential component $\mathbf{E}_T$ is uniquely well-defined on $\Gamma$.

    From a physical perspective, this formulation arises when the jump in the tangential magnetic field $\mathbf{H} = (\mathrm{i}k\mu_r)^{-1} \nabla \times \mathbf{E}$ is prescribed via a nonlinear constitutive relation of the form $[\mathbf{n} \times \mathbf{H}] = \boldsymbol{\psi}(\boldsymbol{x}, \mathbf{E}_T)$. The resulting interface term is then identified as $\mathbf{g}(\boldsymbol{x}, \mathbf{E}_T):= \mathrm{i}k \boldsymbol{\psi}(\boldsymbol{x}, \mathbf{E}_T)$.
\end{itemize}

To facilitate the variational analysis, we truncate the unbounded exterior region by enclosing the obstacle $D$ within an open ball $B_R$ of sufficiently large radius $R$, such that $\overline{D} \subset B_R$. Let $\Gamma_R:= \partial B_R$ denote the artificial outer boundary. The computational domain $\Omega$ is defined according to the specific scattering configuration:
\[
\Omega:=
\begin{cases}
B_R \setminus \overline{D} & \text{for impenetrable obstacles (NIBC and NPEC),} \\
B_R & \text{for penetrable obstacles (NTC).}
\end{cases}
\]
Accordingly, the boundary $\partial\Omega$ comprises the artificial boundary $\Gamma_R$ and, in the impenetrable case, the physical interface $\Gamma$. For the penetrable model, $\Gamma$ constitutes an interior interface within the domain $\Omega = B_R$.

The total field $\mathbf{E}$ is sought in the Hilbert space
\[
H(\operatorname{curl}, \Omega):= \left\{ \mathbf{u} \in L^2(\Omega)^3 : \nabla \times \mathbf{u} \in L^2(\Omega)^3 \right\}.
\]
In the penetrable setting, the first transmission condition---prescribing the continuity of the tangential trace $[\mathbf{n} \times \mathbf{E}] = 0$---is implicitly satisfied by the requirement $\mathbf{E} \in H(\operatorname{curl}, B_R)$. The second transmission condition, involving the jump in the magnetic trace, is subsequently realized as a natural condition within the variational formulation.

The Silver--M\"uller radiation condition is prescribed on the artificial boundary $\Gamma_R$ through the electromagnetic Calder\'on operator, which yields an exact transparent boundary condition for the scattered field. For the boundary manifolds considered herein, we define the tangential trace operator $\gamma_t$ and the tangential component operator $\gamma_T$ by
\[
\gamma_t \mathbf{u} := \mathbf{n} \times \mathbf{u}, \qquad \gamma_T \mathbf{u} := \mathbf{n} \times (\mathbf{u} \times \mathbf{n}),
\]
where the unit normal $\mathbf{n}$ is consistent with the orientations previously established for the respective boundary components.

On a Lipschitz boundary, the Maxwell tangential trace mappings satisfy the well-established continuity properties
\[
\gamma_t \colon H(\operatorname{curl}; \Omega) \to H^{-1/2}(\operatorname{div}_{\partial\Omega}, \partial\Omega), \qquad \gamma_T \colon H(\operatorname{curl}; \Omega) \to H^{-1/2}(\operatorname{curl}_{\partial\Omega}, \partial\Omega)
\]
(cf. \cite{BuffaCostabelSheen2002}). We denote the surface divergence and surface curl operators on the respective boundary components by $\operatorname{div}_S$ and $\operatorname{curl}_S$ for $S \in \{\Gamma, \Gamma_R\}$. These trace spaces are of negative fractional order, reflecting the fact that the tangential trace of an arbitrary field in the energy space $H(\operatorname{curl}; \Omega)$ generally lacks $L^2(\partial\Omega)^3$ regularity. Since the nonlinear boundary responses considered in this work are defined pointwise with respect to the tangential electric field, their rigorous characterization necessitates the use of trace-regular subspaces where the tangential traces possess $L^2$-representatives on the interface $\Gamma$.

The electromagnetic Calder\'on operator (cf. \cite{KirschMonk1995, Monk2003}) is realized on the artificial boundary $\Gamma_R$ as a mapping 
\[
\Lambda \colon H^{-1/2}(\operatorname{div}_{\Gamma_R}, \Gamma_R) \to H^{-1/2}(\operatorname{div}_{\Gamma_R}, \Gamma_R).
\]
For a given tangential field $\boldsymbol{\xi} \in H^{-1/2}(\operatorname{div}_{\Gamma_R}, \Gamma_R)$, let $\mathbf{W} \in H_{\mathrm{loc}}(\operatorname{curl}, \mathbb{R}^3 \setminus \overline{B_R})$ be the unique radiating solution to the exterior boundary value problem
\[
\begin{cases}
\nabla \times \nabla \times \mathbf{W} - k^2 \mathbf{W} = 0 & \text{in } \mathbb{R}^3 \setminus \overline{B_R}, \\
\gamma_t \mathbf{W} = \boldsymbol{\xi} & \text{on } \Gamma_R.
\end{cases}
\]
The operator $\Lambda$ is characterized by the relation $\Lambda \boldsymbol{\xi} := \gamma_t((\mathrm{i}k)^{-1} \nabla \times \mathbf{W})$, which yields the non-local identity $\gamma_t(\nabla \times \mathbf{W}) = \mathrm{i}k \Lambda(\gamma_t \mathbf{W})$ on $\Gamma_R$. When applied to the scattered field $\mathbf{E}^s$, this provides the exact transparent boundary condition
\begin{equation}\label{eq:calderon_condition1}
\gamma_t(\nabla \times \mathbf{E}^s) = \mathrm{i}k \Lambda(\gamma_t \mathbf{E}^s) \qquad \text{on } \Gamma_R.
\end{equation}
In what follows, boundary integrals on $\Gamma_R$ involving the Calder\'on operator are to be understood as duality pairings between the Maxwell trace spaces $H^{-1/2}(\operatorname{div}_{\Gamma_R}, \Gamma_R)$ and $H^{-1/2}(\operatorname{curl}_{\Gamma_R}, \Gamma_R)$. Specifically, we adopt the formal integral notation
\[
\int_{\Gamma_R} (\Lambda \gamma_t \mathbf{U}) \cdot \gamma_T \overline{\mathbf{V}} \, \mathrm{d}s := \left\langle \Lambda \gamma_t \mathbf{U}, \gamma_T \overline{\mathbf{V}} \right\rangle_{\Gamma_R}
\]
primarily for notational convenience; the analytical interpretation remains strictly that of the underlying duality pairing between the trace spaces.

For the well-posedness analysis, we conduct our study within a trace-regular Maxwell space. Specifically, we define the Hilbert space
\[
\mathbf{X}(\Omega) := \left\{ \mathbf{u} \in H(\operatorname{curl}; \Omega) : \gamma_T \mathbf{u} \in L^2(\Gamma)^3 \right\},
\]
where $H(\operatorname{curl}; \Omega) := \{ \mathbf{u} \in L^2(\Omega)^3 : \nabla \times \mathbf{u} \in L^2(\Omega)^3 \}$. The space $\mathbf{X}(\Omega)$ is endowed with the natural augmented norm
\[
\|\mathbf{u}\|_{\mathbf{X}(\Omega)}^2 := \|\mathbf{u}\|_{H(\operatorname{curl}; \Omega)}^2 + \|\gamma_T \mathbf{u}\|_{L^2(\Gamma)}^2.
\]
Since the tangential component $\gamma_T \mathbf{u} = \mathbf{n} \times (\mathbf{u} \times \mathbf{n})$ is inherently tangential, the requirement $\gamma_T \mathbf{u} \in L^2(\Gamma)^3$ is equivalent to seeking traces in the subspace 
\[
L_t^2(\Gamma) := \left\{ \boldsymbol{\zeta} \in L^2(\Gamma)^3 : \boldsymbol{\zeta} \cdot \mathbf{n} = 0 \text{ a.e. on } \Gamma \right\}.
\]

This enhanced $L^2$-regularity is necessitated by the pointwise definition of the nonlinear boundary responses with respect to $\gamma_T \mathbf{E}$. Consequently, the nonlinear boundary integrals on $\Gamma$ are rigorously interpreted as $L^2(\Gamma)^3$ inner products or, more generally, as $L^2$-trace duality pairings, ensuring the variational formulation is well-defined.

The nonlinear response map
\[
\mathbf{g} \colon \Gamma \times \mathbb{C}^3 \to \mathbb{C}^3
\]
is assumed to be a Carathe\'odory function; that is, $\mathbf{g}(\cdot, \boldsymbol{z})$ is measurable for every $\boldsymbol{z} \in \mathbb{C}^3$, and $\mathbf{g}(\boldsymbol{x}, \cdot)$ is continuous for almost every $\boldsymbol{x} \in \Gamma$. Since the boundary operators in the nonlinear impedance and transmission conditions map into tangential fields, the nonlinear response is interpreted as a tangential mapping. Accordingly, we impose the tangency constraint
\[
\mathbf{g}(\boldsymbol{x}, \boldsymbol{z}) \cdot \mathbf{n}(\boldsymbol{x}) = 0
\]
for all tangential arguments $\boldsymbol{z}$. For any $\boldsymbol{\zeta} \in L_t^2(\Gamma)$, we consider the induced Nemytskii operator
\[
\mathcal{G}(\boldsymbol{\zeta})(\boldsymbol{x}) := \mathbf{g}(\boldsymbol{x}, \boldsymbol{\zeta}(\boldsymbol{x})).
\]
We assume that $\mathcal{G}$ maps $L_t^2(\Gamma)$ into itself and is Lipschitz continuous. Specifically, we assume there exists a constant $L_g > 0$ such that
\begin{equation}
\label{eq:g_lipschitz_L2}
\left\| \mathcal{G}(\boldsymbol{\zeta}_1) - \mathcal{G}(\boldsymbol{\zeta}_2) \right\|_{L^2(\Gamma)} \le L_g \|\boldsymbol{\zeta}_1 - \boldsymbol{\zeta}_2\|_{L^2(\Gamma)}
\end{equation}
for all $\boldsymbol{\zeta}_1, \boldsymbol{\zeta}_2 \in L_t^2(\Gamma)$. These assumptions guarantee that the pointwise nonlinear boundary response is well-defined and amenable to analysis within the trace-regular variational framework.

We begin by establishing the well-posedness of the variational formulation for the nonlinear impedance boundary condition. The nonlinear PEC case, which necessitates an auxiliary trace compatibility constraint, is addressed in a subsequent section via a boundary lifting procedure. The nonlinear transmission problem is investigated in the sequel within the context of a coefficient-dependent trace-regular space.

\subsection{The nonlinear impedance boundary condition}

The variational characterization of the nonlinear impedance boundary value problem \eqref{eq:exterior_equation}--\eqref{eq:NIBC} entails finding a field $\mathbf{E} \in \mathbf{X}(\Omega)$ such that
\begin{equation}
\label{eq:main_formulation}
\mathcal{A}(\mathbf{E}; \mathbf{V}) = \langle \mathbf{F}, \mathbf{V} \rangle, \qquad \forall \, \mathbf{V} \in \mathbf{X}(\Omega).
\end{equation}
The semi-linear form $\mathcal{A} \colon \mathbf{X}(\Omega) \times \mathbf{X}(\Omega) \to \mathbb{C}$, which is linear with respect to its second argument, is prescribed by
\begin{equation}
\label{sef}
\begin{aligned}
\mathcal{A}(\mathbf{E}; \mathbf{V}) 
&:= \int_\Omega \left[ (\nabla \times \mathbf{E}) \cdot (\nabla \times \overline{\mathbf{V}}) - k^2 \mathbf{E} \cdot \overline{\mathbf{V}} \right] \, \mathrm{d}\boldsymbol {x} \\
&\quad + \mathrm{i}k\lambda \int_\Gamma (\gamma_T \mathbf{E}) \cdot (\gamma_T \overline{\mathbf{V}}) \, \mathrm{d}s - \int_\Gamma \mathbf{g}(\cdot, \gamma_T \mathbf{E}) \cdot (\gamma_T \overline{\mathbf{V}}) \, \mathrm{d}s \\
&\quad + \mathrm{i}k \left\langle \Lambda \gamma_t \mathbf{E}, \gamma_T \overline{\mathbf{V}} \right\rangle_{\Gamma_R}.
\end{aligned}
\end{equation}
The linear excitation functional $\mathbf{F} \in \mathbf{X}(\Omega)^*$, which incorporates the contribution of the incident field at the artificial boundary $\Gamma_R$, is given by
\begin{equation}
\label{eq:F_definition}
\langle \mathbf{F}, \mathbf{V} \rangle := \left\langle \mathrm{i}k \Lambda(\gamma_t \mathbf{E}^i) - \gamma_t (\nabla \times \mathbf{E}^i), \gamma_T \overline{\mathbf{V}} \right\rangle_{\Gamma_R}.
\end{equation}

Let $\mathscr {A}_l \in \mathcal{L}(\mathbf{X}(\Omega), \mathbf{X}(\Omega)^*)$ denote the bounded linear operator associated with the sesquilinear part of the impedance problem, defined by
\begin{equation}
\label{eq:Al_definition}
\begin{aligned}
\langle \mathscr {A}_l \mathbf{E}, \mathbf{V} \rangle 
&:= \int_\Omega \left[ (\nabla \times \mathbf{E}) \cdot (\nabla \times \overline{\mathbf{V}}) - k^2 \mathbf{E} \cdot \overline{\mathbf{V}} \right] \, \mathrm{d}\boldsymbol {x} \\
&\quad + \mathrm{i} k \lambda \int_\Gamma (\gamma_T \mathbf{E}) \cdot (\gamma_T \overline{\mathbf{V}}) \, \mathrm{d}s + \mathrm{i} k \left\langle \Lambda \gamma_t \mathbf{E}, \gamma_T \overline{\mathbf{V}} \right\rangle_{\Gamma_R}.
\end{aligned}
\end{equation}
By appealing to the well-established weak well-posedness theory for linear Maxwell impedance problems in trace-regular spaces (see, e.g., \cite[Theorem 4.17]{Monk2003}), it follows that for any prescribed tangential boundary datum $\mathbf{f} \in L_t^2(\Gamma)$, the linear variational problem
\begin{equation}
\label{eq:linear_impedance_data_problem}
\langle \mathscr {A}_l \mathbf{E}, \mathbf{V} \rangle = \langle \mathbf{F}, \mathbf{V} \rangle + \int_\Gamma \mathbf{f} \cdot \gamma_T \overline{\mathbf{V}} \, \mathrm{d}s, \qquad \forall \, \mathbf{V} \in \mathbf{X}(\Omega)
\end{equation}
admits a unique weak solution $\mathbf{E} \in \mathbf{X}(\Omega)$. Furthermore, the solution satisfies a stability estimate with respect to the boundary data: if $\mathbf{E}_1, \mathbf{E}_2 \in \mathbf{X}(\Omega)$ denote the solutions corresponding to data $\mathbf{f}_1, \mathbf{f}_2 \in L_t^2(\Gamma)$, respectively, then there exists a stability constant $C_{\rm imp} > 0$ such that
\begin{equation}
\label{eq:linear_stability}
\|\mathbf{E}_1 - \mathbf{E}_2\|_{\mathbf{X}(\Omega)} \le C_{\rm imp} \|\mathbf{f}_1 - \mathbf{f}_2\|_{L^2(\Gamma)}.
\end{equation}

\begin{theo}
\label{th:existence_nibc}
Suppose that the nonlinear response $\mathbf{g}$ satisfies the Lipschitz condition \eqref{eq:g_lipschitz_L2}. Let $C_{\rm imp}$ denote the stability constant of the associated linear impedance problem. If 
\[
C_{\rm imp} L_g < 1,
\]
then the nonlinear impedance problem \eqref{eq:main_formulation} possesses a unique weak solution $\mathbf{E} \in \mathbf{X}(\Omega)$.
\end{theo}

\begin{proof}
For any $\mathbf{U} \in \mathbf{X}(\Omega)$, we define the solution operator $\mathcal{T} \colon \mathbf{X}(\Omega) \to \mathbf{X}(\Omega)$ such that $\mathcal{T}\mathbf{U}$ is the unique solution to the linear variational problem
\begin{equation}\label{eq:fixed_point_op}
\langle \mathscr {A}_l (\mathcal{T}\mathbf{U}), \mathbf{V} \rangle = \langle \mathbf{F}, \mathbf{V} \rangle + \int_\Gamma \mathbf{g}(\cdot, \gamma_T \mathbf{U}) \cdot \gamma_T \overline{\mathbf{V}} \, \mathrm{d}s, \qquad \forall \, \mathbf{V} \in \mathbf{X}(\Omega).
\end{equation}
The mapping $\mathcal{T}$ is well-defined since $\gamma_T \mathbf{U} \in L_t^2(\Gamma)$, and by the properties of $\mathbf{g}$, the term $\mathbf{g}(\cdot, \gamma_T \mathbf{U})$ constitutes an admissible tangential boundary datum in $L_t^2(\Gamma)$.

Let $\mathbf{U}_1, \mathbf{U}_2 \in \mathbf{X}(\Omega)$ be arbitrary, and let $\mathbf{E}_j := \mathcal{T}\mathbf{U}_j$ for $j=1,2$. Applying the linear stability estimate \eqref{eq:linear_stability} with the boundary data $\mathbf{g}(\cdot, \gamma_T \mathbf{U}_1)$ and $\mathbf{g}(\cdot, \gamma_T \mathbf{U}_2)$, we obtain
\[
\|\mathbf{E}_1 - \mathbf{E}_2\|_{\mathbf{X}(\Omega)} \le C_{\rm imp} \left\| \mathbf{g}(\cdot, \gamma_T \mathbf{U}_1) - \mathbf{g}(\cdot, \gamma_T \mathbf{U}_2) \right\|_{L^2(\Gamma)}.
\]
Invoking the Lipschitz continuity of the nonlinear response \eqref{eq:g_lipschitz_L2}, it follows that
\[
\begin{aligned}
\|\mathcal{T}\mathbf{U}_1 - \mathcal{T}\mathbf{U}_2\|_{\mathbf{X}(\Omega)} &\le C_{\rm imp} L_g \|\gamma_T(\mathbf{U}_1 - \mathbf{U}_2)\|_{L^2(\Gamma)} \\
&\le C_{\rm imp} L_g \|\mathbf{U}_1 - \mathbf{U}_2\|_{\mathbf{X}(\Omega)}.
\end{aligned}
\]
Under the assumption $C_{\rm imp} L_g < 1$, the operator $\mathcal{T}$ is a contraction on the Hilbert space $\mathbf{X}(\Omega)$. By the Banach fixed-point theorem, there exists a unique fixed point $\mathbf{E} \in \mathbf{X}(\Omega)$ such that $\mathcal{T}\mathbf{E} = \mathbf{E}$.

Substituting the fixed point into \eqref{eq:fixed_point_op} yields
\[
\langle \mathscr {A}_l \mathbf{E}, \mathbf{V} \rangle = \langle \mathbf{F}, \mathbf{V} \rangle + \int_\Gamma \mathbf{g}(\cdot, \gamma_T \mathbf{E}) \cdot \gamma_T \overline{\mathbf{V}} \, \mathrm{d}s, \qquad \forall \, \mathbf{V} \in \mathbf{X}(\Omega),
\]
which is precisely the variational formulation \eqref{eq:main_formulation}. Consequently, the fixed point $\mathbf{E}$ constitutes the unique weak solution to the nonlinear impedance problem.
\end{proof}

\begin{rema}
The smallness condition $C_{\rm imp} L_g < 1$ constitutes a sufficient requirement for well-posedness within the established contraction mapping framework. While this condition restricts the admissible nonlinearities to a specific Lipschitz class and is not intended to be optimal, it provides a mathematically rigorous foundation for the sensitivity analysis developed in the sequel. We note that more general solvability results could potentially be obtained by invoking monotonicity or dissipativity arguments; however, such extensions are not required for the primary objectives of this study.

To demonstrate that the class of admissible nonlinearities is non-empty, consider the following example. Let $a \in L^\infty(\Gamma)$ and $\beta \in \mathbb{C}$, and define the tangential projection operator $\mathbf{P}_T(\boldsymbol{x}) := I - \mathbf{n}(\boldsymbol{x}) \otimes \mathbf{n}(\boldsymbol{x})$. We consider the nonlinear response map
\[
\mathbf{g}(\boldsymbol{x}, \boldsymbol{z}) = \beta a(\boldsymbol{x}) \frac{\mathbf{P}_T(\boldsymbol{x}) \boldsymbol{z}}{1 + |\mathbf{P}_T(\boldsymbol{x}) \boldsymbol{z}|^2}.
\]
It is evident that $\mathbf{g}$ is purely tangential and induces a Lipschitz-continuous Nemytskii operator from $L_t^2(\Gamma)$ into itself. Specifically, the map satisfies the estimate
\[
\left\| \mathbf{g}(\cdot, \boldsymbol{\zeta}_1) - \mathbf{g}(\cdot, \boldsymbol{\zeta}_2) \right\|_{L^2(\Gamma)} \le |\beta| \, \|a\|_{L^\infty(\Gamma)} \|\boldsymbol{\zeta}_1 - \boldsymbol{\zeta}_2\|_{L^2(\Gamma)}.
\]
Consequently, the smallness condition for well-posedness is satisfied provided that the parameters satisfy
\[
C_{\rm imp} |\beta| \, \|a\|_{L^\infty(\Gamma)} < 1.
\]
\end{rema}

\subsection{The nonlinear perfect electric conductor condition (NPEC)}

We next consider the nonlinear perfect electric conductor (NPEC) boundary
condition
\begin{equation}\label{eq:npec_boundary_condition}
\gamma_t\mathbf E
=
\mathbf g(\cdot,\gamma_T\mathbf E)
\quad\text{on }\Gamma .
\end{equation}
Unlike the nonlinear impedance condition, which enters the weak formulation
as a natural boundary condition, the NPEC condition is of essential type.
In order to avoid imposing this nonlinear essential boundary condition
directly on the trial and test spaces, we switch to a dual  variational
formulation. To derive this formulation, we introduce the auxiliary variable
\[
\mathbf Q:=\nabla\times\mathbf E .
\]
Then the time-harmonic Maxwell system can be rewritten as the first-order
system
\[
\mathbf Q-\nabla\times\mathbf E=0,
\qquad
\nabla\times\mathbf Q-k^2\mathbf E=0
\quad\text{in }\Omega .
\]

Accordingly, the dual  formulation is posed in the product space
\[
\mathbb X_p(\Omega)
:=
\mathbf X(\Omega)\times H(\operatorname{curl};\Omega).
\]
Here the electric field is required to belong to the trace-regular space
\(\mathbf X(\Omega)\), because the nonlinear response depends on
\(\gamma_T\mathbf E\), whereas the auxiliary variable \(\mathbf Q\) is only
taken in \(H(\operatorname{curl};\Omega)\). In particular, no \(L^2\)-tangential
trace regularity on \(\Gamma\) is imposed on \(\mathbf Q\). The term involving
\(\Lambda^{-1}\gamma_t\mathbf Q\) on \(\Gamma_R\) is interpreted in the usual
Maxwell trace duality sense.

The dual NPEC problem is to find
\[
(\mathbf E,\mathbf Q)\in\mathbb X_p(\Omega)
\]
such that
\begin{equation}\label{eq:mixed_formulation}
\mathcal A^{\rm p}
\bigl((\mathbf E,\mathbf Q),(\boldsymbol\tau,\boldsymbol\nu)\bigr)
=
\mathcal F^{\rm p}(\boldsymbol\tau,\boldsymbol\nu),
\qquad
\forall(\boldsymbol\tau,\boldsymbol\nu)\in\mathbb X_p(\Omega),
\end{equation}
where the semilinear mixed form is defined by
\begin{equation}\label{eq:mixed_form}
\begin{aligned}
\mathcal A^{\rm p}
\bigl((\mathbf E,\mathbf Q),(\boldsymbol\tau,\boldsymbol\nu)\bigr)
&:=
\int_\Omega
\mathbf Q\cdot\overline{\boldsymbol\tau}\,dx
-
\int_\Omega
\mathbf E\cdot(\nabla\times\overline{\boldsymbol\tau})\,dx\\
&\quad
+
\frac{i}{k}
\left\langle
\Lambda^{-1}\gamma_t\mathbf Q,
\gamma_T\overline{\boldsymbol\tau}
\right\rangle_{\Gamma_R}
+
\int_\Gamma
\mathbf g(\cdot,\gamma_T\mathbf E)
\cdot
\gamma_T\overline{\boldsymbol\tau}\,ds\\
&\quad+
\int_\Omega
(\nabla\times\mathbf Q)\cdot\overline{\boldsymbol\nu}\,dx-
k^2
\int_\Omega
\mathbf E\cdot\overline{\boldsymbol\nu}\,dx .
\end{aligned}
\end{equation}
The right-hand side functional is given by
\begin{equation}\label{eq:mixed_rhs}
\mathcal F^{\rm p}(\boldsymbol\tau,\boldsymbol\nu)
:=
\left\langle
\gamma_t\mathbf E^i
+
\frac{i}{k}\Lambda^{-1}\gamma_t(\nabla\times\mathbf E^i),
\gamma_T\overline{\boldsymbol\tau}
\right\rangle_{\Gamma_R}.
\end{equation}
The mixed formulation introduced above serves as the foundation for the subsequent domain perturbation analysis. However, for established well-posedness results regarding the direct NPEC problem, it is more advantageous to employ an equivalent lifting characterization for the electric field alone. Once the electric field $\mathbf{E}$ is determined, the auxiliary magnetic variable can be recovered via $\mathbf{Q} = \nabla \times \mathbf{E}$, provided the solution satisfies the enhanced regularity condition $\nabla \times \mathbf{E} \in H(\operatorname{curl};\Omega)$.

To this end, we define the trace space of regular tangential fields as the intersection
\[
Y_t^{\mathrm{reg}}(\Gamma) := H^{-1/2}(\operatorname{div}_\Gamma, \Gamma) \cap L_t^2(\Gamma),
\]
endowed with the natural graph norm
\[
\|\boldsymbol{\eta}\|_{Y_t^{\mathrm{reg}}(\Gamma)}^2 := \|\boldsymbol{\eta}\|_{H^{-1/2}(\operatorname{div}_\Gamma, \Gamma)}^2 + \|\boldsymbol{\eta}\|_{L^2(\Gamma)}^2.
\]
When restricted to $Y_t^{\mathrm{reg}}(\Gamma)$, the existence of a bounded right inverse for the Maxwell tangential trace operator (cf. \cite{BuffaCostabelSheen2002, Monk2003}) ensures the existence of a bounded lifting operator
\[
\mathcal{R}_\Gamma \colon Y_t^{\mathrm{reg}}(\Gamma) \to \mathbf{X}(\Omega), \qquad \gamma_t(\mathcal{R}_\Gamma \boldsymbol{\eta}) = \boldsymbol{\eta} \quad \text{on } \Gamma.
\]
The validity of this construction follows from the fundamental identity $\gamma_T \mathbf{u} = -\mathbf{n} \times \gamma_t \mathbf{u}$; specifically, if $\gamma_t(\mathcal{R}_\Gamma \boldsymbol{\eta}) = \boldsymbol{\eta}$ with $\boldsymbol{\eta} \in L_t^2(\Gamma)$, then $\gamma_T(\mathcal{R}_\Gamma \boldsymbol{\eta}) = -\mathbf{n} \times \boldsymbol{\eta} \in L_t^2(\Gamma)$, thereby satisfying the defining regularity of the space $\mathbf{X}(\Omega)$.

For the NPEC boundary condition, we postulate that the nonlinear response $\mathbf{g}$ maps tangential $L^2$-fields into the regular trace space; specifically, 
\[
\mathbf{g}(\cdot, \boldsymbol{\zeta}) \in Y_t^{\mathrm{reg}}(\Gamma) \qquad \text{for all } \boldsymbol{\zeta} \in L_t^2(\Gamma).
\]
Furthermore, we require that $\mathbf{g}$ satisfies a uniform Lipschitz estimate with respect to the trace regularity, namely,
\begin{equation}\label{eq:NPEC_lipschitz_trace}
\left\| \mathbf{g}(\cdot, \boldsymbol{\zeta}_1) - \mathbf{g}(\cdot, \boldsymbol{\zeta}_2) \right\|_{Y_t^{\mathrm{reg}}(\Gamma)} \le L_{\mathrm{pec}} \|\boldsymbol{\zeta}_1 - \boldsymbol{\zeta}_2\|_{L^2(\Gamma)}
\end{equation}
for a constant $L_{\mathrm{pec}} > 0$ and all $\boldsymbol{\zeta}_1, \boldsymbol{\zeta}_2 \in L_t^2(\Gamma)$. This trace compatibility requirement ensures that the pointwise nonlinear PEC response generates boundary data consistent with the range of the $\gamma_t$-trace operator, thereby providing the necessary regularity for the subsequent lifting argument.

Let $\mathbf{X}_0(\Omega)$ be the subspace of $\mathbf{X}(\Omega)$ defined by
\[
\mathbf{X}_0(\Omega) := \left\{ \mathbf{V} \in \mathbf{X}(\Omega) : \gamma_t \mathbf{V} = 0 \text{ on } \Gamma \right\}.
\]
We consider the sesquilinear form $a(\cdot, \cdot)$ associated with the linear Maxwell problem, given by
\[
\begin{aligned}
a(\mathbf{U}, \mathbf{V}) &:= \int_{\Omega} \left[ (\nabla \times \mathbf{U}) \cdot (\nabla \times \overline{\mathbf{V}}) - k^2 \mathbf{U} \cdot \overline{\mathbf{V}} \right] \mathrm{d}\boldsymbol{x} \\
&\quad + \mathrm{i}k \int_{\Gamma_R} (\Lambda \gamma_t \mathbf{U}) \cdot (\gamma_T \overline{\mathbf{V}}) \,\mathrm{d}s,
\end{aligned}
\]
where the boundary integral over $\Gamma_R$ is understood in the sense of the standard Maxwell trace duality. Based on the well-posedness theory for the truncated linear Maxwell problem with perfect electric conductor (PEC) boundary conditions, the operator
\[
A_0 : \mathbf{X}_0(\Omega) \to \mathbf{X}_0(\Omega)^*, \quad \langle A_0 \mathbf{U}, \mathbf{V} \rangle := a(\mathbf{U}, \mathbf{V}),
\]
is a bounded linear isomorphism. We denote the norm of its inverse by $C_0 := \|A_0^{-1}\|$.

\begin{theo}
\label{thm:nonlinear_pec}
Assume that the trace compatibility condition \eqref{eq:NPEC_lipschitz_trace} is satisfied. If
\[
(1+C_0 C_a) \|\mathcal{R}_\Gamma\| L_{\mathrm{pec}} < 1,
\]
then the nonlinear PEC problem admits a unique weak solution $\mathbf{E} \in \mathbf{X}(\Omega)$.
\end{theo}

\begin{proof}
For any $\mathbf{U} \in \mathbf{X}(\Omega)$, we define
\[
\mathbf{w}_{\mathbf{U}} := \mathcal{R}_\Gamma \mathbf{g}(\cdot, \gamma_T \mathbf{U}).
\]
This definition is justified by the fact that $\gamma_T \mathbf{U} \in L_t^2(\Gamma)$, which, combined with the trace compatibility assumption, ensures that $\mathbf{g}(\cdot, \gamma_T \mathbf{U}) \in Y_t^{\mathrm{reg}}(\Gamma)$. By the properties of the lifting operator, it follows that
\[
\gamma_t \mathbf{w}_{\mathbf{U}} = \mathbf{g}(\cdot, \gamma_T \mathbf{U}) \quad \text{on } \Gamma.
\]
We seek a solution in the decomposed form $\mathbf{E} = \mathbf{E}_0 + \mathbf{w}_{\mathbf{U}}$, where $\mathbf{E}_0 \in \mathbf{X}_0(\Omega)$. For a fixed $\mathbf{U}$, let $\mathbf{E}_0(\mathbf{U})$ be the unique solution in $\mathbf{X}_0(\Omega)$ to the variational problem
\begin{equation}\label{eq:NPEC_E0_weak}
a(\mathbf{E}_0, \mathbf{V}) = \langle F, \mathbf{V} \rangle - a(\mathbf{w}_{\mathbf{U}}, \mathbf{V}) \quad \forall \mathbf{V} \in \mathbf{X}_0(\Omega).
\end{equation}

The well-posedness of the linear PEC problem ensures that $\mathbf{E}_0(\mathbf{U})$ is uniquely determined. We then define the map
\[
\mathcal{T} : \mathbf{X}(\Omega) \to \mathbf{X}(\Omega), \quad \mathcal{T}(\mathbf{U}) := \mathbf{E}_0(\mathbf{U}) + \mathbf{w}_{\mathbf{U}}.
\]
To establish that $\mathcal{T}$ is a contraction, let $\mathbf{U}_1, \mathbf{U}_2 \in \mathbf{X}(\Omega)$ and set $\mathbf{w}_j := \mathbf{w}_{\mathbf{U}_j}$ and $\mathbf{E}_{0,j} := \mathbf{E}_0(\mathbf{U}_j)$ for $j=1,2$. From the boundedness of the lifting operator and the Lipschitz condition \eqref{eq:NPEC_lipschitz_trace}, we have
\[
\begin{aligned}
\|\mathbf{w}_1 - \mathbf{w}_2\|_{\mathbf{X}(\Omega)} &\le \|\mathcal{R}_\Gamma\| \left\| \mathbf{g}(\cdot, \gamma_T \mathbf{U}_1) - \mathbf{g}(\cdot, \gamma_T \mathbf{U}_2) \right\|_{Y_t^{\mathrm{reg}}(\Gamma)} \\
&\le \|\mathcal{R}_\Gamma\| L_{\mathrm{pec}} \|\gamma_T(\mathbf{U}_1 - \mathbf{U}_2)\|_{L^2(\Gamma)} \\
&\le \|\mathcal{R}_\Gamma\| L_{\mathrm{pec}} \|\mathbf{U}_1 - \mathbf{U}_2\|_{\mathbf{X}(\Omega)}.
\end{aligned}
\]
Subtracting the variations of \eqref{eq:NPEC_E0_weak} for $\mathbf{U}_1$ and $\mathbf{U}_2$ yields
\[
a(\mathbf{E}_{0,1} - \mathbf{E}_{0,2}, \mathbf{V}) = -a(\mathbf{w}_1 - \mathbf{w}_2, \mathbf{V}) \quad \forall \mathbf{V} \in \mathbf{X}_0(\Omega).
\]
Recall that $C_0 = \|A_0^{-1}\|$ and let $C_a$ be the continuity constant of $a(\cdot,\cdot)$, i.e.,
\[
|a(\mathbf{U},\mathbf{V})| \le C_a \|\mathbf{U}\|_{\mathbf{X}(\Omega)}\|\mathbf{V}\|_{\mathbf{X}(\Omega)}.
\]
Then, using the fact that $\mathbf{E}_{0,1} - \mathbf{E}_{0,2} = A_0^{-1}\bigl(-a(\mathbf{w}_1-\mathbf{w}_2,\cdot)\bigr)$, we obtain
\[
\|\mathbf{E}_{0,1} - \mathbf{E}_{0,2}\|_{\mathbf{X}(\Omega)} \le C_0 C_a \|\mathbf{w}_1 - \mathbf{w}_2\|_{\mathbf{X}(\Omega)}.
\]
It follows from the triangle inequality that
\[
\begin{aligned}
\|\mathcal{T}\mathbf{U}_1 - \mathcal{T}\mathbf{U}_2\|_{\mathbf{X}(\Omega)} &\le \|\mathbf{E}_{0,1} - \mathbf{E}_{0,2}\|_{\mathbf{X}(\Omega)} + \|\mathbf{w}_1 - \mathbf{w}_2\|_{\mathbf{X}(\Omega)} \\
&\le (1 + C_0 C_a) \|\mathbf{w}_1 - \mathbf{w}_2\|_{\mathbf{X}(\Omega)} \\
&\le (1 + C_0 C_a) \|\mathcal{R}_\Gamma\| L_{\mathrm{pec}} \|\mathbf{U}_1 - \mathbf{U}_2\|_{\mathbf{X}(\Omega)}.
\end{aligned}
\]
Under the smallness condition $(1 + C_0 C_a) \|\mathcal{R}_\Gamma\| L_{\mathrm{pec}} < 1$, the operator $\mathcal{T}$ is a contraction on the Hilbert space $\mathbf{X}(\Omega)$. The Banach fixed-point theorem then guarantees the existence of a unique fixed point $\mathbf{E} \in \mathbf{X}(\Omega)$.

By construction, this fixed point satisfies $\mathbf{E} = \mathbf{E}_0(\mathbf{E}) + \mathbf{w}_{\mathbf{E}}$. Since $\mathbf{E}_0(\mathbf{E}) \in \mathbf{X}_0(\Omega)$, we have $\gamma_t \mathbf{E} = \gamma_t \mathbf{w}_{\mathbf{E}} = \mathbf{g}(\cdot, \gamma_T \mathbf{E})$, which confirms that $\mathbf{E}$ satisfies the NPEC boundary condition. Moreover, \eqref{eq:NPEC_E0_weak} implies $a(\mathbf{E}, \mathbf{V}) = \langle \mathbf{F}, \mathbf{V} \rangle$ for all $\mathbf{V} \in \mathbf{X}_0(\Omega)$, so $\mathbf{E}$ is indeed a weak solution.

Conversely, suppose $\mathbf{E} \in \mathbf{X}(\Omega)$ is any weak solution to the NPEC problem. Defining $\mathbf{w}_{\mathbf{E}} = \mathcal{R}_\Gamma \mathbf{g}(\cdot, \gamma_T \mathbf{E})$, we see that $\mathbf{E} - \mathbf{w}_{\mathbf{E}}$ has a vanishing tangential trace on $\Gamma$ and thus belongs to $\mathbf{X}_0(\Omega)$. The weak formulation dictates that $\mathbf{E} - \mathbf{w}_{\mathbf{E}}$ must satisfy \eqref{eq:NPEC_E0_weak} with $\mathbf{U} = \mathbf{E}$. By the uniqueness of the solution to the linear problem, we must have $\mathbf{E} - \mathbf{w}_{\mathbf{E}} = \mathbf{E}_0(\mathbf{E})$, or equivalently, $\mathbf{E} = \mathcal{T}\mathbf{E}$. Thus, every weak solution is a fixed point of $\mathcal{T}$. The uniqueness of the fixed point then concludes the proof.
\end{proof}

\subsection{The nonlinear transmission condition (NTC)}

Lastly, we address the nonlinear transmission condition on the interior
interface \(\Gamma:=\partial D\). In this configuration, \(\Omega=B_R\), as
specified above. Here \(\mathbf X(\Omega)\) is understood as the trace-regular
subspace of \(\mathbf H(\operatorname{curl};\Omega)\) consisting of fields
whose common tangential trace on the interior interface \(\Gamma\) belongs to
\(L_t^2(\Gamma)\):
\[
\mathbf X(\Omega)
:=
\left\{
\mathbf E\in \mathbf H(\operatorname{curl};\Omega):
\gamma_T\mathbf E|_{\Gamma}\in L_t^2(\Gamma)
\right\}.
\]
The global membership \(\mathbf E\in \mathbf H(\operatorname{curl};\Omega)\)
incorporates the tangential continuity condition
\[
[\mathbf n\times \mathbf E]=0
\quad\text{on }\Gamma
\]
in the trace sense. Therefore, the common tangential component on \(\Gamma\)
is well defined and will be denoted by \(\gamma_T\mathbf E\).

The weak formulation of the nonlinear transmission problem consists in finding $\mathbf{E} \in \mathbf{X}(\Omega)$ such that
\[
\mathcal{A}^{\mathrm{tr}}(\mathbf{E}; \mathbf{V}) = \langle \mathbf{F}^{\mathrm{tr}}, \mathbf{V} \rangle \qquad \forall \mathbf{V} \in \mathbf{X}(\Omega),
\]
where the nonlinear form $\mathcal{A}^{\mathrm{tr}}(\cdot; \cdot)$ is defined by
\begin{equation}\label{eq:alternative_form}
\begin{aligned}
\mathcal{A}^{\mathrm{tr}}(\mathbf{E}; \mathbf{V}) &:= \int_{\Omega} \left[ \mu_r^{-1}(\nabla \times \mathbf{E}) \cdot (\nabla \times \overline{\mathbf{V}}) - k^2 \varepsilon_r \mathbf{E} \cdot \overline{\mathbf{V}} \right] \mathrm{d}\boldsymbol{x} \\
&\quad - \int_{\Gamma} \mathbf{g}(\cdot, \gamma_T \mathbf{E}) \cdot \gamma_T \overline{\mathbf{V}} \,\mathrm{d}s + \mathrm{i}k \langle \Lambda \gamma_t \mathbf{E}, \gamma_T \overline{\mathbf{V}} \rangle_{\Gamma_R}.
\end{aligned}
\end{equation}
The functional $\mathbf{F}^{\mathrm{tr}} \in \mathbf{X}(\Omega)^*$ is given by the incident field $\mathbf{E}^i$ as
\[
\langle \mathbf{F}^{\mathrm{tr}}, \mathbf{V} \rangle := \left\langle \mathrm{i}k \Lambda (\gamma_t \mathbf{E}^i) - \gamma_t (\nabla \times \mathbf{E}^i), \gamma_T \overline{\mathbf{V}} \right\rangle_{\Gamma_R}.
\]

Let \(\mathscr A_l^{\rm tr}:\mathbf X(\Omega)\to\mathbf X(\Omega)^*\) denote the
linear transmission operator obtained from \(\mathcal A\) by omitting the
nonlinear interface term:
\[
\begin{aligned}
\langle \mathscr  A_l^{\rm tr}\mathbf E,\mathbf V\rangle
&:=
\int_\Omega
\left[
\mu_r^{-1}(\nabla\times\mathbf E)
\cdot(\nabla\times\overline{\mathbf V})
-
k^2\varepsilon_r\mathbf E\cdot\overline{\mathbf V}
\right]\,\mathrm d\boldsymbol{x}
\\
&\quad
+
\mathrm i k
\left\langle
\Lambda\gamma_t\mathbf E,
\gamma_T\overline{\mathbf V}
\right\rangle_{\Gamma_R}.
\end{aligned}
\]
We use the standard weak well-posedness result for the corresponding linear
Maxwell transmission problem in \(\mathbf X(\Omega)\). In particular, if
\(\mathbf E_1,\mathbf E_2\in\mathbf X(\Omega)\) are the linear transmission
solutions corresponding to interface data
\(\mathbf f_1,\mathbf f_2\in L_t^2(\Gamma)\), respectively, then there exists
\(C_{\rm tr}>0\) such that
\begin{equation}\label{eq:linear_transmission_stability}
\|\mathbf E_1-\mathbf E_2\|_{\mathbf X(\Omega)}
\le
C_{\rm tr}
\|\mathbf f_1-\mathbf f_2\|_{L^2(\Gamma)} .
\end{equation}

\begin{coro}
\label{th:existence_ntc}
Assume that the nonlinear term $\mathbf{g}$ satisfies the $L^2$-Lipschitz condition \eqref{eq:g_lipschitz_L2}. If the smallness condition
\[
C_{\mathrm{tr}} L_g < 1
\]
holds, then the nonlinear transmission problem admits a unique weak solution $\mathbf{E} \in \mathbf{X}(\Omega)$.
\end{coro}

\begin{proof}
For a fixed $\mathbf{U} \in \mathbf{X}(\Omega)$, let $\mathcal{T}\mathbf{U} \in \mathbf{X}(\Omega)$ be the unique solution of the linear transmission problem
\[
\langle \mathscr {A}_l^{\mathrm{tr}} \mathcal{T}\mathbf{U}, \mathbf{V} \rangle = \langle \mathbf{F}^{\mathrm{tr}}, \mathbf{V} \rangle + \int_{\Gamma} \mathbf{g}(\cdot, \gamma_T \mathbf{U}) \cdot \gamma_T \overline{\mathbf{V}} \,\mathrm{d}s \qquad \forall \mathbf{V} \in \mathbf{X}(\Omega),
\]
where $\mathscr {A}_l^{\mathrm{tr}}$ denotes the linear part of the transmission operator. The map $\mathcal{T}$ is well-defined since $\gamma_T \mathbf{U} \in L_t^2(\Gamma)$ implies $\mathbf{g}(\cdot, \gamma_T \mathbf{U}) \in L_t^2(\Gamma)$ by \eqref{eq:g_lipschitz_L2}.

Let $\mathbf{U}_1, \mathbf{U}_2 \in \mathbf{X}(\Omega)$ and define $\mathbf{E}_j := \mathcal{T}\mathbf{U}_j$ for $j=1,2$. Invoking the stability estimate for the linear transmission problem \eqref{eq:linear_transmission_stability}, we have
\[
\|\mathbf{E}_1 - \mathbf{E}_2\|_{\mathbf{X}(\Omega)} \le C_{\mathrm{tr}} \left\| \mathbf{g}(\cdot, \gamma_T \mathbf{U}_1) - \mathbf{g}(\cdot, \gamma_T \mathbf{U}_2) \right\|_{L^2(\Gamma)}.
\]
Applying the Lipschitz condition \eqref{eq:g_lipschitz_L2}, we obtain
\[
\|\mathcal{T}\mathbf{U}_1 - \mathcal{T}\mathbf{U}_2\|_{\mathbf{X}(\Omega)} \le C_{\mathrm{tr}} L_g \|\mathbf{U}_1 - \mathbf{U}_2\|_{\mathbf{X}(\Omega)}.
\]
Thus, $\mathcal{T}$ is a contraction on the Hilbert space $\mathbf{X}(\Omega)$ provided that $C_{\mathrm{tr}} L_g < 1$. By the Banach fixed-point theorem, there exists a unique fixed point $\mathbf{E} \in \mathbf{X}(\Omega)$. By construction, this fixed point is the unique weak solution to the nonlinear transmission problem.
\end{proof}

\section{Domain dependence and material derivatives}
\label{sec:3}

In this section, we analyze the sensitivity of the electric field with respect to shape variations of the boundary $\Gamma = \partial D$. Let the obstacle $D \subset \mathbb{R}^3$ be a bounded domain of class $C^2$. We consider a perturbation field $\mathbf{h} \in C_c^1(B_R; \mathbb{R}^3)$ whose support is strictly contained in a neighborhood of $\Gamma$, thereby ensuring that the artificial boundary $\Gamma_R$ remains invariant. For $\|\mathbf{h}\|_{C^1}$ sufficiently small, the mapping
\[
\varphi_{\mathbf{h}}(\boldsymbol{x}) := \boldsymbol{x} + \mathbf{h}(\boldsymbol{x})
\]
is a $C^1$-diffeomorphism of $\mathbb{R}^3$. The resulting perturbed domains and boundaries are defined by
\[
D_{\mathbf{h}} := \varphi_{\mathbf{h}}(D), \qquad \Omega_{\mathbf{h}} := B_R \setminus \overline{D_{\mathbf{h}}}, \qquad \Gamma_{\mathbf{h}} := \varphi_{\mathbf{h}}(\Gamma).
\]

\subsection{The nonlinear impedance boundary condition}

Let $\mathbf{E}_{\mathbf{h}}$ denote the total electric field associated with the perturbed obstacle $D_{\mathbf{h}}$. Within the analytical framework established in Section~\ref{sec:2}, we assume that $\mathbf{E}_{\mathbf{h}}$ is the unique weak solution to the perturbed nonlinear impedance problem in the trace-regular space $\mathbf{X}(\Omega_{\mathbf{h}})$. This field is characterized by the variational equation
\begin{equation}\label{eq:perturbed_variational}
\mathcal{A}_{\mathbf{h}}(\mathbf{E}_{\mathbf{h}}, \mathbf{V}_{\mathbf{h}}) = \langle F, \mathbf{V}_{\mathbf{h}} \rangle \qquad \forall \mathbf{V}_{\mathbf{h}} \in \mathbf{X}(\Omega_{\mathbf{h}}),
\end{equation}
where the nonlinear form $\mathcal{A}_{\mathbf{h}}(\cdot, \cdot)$ is defined by
\begin{equation}\label{eq:perturbed_var_form}
\begin{aligned}
\mathcal{A}_{\mathbf{h}}(\mathbf{E}_{\mathbf{h}}, \mathbf{V}_{\mathbf{h}}) &:= \int_{\Omega_{\mathbf{h}}} \left[ (\nabla \times \mathbf{E}_{\mathbf{h}}) \cdot (\nabla \times \overline{\mathbf{V}}_{\mathbf{h}}) - k^2 \mathbf{E}_{\mathbf{h}} \cdot \overline{\mathbf{V}}_{\mathbf{h}} \right] \mathrm{d}\boldsymbol{y} \\
&\quad + \mathrm{i}k \int_{\Gamma_{\mathbf{h}}} \lambda (\gamma_T^{\mathbf{h}} \mathbf{E}_{\mathbf{h}}) \cdot (\gamma_T^{\mathbf{h}} \overline{\mathbf{V}}_{\mathbf{h}}) \,\mathrm{d}s_{\boldsymbol{y}} \\
&\quad - \int_{\Gamma_{\mathbf{h}}} \mathbf{g}(\boldsymbol{y}, \gamma_T^{\mathbf{h}} \mathbf{E}_{\mathbf{h}}) \cdot \gamma_T^{\mathbf{h}} \overline{\mathbf{V}}_{\mathbf{h}} \,\mathrm{d}s_{\boldsymbol{y}} \\
&\quad + \mathrm{i}k \langle \Lambda \gamma_t \mathbf{E}_{\mathbf{h}}, \gamma_T \overline{\mathbf{V}}_{\mathbf{h}} \rangle_{\Gamma_R}.
\end{aligned}
\end{equation}
Here $\gamma_T^{\mathbf{h}}$ denotes the tangential trace on the perturbed boundary $\Gamma_{\mathbf{h}}$.
Since the perturbation $\mathbf{h}$ is supported away from $\Gamma_R$, the artificial boundary and the source functional remain invariant under the deformation. Consequently, the right-hand side of \eqref{eq:perturbed_variational} is given by
\begin{equation}\label{eq:F_perturbed}
\langle F, \mathbf{V}_{\mathbf{h}} \rangle = \left\langle \mathrm{i}k \Lambda (\gamma_t \mathbf{E}^i) - \gamma_t (\nabla \times \mathbf{E}^i), \gamma_T \overline{\mathbf{V}}_{\mathbf{h}} \right\rangle_{\Gamma_R}.
\end{equation}

To rigorously define the spatial dependence of the nonlinear boundary term on $\Gamma_{\mathbf{h}}$, we assume that $\mathbf{g}$ is defined on a tubular neighborhood $U_\Gamma$ of the reference interface $\Gamma$, taking the form
\[
\mathbf{g} : U_\Gamma \times \mathbb{C}^3 \to \mathbb{C}^3.
\]
Consequently, for any point $\boldsymbol{y} = \varphi_{\mathbf{h}}(\boldsymbol{x}) \in \Gamma_{\mathbf{h}}$, the composition $\mathbf{g}(\boldsymbol{y}, \gamma_T^{\mathbf{h}}\mathbf{E}_{\mathbf{h}}(\boldsymbol{y}))$ is well-defined provided $\|\mathbf{h}\|_{C^1}$ is sufficiently small.

We now pull the perturbed problem back to the fixed reference domain $\Omega$. Let $\boldsymbol{y} = \varphi_{\mathbf{h}}(\boldsymbol{x})$ and denote the Jacobian matrix of the transformation by $J_{\varphi} := \nabla \varphi_{\mathbf{h}} = \mathrm{Id} + \nabla \mathbf{h}$. It is well known that the standard composition $\mathbf{E}_{\mathbf{h}} \circ \varphi_{\mathbf{h}}$ does not preserve the $H(\mathrm{curl})$-structure of the fields. We therefore utilize the covariant Piola transform, defined by
\begin{equation}\label{eq:piola_transform}
\hat{\mathbf{E}}_{\mathbf{h}} := J_{\varphi}^{\top} (\mathbf{E}_{\mathbf{h}} \circ \varphi_{\mathbf{h}}).
\end{equation}
This transformation ensures that $\hat{\mathbf{E}}_{\mathbf{h}} \in H(\mathrm{curl}; \Omega)$ if and only if $\mathbf{E}_{\mathbf{h}} \in H(\mathrm{curl}; \Omega_{\mathbf{h}})$.

To characterize the boundary terms, we define the unit normal to the perturbed boundary and the associated tangential projection operator as
\[
\widetilde{\mathbf{n}}_{\mathbf{h}} := \frac{J_{\varphi}^{-\top} \mathbf{n}}{\|J_{\varphi}^{-\top} \mathbf{n}\|}, \qquad P_{\mathbf{h}} := \mathrm{Id} - \widetilde{\mathbf{n}}_{\mathbf{h}} \otimes \widetilde{\mathbf{n}}_{\mathbf{h}}.
\]
For any field $\mathbf{U} \in \mathbf{X}(\Omega)$, we introduce the transported tangential trace 
\begin{equation}\label{eq:transported_trace_definition}
\gamma_{T,\mathbf{h}} \mathbf{U} := P_{\mathbf{h}} J_{\varphi}^{-\top} \gamma_T \mathbf{U},
\end{equation}
where $\gamma_T$ is the tangential trace on the reference boundary $\Gamma$. For sufficiently smooth fields, \eqref{eq:transported_trace_definition} admits the equivalent representation
\begin{equation}\label{eq:E_tangential}
\gamma_{T,\mathbf{h}} \mathbf{U} = \widetilde{\mathbf{n}}_{\mathbf{h}} \times \bigl( J_{\varphi}^{-\top} \mathbf{U} \times \widetilde{\mathbf{n}}_{\mathbf{h}} \bigr).
\end{equation}
A crucial property of the Piola transform is that the pullback of the tangential trace on $\Gamma_{\mathbf{h}}$ coincides with the transported trace of the pullback field:
\begin{equation}\label{eq:transported_tangential_trace}
(\gamma_T^{\mathbf{h}} \mathbf{E}_{\mathbf{h}}) \circ \varphi_{\mathbf{h}} = \gamma_{T,\mathbf{h}} \hat{\mathbf{E}}_{\mathbf{h}}.
\end{equation}
Since the operator $P_{\mathbf{h}} J_{\varphi}^{-\top}$ is uniformly bounded for sufficiently small $\|\mathbf{h}\|_{C^1}$, we have the estimate
\[
\|\gamma_{T,\mathbf{h}} \mathbf{U}\|_{L^2(\Gamma)} \le C \|\gamma_T \mathbf{U}\|_{L^2(\Gamma)} \le C \|\mathbf{U}\|_{\mathbf{X}(\Omega)}.
\]
It follows that $\mathbf{E}_{\mathbf{h}} \in \mathbf{X}(\Omega_{\mathbf{h}})$ implies $\hat{\mathbf{E}}_{\mathbf{h}} \in \mathbf{X}(\Omega)$, establishing the required regularity on the fixed domain.

The covariant Piola transform preserves the \(H(\operatorname{curl})\)-structure
of the fields and is therefore the natural pullback for Maxwell equations
\cite[Lemma~3.58]{Monk2003}. Under this transform, the curl operator satisfies
\[
(\nabla_{\boldsymbol y}\times \mathbf E_{\mathbf h})\circ\varphi_{\mathbf h}
=
\frac{1}{\det J_{\varphi}}
J_{\varphi}
(\nabla_{\boldsymbol x}\times \widehat{\mathbf E}_{\mathbf h}).
\]
By pulling back the variational problem \eqref{eq:perturbed_var_form} to the fixed reference domain $\Omega$, we obtain the transformed formulation: find $\hat{\mathbf{E}}_{\mathbf{h}} \in \mathbf{X}(\Omega)$ such that
\begin{equation}\label{eq:pullback_problem}
\widetilde{\mathcal{A}}_{\mathbf{h}}(\hat{\mathbf{E}}_{\mathbf{h}}, \mathbf{V}) = \langle F, \mathbf{V} \rangle \qquad \forall \mathbf{V} \in \mathbf{X}(\Omega),
\end{equation}
where the transformed nonlinear form $\widetilde{\mathcal{A}}_{\mathbf{h}}$ is given by
\begin{equation}\label{eq:weak_form_pullback_impedance}
\begin{aligned}
\widetilde{\mathcal{A}}_{\mathbf{h}}(\hat{\mathbf{E}}_{\mathbf{h}}, \mathbf{V}) &:= \int_{\Omega} (\nabla \times \hat{\mathbf{E}}_{\mathbf{h}})^{\top} \mathcal{M}_{\mathbf{h}} (\nabla \times \overline{\mathbf{V}}) \,\mathrm{d}\boldsymbol{x} - k^2 \int_{\Omega} \hat{\mathbf{E}}_{\mathbf{h}}^{\top} \mathcal{N}_{\mathbf{h}} \overline{\mathbf{V}} \,\mathrm{d}\boldsymbol{x} \\
&\quad + \mathrm{i}k \int_{\Gamma} \lambda (\gamma_{T,\mathbf{h}} \hat{\mathbf{E}}_{\mathbf{h}}) \cdot (\gamma_{T,\mathbf{h}} \overline{\mathbf{V}}) \, \omega_{\mathbf{h}} \,\mathrm{d}s \\
&\quad - \int_{\Gamma} \mathbf{g} \bigl( \varphi_{\mathbf{h}}(\cdot), \gamma_{T,\mathbf{h}} \hat{\mathbf{E}}_{\mathbf{h}} \bigr) \cdot (\gamma_{T,\mathbf{h}} \overline{\mathbf{V}}) \, \omega_{\mathbf{h}} \,\mathrm{d}s \\
&\quad + \mathrm{i}k \langle \Lambda \gamma_t \hat{\mathbf{E}}_{\mathbf{h}}, \gamma_T \overline{\mathbf{V}} \rangle_{\Gamma_R}.
\end{aligned}
\end{equation}
The geometric coefficients in the bulk are defined as
\[
\mathcal{M}_{\mathbf{h}} := \frac{1}{\det J_{\varphi}} J_{\varphi}^{\top} J_{\varphi}, \qquad \mathcal{N}_{\mathbf{h}} := \det J_{\varphi} J_{\varphi}^{-1} J_{\varphi}^{-\top},
\]
and the surface Jacobian is $\omega_{\mathbf{h}} := \det J_{\varphi} \|J_{\varphi}^{-\top} \mathbf{n}\|$. Equivalently, if $\phi$ is a local parametrization of $\Gamma$ and we set $\widehat{\phi} = \varphi_{\mathbf{h}} \circ \phi$, then the surface Jacobian is given by the ratio of the square roots of the metric determinants:
\[
\omega_{\mathbf{h}} = \frac{\sqrt{\det(J_{\widehat{\phi}}^{\top} J_{\widehat{\phi}})}}{\sqrt{\det(J_{\phi}^{\top} J_{\phi})}}.
\]

We now record the first-order expansions of the geometric quantities used
below. Under the \(C^2\)-regularity assumption on \(\Gamma\), the pulled-back
unit normal
\[
\widetilde{\mathbf{n}}_{\mathbf{h}} = \mathbf{n} + \delta \mathbf{n} + o(\|\mathbf{h}\|_{C^1}),
\]
where the shape derivative $\delta \mathbf{n}$, as derived in \cite[Chap.~9, Eq.~(4.38)]{Delfour2011}, is given by
\begin{equation}\label{eq:normal_shape_derivative}
\delta \mathbf{n} = -\nabla_{\Gamma} h_n + J_{\mathbf{n}} \mathbf{h}_T.
\end{equation}
Here, the perturbation field is decomposed into its tangential and normal components on $\Gamma$ as $\mathbf{h} = \mathbf{h}_T + h_n \mathbf{n}$, with $\mathbf{h}_T := \mathbf{n} \times (\mathbf{h} \times \mathbf{n})$ and $h_n := \mathbf{h} \cdot \mathbf{n}$. Similarly, the surface Jacobian expands as
\begin{equation}\label{eq:surface_jacobian_expansion}
\omega_{\mathbf{h}} = 1 + \operatorname{div}_{\Gamma} \mathbf{h}_T + 2\kappa h_n + o(\|\mathbf{h}\|_{C^1}),
\end{equation}
where $2\kappa = \operatorname{div}_{\Gamma} \mathbf{n}$ denotes the additive  curvature of $\Gamma$ consistent with our choice of unit normal and $\kappa$ is the usual mean curvature.

For the reader's convenience, we briefly summarize the definitions of standard surface differential operators following \cite[Section~9.5]{Delfour2011}. Let $\Gamma$ be a sufficiently smooth surface with unit outward normal $\mathbf{n}$. Given a scalar function $\psi$ and a tangential vector field $\mathbf{v}$ defined on $\Gamma$, let $\widetilde{\psi}$ and $\widetilde{\mathbf{v}}$ denote their respective smooth extensions to a tubular neighborhood of $\Gamma$. The surface gradient, surface divergence, vector curl, and scalar curl are defined by
\begin{align*}
\nabla_\Gamma\psi &:= \nabla\widetilde\psi - (\mathbf{n}\cdot\nabla\widetilde\psi)\mathbf{n}, \\
\operatorname{div}_\Gamma\mathbf{v} &:= \nabla\cdot\widetilde{\mathbf{v}} - \mathbf{n}^\top(\nabla\widetilde{\mathbf{v}})\mathbf{n}, \\
\operatorname{Curl}_\Gamma\psi &:= \mathbf{n}\times\nabla_\Gamma\psi, \\
\operatorname{curl}_\Gamma\mathbf{v} &:= (\nabla\times\widetilde{\mathbf{v}})\cdot\mathbf{n}.
\end{align*}
These intrinsic operators depend only on the boundary data $\psi$ and $\mathbf{v}$, and are invariant with respect to the choice of the extensions $\widetilde{\psi}$ and $\widetilde{\mathbf{v}}$.

The geometric quantities $\gamma_{T,\mathbf{h}}$, $\mathcal{M}_{\mathbf{h}}$, $\mathcal{N}_{\mathbf{h}}$, and $\omega_{\mathbf{h}}$ introduced in the preceding pullback analysis are universal in the sense that they appear in the NPEC and transmission problems discussed hereafter. We provide their first-order linearizations in the following lemma.

\begin{lemm}\label{le:positive}
Let $\varphi_{\mathbf{h}} = \mathrm{Id} + \mathbf{h}$ with $\mathbf{h} \in C_c^1(B_R; \mathbb{R}^3)$. As $\|\mathbf{h}\|_{C^1} \to 0$, the following asymptotic expansions hold for the bulk geometric coefficients:
\begin{align*}
\mathcal{M}_{\mathbf{h}} &= \mathrm{Id} - (\operatorname{div}\mathbf{h})\mathrm{Id} + J_{\mathbf{h}} + J_{\mathbf{h}}^{\top} + o(\|\mathbf{h}\|_{C^1}), \\
\mathcal{N}_{\mathbf{h}} &= \mathrm{Id} + (\operatorname{div}\mathbf{h})\mathrm{Id} - J_{\mathbf{h}} - J_{\mathbf{h}}^{\top} + o(\|\mathbf{h}\|_{C^1}),
\end{align*}
where $J_{\mathbf{h}} := \nabla\mathbf{h}$ is the displacement gradient. Furthermore, if $\Gamma$ is sufficiently smooth, the surface Jacobian admits the expansion
\[
\omega_{\mathbf{h}} = 1 + \operatorname{div}_\Gamma\mathbf{h}_T + 2\kappa h_n + o(\|\mathbf{h}\|_{C^1}),
\]
where $\mathbf{h}_T$ and $h_n$ denote the tangential and normal components of $\mathbf{h}$ on $\Gamma$, respectively.
\end{lemm}

\begin{proof}
The expansions for $\mathcal{M}_{\mathbf{h}}$ and $\mathcal{N}_{\mathbf{h}}$ follow from a straightforward application of the linearizations
\[
J_{\varphi} = \mathrm{Id} + J_{\mathbf{h}} + o(\|\mathbf{h}\|_{C^1}), \qquad J_{\varphi}^{-1} = \mathrm{Id} - J_{\mathbf{h}} + o(\|\mathbf{h}\|_{C^1}),
\]
and
\[
\det J_{\varphi} = 1 + \operatorname{div} \mathbf{h} + o(\|\mathbf{h}\|_{C^1}), \qquad (\det J_{\varphi})^{-1} = 1 - \operatorname{div} \mathbf{h} + o(\|\mathbf{h}\|_{C^1}).
\]
Substituting these identities into the definitions of $\mathcal{M}_{\mathbf{h}}$ and $\mathcal{N}_{\mathbf{h}}$ and discarding terms of order $o(\|\mathbf{h}\|_{C^1})$ yields the stated results; a detailed derivation of these bulk expansions can be found in \cite{Hettlich2012}.

The expansion for the surface Jacobian follows from the definition $\omega_{\mathbf{h}} = \det J_{\varphi} \|J_{\varphi}^{-\top} \mathbf{n}\|$ combined with the classical first variation formula for the surface measure. We refer the reader to \cite[Chapter~9]{Delfour2011} for further details.
\end{proof}
We henceforth assume that the nonlinear boundary response $\mathbf{g}$ satisfies the following conditions, which will be maintained for both the NPEC and nonlinear transmission problems analyzed in the sequel.

\begin{assu}\label{ass:enhanced_g}
Let \(U_\Gamma\) be a tubular neighborhood of \(\Gamma\). The nonlinear
response
\[
\mathbf g:U_\Gamma\times\mathbb C^3\to\mathbb C^3
\]
is assumed to be a Carath\'eodory function satisfying the following
conditions.

\begin{enumerate}
\item[\textup{(A1)}]
For each \(\boldsymbol z\in\mathbb C^3\), the map
$\boldsymbol x\mapsto \mathbf g(\boldsymbol x,\boldsymbol z)$
is continuously differentiable in \(U_\Gamma\). Moreover, there exist
constants \(c_0,c_1>0\) and nonnegative functions
\(\psi_0,\psi_1\) defined in \(U_\Gamma\), whose restrictions to
\(\Gamma\) and to the perturbed boundaries \(\Gamma_{\mathbf h}\) are
uniformly bounded in \(L^2\) for all sufficiently small
\(\|\mathbf h\|_{C^1}\), such that
\[
|\mathbf g(\boldsymbol x,\boldsymbol z)|
\le
\psi_0(\boldsymbol x)+c_0|\boldsymbol z|,\qquad
|\nabla_{\boldsymbol x}\mathbf g(\boldsymbol x,\boldsymbol z)|
\le
\psi_1(\boldsymbol x)+c_1|\boldsymbol z|
\]
for a.e. \(\boldsymbol x\) on \(\Gamma\) and on \(\Gamma_{\mathbf h}\), and
for all \(\boldsymbol z\in\mathbb C^3\).

\item[\textup{(A2)}]
The first variations
$\mathbf g_{\boldsymbol z}\text,\quad
\mathbf g_{\boldsymbol x\boldsymbol z}$
with respect to the second variable exist in the real sense. More precisely,
for a.e. \(\boldsymbol x\) on \(\Gamma\) and on the perturbed boundaries
\(\Gamma_{\mathbf h}\), the maps
$
\boldsymbol w\mapsto
\mathbf g_{\boldsymbol z}(\boldsymbol x,\boldsymbol z;\boldsymbol w),
\quad
\boldsymbol w\mapsto
\mathbf g_{\boldsymbol x\boldsymbol z}
(\boldsymbol x,\boldsymbol z;\boldsymbol w)
$
are \(\mathbb R\)-linear in \(\boldsymbol w\in\mathbb C^3\), and satisfy
\[
\mathbf g(\boldsymbol x,\boldsymbol z+\boldsymbol w)
-
\mathbf g(\boldsymbol x,\boldsymbol z)
=
\mathbf g_{\boldsymbol z}
(\boldsymbol x,\boldsymbol z;\boldsymbol w)
+
o(|\boldsymbol w|),
\]
and
\[
\nabla_{\boldsymbol x}\mathbf g
(\boldsymbol x,\boldsymbol z+\boldsymbol w)
-
\nabla_{\boldsymbol x}\mathbf g
(\boldsymbol x,\boldsymbol z)
=
\mathbf g_{\boldsymbol x\boldsymbol z}
(\boldsymbol x,\boldsymbol z;\boldsymbol w)
+
o(|\boldsymbol w|)
\]
as \(|\boldsymbol w|\to0\). In addition, there exists a constant
\(C_g>0\), independent of \(\boldsymbol x,\boldsymbol z,\boldsymbol w\) and
of the admissible perturbation \(\mathbf h\), such that
\[
|\mathbf g_{\boldsymbol z}
(\boldsymbol x,\boldsymbol z;\boldsymbol w)|
+
|\mathbf g_{\boldsymbol x\boldsymbol z}
(\boldsymbol x,\boldsymbol z;\boldsymbol w)|
\le
C_g|\boldsymbol w|.
\]

\item[\textup{(A3)}]
We further assume that the associated Nemytskii operators admit first-order
expansions in \(L_t^2(\Gamma)\). More precisely, for any
\(\boldsymbol\zeta\in L_t^2(\Gamma)\) and for any
\(\boldsymbol\eta\in L_t^2(\Gamma)\) with
\(\|\boldsymbol\eta\|_{L^2(\Gamma)}\to0\), one has
\[
\left\|\mathbf g(\cdot,\boldsymbol\zeta+\boldsymbol\eta)-
\mathbf g(\cdot,\boldsymbol\zeta)-\mathbf g_{\boldsymbol z}
(\cdot,\boldsymbol\zeta;\boldsymbol\eta)
\right\|_{L^2(\Gamma)}
=o(\|\boldsymbol\eta\|_{L^2(\Gamma)}),
\]
and
\[
\left\|\nabla_{\boldsymbol x}\mathbf g
(\cdot,\boldsymbol\zeta+\boldsymbol\eta)-
\nabla_{\boldsymbol x}\mathbf g(\cdot,\boldsymbol\zeta)
-\mathbf g_{\boldsymbol x\boldsymbol z}
(\cdot,\boldsymbol\zeta;\boldsymbol\eta)
\right\|_{L^2(\Gamma)}
=o(\|\boldsymbol\eta\|_{L^2(\Gamma)}).
\]
The remainders are uniform for
\(\boldsymbol\zeta\) in bounded subsets of \(L_t^2(\Gamma)\).
\item[\textup{(A4)}]
The estimates and expansions above are stable under admissible domain
perturbations. More precisely, after pulling back the quantities defined on
\(\Gamma_{\mathbf h}\) to the reference boundary \(\Gamma\), the same
\(L^2\)-bounds and Nemytskii remainders hold with constants independent of
\(\mathbf h\), for all sufficiently small \(\|\mathbf h\|_{C^1}\).
\end{enumerate}
\end{assu}
Equipped with Lemma~\ref{le:positive} and Assumption~\ref{ass:enhanced_g}, we are now prepared to establish the continuity of the pulled-back electric field with respect to the shape perturbation $\mathbf{h}$.

\begin{theo}\label{th:continuity_domain_dependence}
Suppose that $\mathbf{g}$ satisfies Assumption~\ref{ass:enhanced_g} and the $L^2$-Lipschitz condition \eqref{eq:g_lipschitz_L2}. Let $\mathbf{E} \in \mathbf{X}(\Omega)$ and $\hat{\mathbf{E}}_{\mathbf{h}} \in \mathbf{X}(\Omega)$ denote the unique weak solutions to the scattering problems \eqref{eq:main_formulation} and \eqref{eq:pullback_problem}, respectively. Further, assume that the boundary response is uniformly $L^2$-Lipschitz in the trace variable: there exists a constant $L_g > 0$ such that
\[
\|\mathbf{g}(\cdot, \boldsymbol{\zeta}_1) - \mathbf{g}(\cdot, \boldsymbol{\zeta}_2)\|_{L^2(\Gamma)} \le L_g \|\boldsymbol{\zeta}_1 - \boldsymbol{\zeta}_2\|_{L^2(\Gamma)}
\]
for all $\boldsymbol{\zeta}_1, \boldsymbol{\zeta}_2$ in a bounded subset of $L_t^2(\Gamma)$. If $L_g$ is sufficiently small, then the solution $\hat{\mathbf{E}}_{\mathbf{h}}$ depends continuously on the deformation in the sense that
\begin{equation}
\label{eq:continuity_limit}
\lim_{\|\mathbf{h}\|_{C^1} \to 0} \|\hat{\mathbf{E}}_{\mathbf{h}} - \mathbf{E}\|_{\mathbf{X}(\Omega)} = 0.
\end{equation}
\end{theo}

\begin{proof}
We begin by decomposing the variational forms $\mathcal{A}$ and $\widetilde{\mathcal{A}}_{\mathbf{h}}$ into their respective linear and nonlinear components:
\[
\mathcal{A}(\mathbf{E}, \mathbf{V}) = \mathcal{A}_l(\mathbf{E}, \mathbf{V}) + \mathcal{A}_n(\mathbf{E}, \mathbf{V}), \qquad \widetilde{\mathcal{A}}_{\mathbf{h}}(\mathbf{E}, \mathbf{V}) = \widetilde{\mathcal{A}}_{\mathbf{h},l}(\mathbf{E}, \mathbf{V}) + \widetilde{\mathcal{A}}_{\mathbf{h},n}(\mathbf{E}, \mathbf{V}).
\]
The linear component $\mathcal{A}_l$ is defined as
\begin{align*}
\mathcal{A}_l(\mathbf{E}, \mathbf{V}) &:= \int_\Omega \Big[ (\nabla \times \mathbf{E}) \cdot (\nabla \times \overline{\mathbf{V}}) - k^2 \mathbf{E} \cdot \overline{\mathbf{V}} \Big] \mathrm{d}\boldsymbol{x} \\
&\quad + \mathrm{i} k \lambda \int_\Gamma (\gamma_T \mathbf{E}) \cdot (\gamma_T \overline{\mathbf{V}}) \,\mathrm{d}s + \mathrm{i} k \langle \Lambda \gamma_t \mathbf{E}, \gamma_T \overline{\mathbf{V}} \rangle_{\Gamma_R},
\end{align*}
while the nonlinear component is given by
\[
\mathcal{A}_n(\mathbf{E}, \mathbf{V}) := - \int_\Gamma \mathbf{g}(\cdot, \gamma_T \mathbf{E}) \cdot \gamma_T \overline{\mathbf{V}} \,\mathrm{d}s.
\]
The pulled-back forms $\widetilde{\mathcal{A}}_{\mathbf{h},l}$ and $\widetilde{\mathcal{A}}_{\mathbf{h},n}$ are defined analogously via the terms in \eqref{eq:weak_form_pullback_impedance}.

Invoking the Riesz representation theorem on the Hilbert space $\mathbf{X}(\Omega)$, we introduce the bounded linear operators $T_l, T_{\mathbf{h},l} \in \mathcal{L}(\mathbf{X}(\Omega))$ and the nonlinear operators $T_n, T_{\mathbf{h},n} : \mathbf{X}(\Omega) \to \mathbf{X}(\Omega)$ satisfying
\[
(T_l \mathbf{w}, \mathbf{V})_{\mathbf{X}(\Omega)} = \mathcal{A}_l(\mathbf{w}, \mathbf{V}), \qquad (T_{\mathbf{h},l} \mathbf{w}, \mathbf{V})_{\mathbf{X}(\Omega)} = \widetilde{\mathcal{A}}_{\mathbf{h},l}(\mathbf{w}, \mathbf{V}),
\]
with analogous identities for $T_n$ and $T_{\mathbf{h},n}$. Finally, we define the total operators
\[
T := T_l + T_n \quad \text{and} \quad T_{\mathbf{h}} := T_{\mathbf{h},l} + T_{\mathbf{h},n}.
\]

We first establish a perturbation estimate for the operators under a fixed field $\mathbf{w} \in \mathbf{X}(\Omega)$. Let $\mathbf{Z}_{\mathbf{h}} := T(\mathbf{w}) - T_{\mathbf{h}}(\mathbf{w})$. By the properties of the Riesz representatives, we have
\[
\|\mathbf{Z}_{\mathbf{h}}\|_{\mathbf{X}(\Omega)}^2 = \mathcal{A}(\mathbf{w}, \mathbf{Z}_{\mathbf{h}}) - \widetilde{\mathcal{A}}_{\mathbf{h}}(\mathbf{w}, \mathbf{Z}_{\mathbf{h}}),
\]
which, by duality, implies
\[
\|\mathbf{Z}_{\mathbf{h}}\|_{\mathbf{X}(\Omega)} \le \sup_{\|\mathbf{V}\|_{\mathbf{X}(\Omega)}=1} \left| \mathcal{A}(\mathbf{w}, \mathbf{V}) - \widetilde{\mathcal{A}}_{\mathbf{h}}(\mathbf{w}, \mathbf{V}) \right|.
\]
Invoking Lemma~\ref{le:positive} for the linear component, it follows that
\[
\|(T_l - T_{\mathbf{h},l})\mathbf{w}\|_{\mathbf{X}(\Omega)} \le C \|\mathbf{h}\|_{C^1} \|\mathbf{w}\|_{\mathbf{X}(\Omega)}.
\]
For the nonlinear component evaluated at the solution $\mathbf{E}$, the assumptions on $\mathbf{g}$ yield
\[
\|(T_n - T_{\mathbf{h},n})\mathbf{E}\|_{\mathbf{X}(\Omega)} \le C_2 \|\mathbf{h}\|_{C^1} \left( \|\boldsymbol{\psi}_0\|_{L^2(\Gamma)} + \|\boldsymbol{\psi}_1\|_{L^2(\Gamma)} + \|\mathbf{E}\|_{\mathbf{X}(\Omega)} \right).
\]
Consequently, for a fixed $\mathbf{E}$, we obtain the estimate
\begin{equation}
\label{eq:operator_perturbation_impedance}
\|(T - T_{\mathbf{h}})\mathbf{E}\|_{\mathbf{X}(\Omega)} \le C_E \|\mathbf{h}\|_{C^1}.
\end{equation}

By the standard well-posedness theory for the corresponding linear Maxwell
problem, the linear operator \(T_l\) associated with the linear part of the
fixed-domain formulation is boundedly invertible on \(\mathbf X(\Omega)\).Hence \(T_l^{-1}\) exists and is bounded. We denote
\[
\widetilde C:=\|T_l^{-1}\|_{\mathcal L(\mathbf X(\Omega))}.
\]
Moreover, by Lemma~\ref{le:positive}, for
\(\|\mathbf h\|_{C^1}\) sufficiently small, we have
\[
\|T_l^{-1}(T_{\mathbf h,l}-T_l)\|_{\mathcal L(\mathbf X(\Omega))}\le
\frac12 .
\]
Since
$T_{\mathbf h,l}=T_l\left(I+T_l^{-1}(T_{\mathbf h,l}-T_l)
\right),$
a standard Neumann series argument implies that \(T_{\mathbf h,l}\) is
invertible and satisfies the uniform stability estimate
\[
\|T_{\mathbf h,l}^{-1}\|_{\mathcal L(\mathbf X(\Omega))}
\le\frac{\|T_l^{-1}\|_{\mathcal L(\mathbf X(\Omega))}
}{1-\|T_l^{-1}(T_{\mathbf h,l}-T_l)\|_{\mathcal L(\mathbf X(\Omega))}
}\le2\widetilde C .\]

Consider now the difference $\hat{\mathbf E}_{\mathbf h} - \mathbf E$. Since $T_{\mathbf h}(\hat{\mathbf E}_{\mathbf h}) = T(\mathbf E) =  F$, it follows that
\[
T_{\mathbf h}(\hat{\mathbf E}_{\mathbf h}) - T_{\mathbf h}(\mathbf E) = T(\mathbf E) - T_{\mathbf h}(\mathbf E).
\]
Decomposing the operators into their linear and nonlinear components, we obtain
\[
T_{\mathbf h,l}(\hat{\mathbf E}_{\mathbf h} - \mathbf E) + \big[ T_{\mathbf h,n}(\hat{\mathbf E}_{\mathbf h}) - T_{\mathbf h,n}(\mathbf E) \big] = (T_l - T_{\mathbf h,l})\mathbf E + (T_n - T_{\mathbf h,n})\mathbf E.
\]
By applying the inverse $T_{\mathbf h,l}^{-1}$, whose existence and uniform boundedness were established above, we find
\begin{equation}\label{eq:difference_identity}
\hat{\mathbf E}_{\mathbf h} - \mathbf E + T_{\mathbf h,l}^{-1} \big[ T_{\mathbf h,n}(\hat{\mathbf E}_{\mathbf h}) - T_{\mathbf h,n}(\mathbf E) \big] = T_{\mathbf h,l}^{-1} \big[ (T_l - T_{\mathbf h,l})\mathbf E + (T_n - T_{\mathbf h,n})\mathbf E \big].
\end{equation}

We first estimate the right-hand side of \eqref{eq:difference_identity}. Invoking the perturbation estimate \eqref{eq:operator_perturbation_impedance} and the uniform stability bound for $T_{\mathbf h,l}^{-1}$, it follows that
\[
\left\| T_{\mathbf h,l}^{-1} \big[ (T_l - T_{\mathbf h,l})\mathbf E + (T_n - T_{\mathbf h,n})\mathbf E \big] \right\|_{\mathbf X(\Omega)} \le 2\widetilde{C} C_E \|\mathbf h\|_{C^1}.
\]
For the left-hand side of \eqref{eq:difference_identity}, the reverse triangle inequality yields the lower bound
\[
\begin{aligned}
&\left\| \hat{\mathbf E}_{\mathbf h} - \mathbf E + T_{\mathbf h,l}^{-1} \big[ T_{\mathbf h,n}(\hat{\mathbf E}_{\mathbf h}) - T_{\mathbf h,n}(\mathbf E) \big] \right\|_{\mathbf X(\Omega)} \\
&\quad \ge \|\hat{\mathbf E}_{\mathbf h} - \mathbf E\|_{\mathbf X(\Omega)} - \left\| T_{\mathbf h,l}^{-1} \big[ T_{\mathbf h,n}(\hat{\mathbf E}_{\mathbf h}) - T_{\mathbf h,n}(\mathbf E) \big] \right\|_{\mathbf X(\Omega)}.
\end{aligned}
\]

To bound the nonlinear difference, we invoke the definition of $T_{\mathbf{h},n}$ and the transformation properties of the tangential trace and surface measure to write
\[
\left\| T_{\mathbf{h},n}(\hat{\mathbf{E}}_{\mathbf{h}}) - T_{\mathbf{h},n}(\mathbf{E}) \right\|_{\mathbf{X}(\Omega)} = \sup_{\|\mathbf{V}\|_{\mathbf{X}(\Omega)}=1} \left| \widetilde{\mathcal{A}}_{\mathbf{h},n}(\hat{\mathbf{E}}_{\mathbf{h}}, \mathbf{V}) - \widetilde{\mathcal{A}}_{\mathbf{h},n}(\mathbf{E}, \mathbf{V}) \right|.
\]
From the definition of the pulled-back form \eqref{eq:weak_form_pullback_impedance}, we have
\[
\begin{aligned}
&\widetilde{\mathcal{A}}_{\mathbf{h},n}(\hat{\mathbf{E}}_{\mathbf{h}}, \mathbf{V}) - \widetilde{\mathcal{A}}_{\mathbf{h},n}(\mathbf{E}, \mathbf{V}) \\
&\quad = - \int_\Gamma \Big[ \mathbf{g}(\varphi_{\mathbf{h}}(\cdot), \gamma_{T,\mathbf{h}} \hat{\mathbf{E}}_{\mathbf{h}}) - \mathbf{g}(\varphi_{\mathbf{h}}(\cdot), \gamma_{T,\mathbf{h}} \mathbf{E}) \Big] \cdot \gamma_{T,\mathbf{h}} \overline{\mathbf{V}} \, \omega_{\mathbf{h}} \, \mathrm{d}s.
\end{aligned}
\]
By leveraging the uniform $L^2$-Lipschitz continuity of $\mathbf{g}$, the stability of the transported trace operator, and the uniform boundedness of the surface Jacobian $\omega_{\mathbf{h}}$, we obtain
\[
\begin{aligned}
\left\| \mathbf{g}(\varphi_{\mathbf{h}}(\cdot), \gamma_{T,\mathbf{h}} \hat{\mathbf{E}}_{\mathbf{h}}) - \mathbf{g}(\varphi_{\mathbf{h}}(\cdot), \gamma_{T,\mathbf{h}} \mathbf{E}) \right\|_{L^2(\Gamma)} &\le L_g \|\gamma_{T,\mathbf{h}} (\hat{\mathbf{E}}_{\mathbf{h}} - \mathbf{E})\|_{L^2(\Gamma)} \\
&\le L_g C_\gamma \|\hat{\mathbf{E}}_{\mathbf{h}} - \mathbf{E}\|_{\mathbf{X}(\Omega)}.
\end{aligned}
\]
Furthermore, using the trace estimate $\|\gamma_{T,\mathbf{h}} \mathbf{V}\|_{L^2(\Gamma)} \le C_\gamma \|\mathbf{V}\|_{\mathbf{X}(\Omega)}$, it follows that
\[
\|T_{\mathbf{h},n}(\hat{\mathbf{E}}_{\mathbf{h}}) - T_{\mathbf{h},n}(\mathbf{E})\|_{\mathbf{X}(\Omega)} \le C_\gamma^2 L_g \|\hat{\mathbf{E}}_{\mathbf{h}} - \mathbf{E}\|_{\mathbf{X}(\Omega)}.
\]
Consequently, the nonlinear term on the left-hand side of \eqref{eq:difference_identity} satisfies
\[
\left\| T_{\mathbf{h},l}^{-1} \big[ T_{\mathbf{h},n}(\hat{\mathbf{E}}_{\mathbf{h}}) - T_{\mathbf{h},n}(\mathbf{E}) \big] \right\|_{\mathbf{X}(\Omega)} \le 2\widetilde{C} C_\gamma^2 L_g \|\hat{\mathbf{E}}_{\mathbf{h}} - \mathbf{E}\|_{\mathbf{X}(\Omega)}.
\]

Combining the preceding estimates yields
\[
\left( 1 - 2\widetilde{C} C_\gamma^2 L_g \right) \|\hat{\mathbf{E}}_{\mathbf{h}} - \mathbf{E}\|_{\mathbf{X}(\Omega)} \le C_3 \|\mathbf{h}\|_{C^1}.
\]
Under the assumption that $L_g$ is sufficiently small such that $2\widetilde{C} C_\gamma^2 L_g < 1$, we deduce
\[
\|\hat{\mathbf{E}}_{\mathbf{h}} - \mathbf{E}\|_{\mathbf{X}(\Omega)} \le \frac{C_3}{1 - 2\widetilde{C} C_\gamma^2 L_g} \|\mathbf{h}\|_{C^1}.
\]
Finally, observing the continuous embedding $\mathbf{X}(\Omega) \hookrightarrow H(\operatorname{curl}; \Omega)$, there exists a constant $C > 0$ such that
\[
\|\hat{\mathbf{E}}_{\mathbf{h}} - \mathbf{E}\|_{H(\operatorname{curl}; \Omega)} \le \frac{C C_3}{1 - 2\widetilde{C} C_\gamma^2 L_g} \|\mathbf{h}\|_{C^1}.
\]
Taking the limit $\|\mathbf{h}\|_{C^1} \to 0$ yields the desired result \eqref{eq:continuity_limit}, completing the proof.

\end{proof}

The linearized Maxwell problem corresponding to the nonlinear impedance condition is formulated as the following $\mathbb{R}$-linear variational problem: Given a functional $\mathbf{F} \in \mathbf{X}(\Omega)^*$, find $\mathbf{W} \in \mathbf{X}(\Omega)$ such that
\begin{equation}
\widehat{\mathcal{A}}(\mathbf{W}, \mathbf{V}) = \langle F, \mathbf{V} \rangle \qquad \forall \, \mathbf{V} \in \mathbf{X}(\Omega),
\end{equation}
where the form $\widehat{\mathcal{A}}$ is defined by
\begin{equation}
\label{eq:linearized_weak_form}
\begin{aligned}
\widehat{\mathcal{A}}(\mathbf{W}, \mathbf{V}) &:= \int_{\Omega} \left[ (\nabla \times \mathbf{W}) \cdot (\nabla \times \overline{\mathbf{V}}) - k^2 \mathbf{W} \cdot \overline{\mathbf{V}} \right] \, \mathrm{d}\boldsymbol{x} \\
&\quad + \mathrm{i}k \lambda\int_{\Gamma}  \, (\gamma_T \mathbf{W}) \cdot (\gamma_T \overline{\mathbf{V}}) \, \mathrm{d}s - \int_{\Gamma} \mathbf{g}_{\boldsymbol{z}}(\cdot, \gamma_T \mathbf{E}; \gamma_T \mathbf{W}) \cdot \gamma_T \overline{\mathbf{V}} \, \mathrm{d}s \\
&\quad + \mathrm{i}k \langle \Lambda \gamma_t \mathbf{W}, \gamma_T\overline{\mathbf{V}} \rangle_{\Gamma_R}.
\end{aligned}
\end{equation}
 Here \(\mathbf E\) is the weak solution of the nonlinear impedance problem at
the reference configuration; thus \(\widehat{\mathcal A}\) is the
\(\mathbb R\)-linearization of the nonlinear variational form at
\(\mathbf E\). Henceforth, we assume that this linearized problem is well
posed in \(\mathbf X(\Omega)\), in the sense that the operator induced by
\(\widehat{\mathcal A}\) is a bounded \(\mathbb R\)-linear isomorphism from
\(\mathbf X(\Omega)\) onto \(\mathbf X(\Omega)^*\).
\begin{theo}\label{th:symmetrisation1}
Assume that the nonlinear impedance problem \eqref{eq:main_formulation} and its associated $\mathbb{R}$-linearized counterpart \eqref{eq:linearized_weak_form} are well-posed in the trace-regular framework $\mathbf{X}(\Omega)$. Let $\mathbf{E} \in \mathbf{X}(\Omega)$ be the weak solution to \eqref{eq:main_formulation}. Suppose that $\mathbf{g}$ satisfies Assumption~\ref{ass:enhanced_g} and that the solution map $\mathbf{h} \mapsto \hat{\mathbf{E}}_{\mathbf{h}}$ is continuous at $\mathbf{h} = \mathbf{0}$ in the sense of Theorem~\ref{th:continuity_domain_dependence}.

Then, the electric field is Fr\'{e}chet differentiable with respect to the domain perturbation $\mathbf{h} \in C_c^1(B_R; \mathbb{R}^3)$ at $\mathbf{h} = \mathbf{0}$. Specifically, there exists a material derivative $\mathbf{W} \in \mathbf{X}(\Omega)$, depending linearly on $\mathbf{h}$, such that
\begin{equation*}
\lim_{\|\mathbf{h}\|_{C^1} \to 0} \frac{\|\hat{\mathbf{E}}_{\mathbf{h}} - \mathbf{E} - \mathbf{W}\|_{\mathbf{X}(\Omega)}}{\|\mathbf{h}\|_{C^1}} = 0.
\end{equation*}
The material derivative $\mathbf{W}$ is the unique weak solution to the $\mathbb{R}$-linear variational problem
\[
\widehat{\mathcal{A}}(\mathbf{W}, \mathbf{V}) = \mathcal{L}_{\mathbf{h}}(\mathbf{V}) \qquad \forall \, \mathbf{V} \in \mathbf{X}(\Omega),
\]
where the right-hand side functional $\mathcal{L}_{\mathbf{h}} \in \mathbf{X}(\Omega)^*$ is defined by
\begin{equation}\label{eq:material_derivative}
\begin{aligned}
\mathcal L_{\mathbf h}(\mathbf V)
&=
\int_{\Omega}
\Bigl[
(\nabla\times \mathbf E)^{\top}
\bigl(
(\operatorname{div}\mathbf h)I
-
J_{\mathbf h}
-
J_{\mathbf h}^{\top}
\bigr)
\nabla\times\overline{\mathbf V}
\\
&\qquad
-
k^2\mathbf E^{\top}
\bigl(
-(\operatorname{div}\mathbf h)I
+
J_{\mathbf h}
+
J_{\mathbf h}^{\top}
\bigr)
\overline{\mathbf V}
\Bigr]\,d\mathbf x
\\
&\quad
-ik\lambda
\int_{\Gamma}
\Bigl[\dot\gamma_{T,\mathbf h}\mathbf E\cdot\gamma_T\overline{\mathbf V}+
\gamma_T\mathbf E\cdot\dot\gamma_{T,\mathbf h}\overline{\mathbf V}+
\dot\omega_{\mathbf h}\,
\gamma_T\mathbf E\cdot\gamma_T\overline{\mathbf V}
\Bigr]\,ds\\&\quad+
\int_{\Gamma}
\Bigl[
\nabla_{\mathbf x}\mathbf g(\cdot,\gamma_T\mathbf E)\mathbf h+
\mathbf g_{\mathbf z}
\bigl(\cdot,\gamma_T\mathbf E;
\dot\gamma_{T,\mathbf h}\mathbf E\bigr)\Bigr]\cdot
\gamma_T\overline{\mathbf V}\,ds
\\
&\quad+\int_{\Gamma}
\mathbf g(\cdot,\gamma_T\mathbf E)
\cdot
\dot\gamma_{T,\mathbf h}\overline{\mathbf V}\,ds\\
&\quad+
\int_{\Gamma}\dot\omega_{\mathbf h}\,
\mathbf g(\cdot,\gamma_T\mathbf E)\cdot
\gamma_T\overline{\mathbf V}\,ds .
\end{aligned}
\end{equation}

\end{theo}

\begin{proof}
We define the increment 
\[
\Delta_{\mathbf{h}} := \hat{\mathbf{E}}_{\mathbf{h}} - \mathbf{E}.
\]
The stability estimate established in the proof of Theorem~\ref{th:continuity_domain_dependence} ensures that, within the trace-regular framework, 
\[
\|\Delta_{\mathbf{h}}\|_{\mathbf{X}(\Omega)} = O(\|\mathbf{h}\|_{C^1}).
\]
By construction, the pulled-back perturbed field $\hat{\mathbf{E}}_{\mathbf{h}}$ and the reference field $\mathbf{E}$ satisfy the respective variational problems
\[
\widetilde{\mathcal{A}}_{\mathbf{h}}(\hat{\mathbf{E}}_{\mathbf{h}}, \mathbf{V}) = \langle {F}, \mathbf{V} \rangle \quad \text{and} \quad \mathcal{A}(\mathbf{E}, \mathbf{V}) = \langle {F}, \mathbf{V} \rangle,
\]
which together yield the fundamental consistency identity
\begin{equation}\label{eq:material_basic_identity}
\widetilde{\mathcal{A}}_{\mathbf{h}}(\hat{\mathbf{E}}_{\mathbf{h}}, \mathbf{V}) - \mathcal{A}(\mathbf{E}, \mathbf{V}) = 0 \qquad \forall \, \mathbf{V} \in \mathbf{X}(\Omega).
\end{equation}

We now perform an asymptotic expansion of the left-hand side of \eqref{eq:material_basic_identity}. According to Lemma~\ref{le:positive}, the geometric coefficients admit the first-order expansions
\[
\mathcal{M}_{\mathbf{h}} = \mathrm{Id} - (\operatorname{div} \mathbf{h})\mathrm{Id} + J_{\mathbf{h}} + J_{\mathbf{h}}^{\top} + o(\|\mathbf{h}\|_{C^1}), \quad \mathcal{N}_{\mathbf{h}} = \mathrm{Id} + (\operatorname{div} \mathbf{h})\mathrm{Id} - J_{\mathbf{h}} - J_{\mathbf{h}}^{\top} + o(\|\mathbf{h}\|_{C^1}).
\]
Throughout this section, we utilize the notation $\dot{(\,\cdot\,)}$ to denote the first-order  variation in the direction of the perturbation field $\mathbf{h}$. Specifically, the surface Jacobian and the transported tangential trace expand as
\[
\omega_{\mathbf{h}} = 1 + \dot{\omega}_{\mathbf{h}} + o(\|\mathbf{h}\|_{C^1}), \qquad \gamma_{T,\mathbf{h}}\mathbf{U} = \gamma_T\mathbf{U} + \dot{\gamma}_{T,\mathbf{h}}\mathbf{U} + o(\|\mathbf{h}\|_{C^1}) \quad \text{in } L^2(\Gamma),
\]
This operator satisfies the stability estimate $\|\dot{\gamma}_{T,\mathbf{h}}\mathbf{U}\|_{L^2(\Gamma)} \le C \|\mathbf{h}\|_{C^1} \|\mathbf{U}\|_{\mathbf{X}(\Omega)}$, which holds uniformly over the class of fields considered herein. Regarding the nonlinear term, Assumption~\ref{ass:enhanced_g} provides the Nemytskii expansion
\[
\mathbf{g}\bigl( \varphi_{\mathbf{h}}(\cdot), \gamma_{T,\mathbf{h}}\hat{\mathbf{E}}_{\mathbf{h}} \bigr) = \mathbf{g}(\cdot, \gamma_T\mathbf{E}) + \mathbf{g}_{\boldsymbol{z}}(\cdot, \gamma_T\mathbf{E}; \gamma_T\Delta_{\mathbf{h}}) + \nabla_{\boldsymbol{x}}\mathbf{g}(\cdot, \gamma_T\mathbf{E})\mathbf{h} + \mathbf{g}_{\boldsymbol{z}}(\cdot, \gamma_T\mathbf{E}; \dot{\gamma}_{T,\mathbf{h}}\mathbf{E}) + \rho_{\mathbf{h}}^{g},
\]
where $\|\rho_{\mathbf{h}}^{g}\|_{L^2(\Gamma)} = o(\|\mathbf{h}\|_{C^1})$. Note that we have utilized the fact that $\|\Delta_{\mathbf{h}}\|_{\mathbf{X}(\Omega)} = O(\|\mathbf{h}\|_{C^1})$ to justify the order of the remainder.

We define the remainder $\mathbf{R}_{\mathbf{h}} := \hat{\mathbf{E}}_{\mathbf{h}} - \mathbf{E} - \mathbf{W}$. By invoking the definition of the linearized form $\widehat{\mathcal{A}}$ and the material derivative $\mathbf{W}$, the identity \eqref{eq:material_basic_identity} leads to the error equation
\[
\widehat{\mathcal{A}}(\mathbf{R}_{\mathbf{h}}, \mathbf{V}) = r_{\mathbf{h}}^{\Omega}(\mathbf{V}) + r_{\mathbf{h}}^{\lambda}(\mathbf{V}) + r_{\mathbf{h}}^{g}(\mathbf{V}),
\]
where the residuals are defined and estimated as follows.

The volume residual is given by
\[
\begin{aligned}
r_{\mathbf{h}}^{\Omega}(\mathbf{V}) &= \int_\Omega (\nabla \times \Delta_{\mathbf{h}})^\top (\mathrm{Id} - \mathcal{M}_{\mathbf{h}}) \nabla \times \overline{\mathbf{V}} \,\mathrm{d}\boldsymbol{x} - k^2 \int_\Omega \Delta_{\mathbf{h}}^\top (\mathrm{Id} - \mathcal{N}_{\mathbf{h}}) \overline{\mathbf{V}} \,\mathrm{d}\boldsymbol{x} \\
&\quad + \int_\Omega (\nabla \times \mathbf{E})^\top \Bigl[ \mathrm{Id} - \mathcal{M}_{\mathbf{h}} - \bigl( (\operatorname{div}\mathbf{h})\mathrm{Id} - J_{\mathbf{h}} - J_{\mathbf{h}}^\top \bigr) \Bigr] \nabla \times \overline{\mathbf{V}} \,\mathrm{d}\boldsymbol{x} \\
&\quad - k^2 \int_\Omega \mathbf{E}^\top \Bigl[ \mathrm{Id} - \mathcal{N}_{\mathbf{h}} + \bigl( (\operatorname{div}\mathbf{h})\mathrm{Id} - J_{\mathbf{h}} - J_{\mathbf{h}}^\top \bigr) \Bigr] \overline{\mathbf{V}} \,\mathrm{d}\boldsymbol{x}.
\end{aligned}
\]
Since $\|\Delta_{\mathbf{h}}\|_{\mathbf{X}(\Omega)} = O(\|\mathbf{h}\|_{C^1})$ and $\|\mathrm{Id} - \mathcal{M}_{\mathbf{h}}\|, \|\mathrm{Id} - \mathcal{N}_{\mathbf{h}}\| = O(\|\mathbf{h}\|_{C^1})$, the first two integrals are of order $O(\|\mathbf{h}\|_{C^1}^2)$. Furthermore, as the bracketed terms are $o(\|\mathbf{h}\|_{C^1})$ by the geometric expansions in Lemma~\ref{le:positive}, it follows that
\[
\|r_{\mathbf{h}}^{\Omega}\|_{\mathbf{X}(\Omega)^*} = o(\|\mathbf{h}\|_{C^1}).
\]

Next, we define the impedance boundary residual as
\[
\begin{aligned}
r_{\mathbf{h}}^{\lambda}(\mathbf{V}) &= \mathrm{i}k \int_\Gamma \lambda \Bigl[ (\gamma_{T,\mathbf{h}}\hat{\mathbf{E}}_{\mathbf{h}}) \cdot (\gamma_{T,\mathbf{h}}\overline{\mathbf{V}}) \omega_{\mathbf{h}} - \gamma_T\mathbf{E}\cdot\gamma_T\overline{\mathbf{V}} - \gamma_T\Delta_{\mathbf{h}}\cdot\gamma_T\overline{\mathbf{V}} \\
&\qquad - \dot{\gamma}_{T,\mathbf{h}}\mathbf{E} \cdot \gamma_T\overline{\mathbf{V}} - \gamma_T\mathbf{E} \cdot \dot{\gamma}_{T,\mathbf{h}}\overline{\mathbf{V}} - (\gamma_T\mathbf{E}\cdot\gamma_T\overline{\mathbf{V}}) \dot{\omega}_{\mathbf{h}} \Bigr] \,\mathrm{d}s.
\end{aligned}
\]
Substituting the $L^2(\Gamma)$-expansions $\gamma_{T,\mathbf{h}}\hat{\mathbf{E}}_{\mathbf{h}} = \gamma_T\mathbf{E} + \gamma_T\Delta_{\mathbf{h}} + \dot{\gamma}_{T,\mathbf{h}}\mathbf{E} + o(\|\mathbf{h}\|_{C^1})$ and $\gamma_{T,\mathbf{h}}\overline{\mathbf{V}} = \gamma_T\overline{\mathbf{V}} + \dot{\gamma}_{T,\mathbf{h}}\overline{\mathbf{V}} + o(\|\mathbf{h}\|_{C^1})$, along with $\omega_{\mathbf{h}} = 1 + \dot{\omega}_{\mathbf{h}} + o(\|\mathbf{h}\|_{C^1})$ in $L^\infty(\Gamma)$, we observe that the zeroth- and first-order terms in $\mathbf{h}$ cancel exactly. Consequently, 
\[
\|r_{\mathbf{h}}^{\lambda}\|_{\mathbf{X}(\Omega)^*} = o(\|\mathbf{h}\|_{C^1}).
\]

Finally, we consider the nonlinear boundary residual
\[
\begin{aligned}
r_{\mathbf{h}}^{g}(\mathbf{V}) &= - \int_\Gamma \Bigl[ \mathbf{g}(\varphi_{\mathbf{h}}(\cdot), \gamma_{T,\mathbf{h}}\hat{\mathbf{E}}_{\mathbf{h}}) \cdot \gamma_{T,\mathbf{h}}\overline{\mathbf{V}} \omega_{\mathbf{h}} - \mathbf{g}(\cdot,\gamma_T\mathbf{E}) \cdot \gamma_T\overline{\mathbf{V}} \\
&\qquad - \mathbf{g}_{\boldsymbol{z}}(\cdot,\gamma_T\mathbf{E};\gamma_T\Delta_{\mathbf{h}}) \cdot \gamma_T\overline{\mathbf{V}} - \Bigl( \nabla_{\boldsymbol{x}}\mathbf{g}(\cdot,\gamma_T\mathbf{E})\mathbf{h} + \mathbf{g}_{\boldsymbol{z}}(\cdot,\gamma_T\mathbf{E}; \dot{\gamma}_{T,\mathbf{h}}\mathbf{E}) \Bigr) \cdot \gamma_T\overline{\mathbf{V}} \\
&\qquad - \mathbf{g}(\cdot,\gamma_T\mathbf{E}) \cdot \dot{\gamma}_{T,\mathbf{h}}\overline{\mathbf{V}} - \mathbf{g}(\cdot,\gamma_T\mathbf{E}) \cdot \gamma_T\overline{\mathbf{V}} \dot{\omega}_{\mathbf{h}} \Bigr] \,\mathrm{d}s.
\end{aligned}
\]
By invoking the $L^2$-Nemytskii Taylor expansion for $\mathbf{g}$ from Assumption~\ref{ass:enhanced_g} and the linearizations of $\gamma_{T,\mathbf{h}}$ and $\omega_{\mathbf{h}}$, it follows that the terms listed above account for all contributions of order up to $O(\|\mathbf{h}\|_{C^1})$. The remainder consists of higher-order terms that are $o(\|\mathbf{h}\|_{C^1})$ in $\mathbf{X}(\Omega)^*$; hence,
\[
\|r_{\mathbf{h}}^{g}\|_{\mathbf{X}(\Omega)^*} = o(\|\mathbf{h}\|_{C^1}).
\]
Combining the estimates for the volume, impedance, and nonlinear boundary residuals, we conclude that
\[
\left\| r_{\mathbf{h}}^{\Omega} + r_{\mathbf{h}}^{\lambda} + r_{\mathbf{h}}^{g} \right\|_{\mathbf{X}(\Omega)^*} = o(\|\mathbf{h}\|_{C^1}).
\]

Let $\mathbf{F}_{\mathbf{h}} \in \mathbf{X}(\Omega)^*$ be the residual functional defined by the sum
\[
\langle {F}_{\mathbf{h}}, \mathbf{V} \rangle := r_{\mathbf{h}}^{\Omega}(\mathbf{V}) + r_{\mathbf{h}}^{\lambda}(\mathbf{V}) + r_{\mathbf{h}}^{g}(\mathbf{V}),
\]
so that the previous estimates imply $\|\mathbf{F}_{\mathbf{h}}\|_{\mathbf{X}(\Omega)^*} = o(\|\mathbf{h}\|_{C^1})$. The remainder $\mathbf{R}_{\mathbf{h}}$ satisfies the variational identity
\[
\widehat{\mathcal{A}}(\mathbf{R}_{\mathbf{h}}, \mathbf{V}) = \langle {F}_{\mathbf{h}}, \mathbf{V} \rangle \qquad \forall \, \mathbf{V} \in \mathbf{X}(\Omega).
\]
By the assumed well-posedness of the $\mathbb{R}$-linearized problem, the operator induced by the form $\widehat{\mathcal{A}}$ is a bounded linear isomorphism. Consequently, there exists a stability constant $C > 0$, independent of $\mathbf{h}$ for sufficiently small $\|\mathbf{h}\|_{C^1}$, such that
\[
\|\mathbf{R}_{\mathbf{h}}\|_{\mathbf{X}(\Omega)} \le C \|{F}_{\mathbf{h}}\|_{\mathbf{X}(\Omega)^*} = o(\|\mathbf{h}\|_{C^1}).
\]
Since$
\mathbf R_{\mathbf h}
=\hat{\mathbf E}_{\mathbf h}-\mathbf E-\mathbf W,
$
we obtain
\[
\lim_{\|\mathbf h\|_{C^1}\to0}
\frac{\|\hat{\mathbf E}_{\mathbf h}-\mathbf E-\mathbf W\|_{\mathbf X(\Omega)}
}{\|\mathbf h\|_{C^1}}=0,
\]
which completes the proof of Theorem~\ref{th:symmetrisation1}.
\end{proof}

\subsection{The nonlinear PEC boundary condition}

We now extend the perturbation framework to the dual formulation of the
NPEC problem. For a given perturbed obstacle \(D_{\mathbf h}\), we introduce
the auxiliary variable
\[
\mathbf Q_{\mathbf h}:=\nabla_{\boldsymbol y}\times\mathbf E_{\mathbf h}.
\]
The perturbed mixed problem is posed in the corresponding product space
\[
\mathbb X_{\rm p}(\Omega_{\mathbf h})
:=
\mathbf X(\Omega_{\mathbf h})
\times
H(\operatorname{curl};\Omega_{\mathbf h}),
\]
where \(\mathbf X(\Omega_{\mathbf h})\) denotes the trace-regular Maxwell
space on the perturbed domain.
Specifically, we seek a pair $(\mathbf{E}_{\mathbf{h}}, \mathbf{Q}_{\mathbf{h}}) \in \mathbb{X}_{\mathrm{p}}(\Omega_{\mathbf{h}})$ such that 
\begin{equation}
\label{eq:perturbed_NPEC_mixed}
\mathcal{A}_{\mathbf{h}}^{\mathrm{p}} \left( (\mathbf{E}_{\mathbf{h}}, \mathbf{Q}_{\mathbf{h}}), (\boldsymbol{\tau}_{\mathbf{h}}, \boldsymbol{\nu}_{\mathbf{h}}) \right) = \mathcal{F}^{\mathrm{p}}(\boldsymbol{\tau}_{\mathbf{h}}, \boldsymbol{\nu}_{\mathbf{h}}) \qquad \forall (\boldsymbol{\tau}_{\mathbf{h}}, \boldsymbol{\nu}_{\mathbf{h}}) \in \mathbb{X}_{\mathrm{p}}(\Omega_{\mathbf{h}}),
\end{equation}
where the nonlinear form $\mathcal{A}_{\mathbf{h}}^{\mathrm{p}}$ is defined by
\begin{equation}
\label{eq:NPEC_form_definition}
\begin{aligned}
\mathcal{A}_{\mathbf{h}}^{\mathrm{p}} \left( (\mathbf{E}_{\mathbf{h}}, \mathbf{Q}_{\mathbf{h}}), (\boldsymbol{\tau}_{\mathbf{h}}, \boldsymbol{\nu}_{\mathbf{h}}) \right) &:= \int_{\Omega_{\mathbf{h}}} \mathbf{Q}_{\mathbf{h}} \cdot \overline{\boldsymbol{\tau}_{\mathbf{h}}} \,\mathrm{d}\boldsymbol{y} - \int_{\Omega_{\mathbf{h}}} \mathbf{E}_{\mathbf{h}} \cdot (\nabla_{\boldsymbol{y}} \times \overline{\boldsymbol{\tau}_{\mathbf{h}}}) \,\mathrm{d}\boldsymbol{y} \\
&\quad + \frac{\mathrm{i}}{k} \left\langle \Lambda^{-1} \gamma_t \mathbf{Q}_{\mathbf{h}}, \gamma_T \boldsymbol{\tau}_{\mathbf{h}} \right\rangle_{\Gamma_R} + \int_{\Gamma_{\mathbf{h}}} \mathbf{g}(\boldsymbol{y}, \gamma_T^{\mathbf{h}} \mathbf{E}_{\mathbf{h}}) \cdot \gamma_T^{\mathbf{h}} \overline{\boldsymbol{\tau}_{\mathbf{h}}} \,\mathrm{d}s_{\boldsymbol{y}} \\
&\quad + \int_{\Omega_{\mathbf{h}}} (\nabla_{\boldsymbol{y}} \times \mathbf{Q}_{\mathbf{h}}) \cdot \overline{\boldsymbol{\nu}_{\mathbf{h}}} \,\mathrm{d}\boldsymbol{y} - k^2 \int_{\Omega_{\mathbf{h}}} \mathbf{E}_{\mathbf{h}} \cdot \overline{\boldsymbol{\nu}_{\mathbf{h}}} \,\mathrm{d}\boldsymbol{y}.
\end{aligned}
\end{equation}
Since the perturbation $\mathbf{h}$ is supported away from $\Gamma_R$, both the artificial boundary and the source functional $\mathcal{F}^{\mathrm{p}}$ remain invariant under the domain deformation.

We now pull the mixed system back to the fixed reference domain $\Omega$. As in the impedance case, we employ the covariant Piola transform for both variables:
\[
\hat{\mathbf{E}}_{\mathbf{h}} := J_\varphi^\top(\mathbf{E}_{\mathbf{h}} \circ \varphi_{\mathbf{h}}), \qquad \hat{\mathbf{Q}}_{\mathbf{h}} := J_\varphi^\top(\mathbf{Q}_{\mathbf{h}} \circ \varphi_{\mathbf{h}}).
\]
It is important to note that while $\mathbf{Q}_{\mathbf{h}} = \nabla_{\boldsymbol{y}} \times \mathbf{E}_{\mathbf{h}}$ holds in the perturbed domain, the covariant Piola transform does not commute with the curl operator in a simple manner. Instead, the transformation rule for the curl implies
\[
\hat{\mathbf{Q}}_{\mathbf{h}} = \mathcal{M}_{\mathbf{h}} (\nabla \times \hat{\mathbf{E}}_{\mathbf{h}}), \quad \text{or equivalently,} \quad \mathcal{N}_{\mathbf{h}} \hat{\mathbf{Q}}_{\mathbf{h}} = \nabla \times \hat{\mathbf{E}}_{\mathbf{h}}.
\]
In the pulled-back mixed formulation, $\hat{\mathbf{Q}}_{\mathbf{h}}$ is treated as an independent unknown in the product space $\mathbb{X}_{\mathrm{p}}(\Omega)$, rather than being explicitly identified with $\nabla \times \hat{\mathbf{E}}_{\mathbf{h}}$.

The pulled-back perturbed NPEC problem thus consists in finding a pair $(\hat{\mathbf{E}}_{\mathbf{h}}, \hat{\mathbf{Q}}_{\mathbf{h}}) \in \mathbb{X}_{\mathrm{p}}(\Omega)$ such that, for all $(\boldsymbol{\tau}, \boldsymbol{\nu}) \in \mathbb{X}_{\mathrm{p}}(\Omega)$,
\[
\widetilde{\mathcal{A}}^{\mathrm{p}}_{\mathbf{h}} \left( (\hat{\mathbf{E}}_{\mathbf{h}}, \hat{\mathbf{Q}}_{\mathbf{h}}), (\boldsymbol{\tau}, \boldsymbol{\nu}) \right) = \mathcal{F}^{\mathrm{p}}(\boldsymbol{\tau}, \boldsymbol{\nu}),
\]
which is expressed explicitly as
\[
\begin{aligned}
&\int_\Omega \hat{\mathbf{Q}}_{\mathbf{h}}^{\top} \mathcal{N}_{\mathbf{h}} \overline{\boldsymbol{\tau}} \,\mathrm{d}\boldsymbol{x} - \int_\Omega \hat{\mathbf{E}}_{\mathbf{h}} \cdot (\nabla \times \overline{\boldsymbol{\tau}}) \,\mathrm{d}\boldsymbol{x} + \frac{\mathrm{i}}{k} \left\langle \Lambda^{-1} \gamma_t \hat{\mathbf{Q}}_{\mathbf{h}}, \gamma_T \boldsymbol{\tau} \right\rangle_{\Gamma_R} \\
&\quad +\int_\Gamma \mathbf{g}(\varphi_{\mathbf{h}}(\cdot), \gamma_{T,\mathbf{h}} \hat{\mathbf{E}}_{\mathbf{h}}) \cdot \gamma_{T,\mathbf{h}} \overline{\boldsymbol{\tau}} \, \omega_{\mathbf{h}} \,\mathrm{d}s + \int_\Omega (\nabla \times \hat{\mathbf{Q}}_{\mathbf{h}}) \cdot \overline{\boldsymbol{\nu}} \,\mathrm{d}\boldsymbol{x} \\
&\quad - k^2 \int_\Omega \hat{\mathbf{E}}_{\mathbf{h}}^{\top} \mathcal{N}_{\mathbf{h}} \overline{\boldsymbol{\nu}} \,\mathrm{d}\boldsymbol{x} = \mathcal{F}^{\mathrm{p}}(\boldsymbol{\tau}, \boldsymbol{\nu}).
\end{aligned}
\]

By the stability of the pulled-back mixed system and the perturbation estimates
derived above, we obtain the following continuity result for the mixed solution
with respect to domain perturbations.

\begin{theo}\label{thm:continuity_p_mixed}
Suppose that the NPEC problem is uniquely solvable for all admissible
perturbations \(\mathbf h\) with \(\|\mathbf h\|_{C^1}\) sufficiently small.
Assume that the associated linear mixed PEC problem is well posed in
\[
\mathbb X_{\mathrm p}(\Omega)
:=
\mathbf X(\Omega)\times \mathbf H(\operatorname{curl};\Omega).
\]
Let \(\widetilde C_{\mathrm p}\) denote the corresponding stability constant.
Assume further that Assumption~\ref{ass:enhanced_g} holds and that
\[
2\widetilde C_{\mathrm p}C_\gamma^2L_g<1.
\]
Then, for all \(\|\mathbf h\|_{C^1}\) sufficiently small, the pulled-back
solutions satisfy
\[
\|\widehat{\mathbf E}_{\mathbf h}-\mathbf E\|_{\mathbf X(\Omega)}
+
\|\widehat{\mathbf Q}_{\mathbf h}-\mathbf Q\|_{\mathbf H(\operatorname{curl};\Omega)}
\le
C\|\mathbf h\|_{C^1}.
\]
In particular,
\[
\widehat{\mathbf E}_{\mathbf h}\to\mathbf E
\quad\text{in }\mathbf H(\operatorname{curl};\Omega)
\quad\text{as }\|\mathbf h\|_{C^1}\to0.
\]
\end{theo}
\begin{proof}

Define the increments associated with the electric and auxiliary fields as
\[
\Delta_{\mathbf{h}}^E := \hat{\mathbf{E}}_{\mathbf{h}} - \mathbf{E}, \qquad \Delta_{\mathbf{h}}^Q := \hat{\mathbf{Q}}_{\mathbf{h}} - \mathbf{Q}.
\]
Let
$
T_l^{\mathrm p}:\mathbb X_{\mathrm p}(\Omega)\to
\mathbb X_{\mathrm p}(\Omega)$
be the Riesz operator associated with the linear mixed form
\(\mathcal A_l^{\mathrm p}\), namely
\[
\bigl(
T_l^{\mathrm p}(\mathbf U,\mathbf Q),
(\boldsymbol\tau,\boldsymbol\nu)
\bigr)_{\mathbb X_{\mathrm p}(\Omega)}
=
\mathcal A_l^{\mathrm p}
\bigl(
(\mathbf U,\mathbf Q),
(\boldsymbol\tau,\boldsymbol\nu)
\bigr).
\]
By the well-posedness of the associated linear mixed PEC problem,
\(T_l^{\mathrm p}\) is a bounded linear isomorphism on
\(\mathbb X_{\mathrm p}(\Omega)\). Hence
\((T_l^{\mathrm p})^{-1}\) exists and is bounded. We set
\[
\widetilde C_{\mathrm p}
:=
\bigl\|(T_l^{\mathrm p})^{-1}\bigr\|_
{\mathcal L(\mathbb X_{\mathrm p}(\Omega))}.
\]By the asymptotic expansion of $\mathcal{N}_{\mathbf{h}}$ provided in Lemma~\ref{le:positive}, it follows that
\[
\|T_{\mathbf{h},l}^{\mathrm{p}} - T_l^{\mathrm{p}}\|_{\mathcal{L}(\mathbb{X}_{\mathrm{p}}(\Omega))} \le C \|\mathbf{h}\|_{C^1}.
\]
Consequently, for $\|\mathbf{h}\|_{C^1}$ sufficiently small, $\|(T_l^{\mathrm{p}})^{-1} (T_{\mathbf{h},l}^{\mathrm{p}} - T_l^{\mathrm{p}})\|_{\mathcal{L}(\mathbb{X}_{\mathrm{p}}(\Omega))} \le 1/2$. A standard Neumann series argument then ensures the invertibility of $T_{\mathbf{h},l}^{\mathrm{p}}$ with the uniform stability bound
\[
\|(T_{\mathbf{h},l}^{\mathrm{p}})^{-1}\|_{\mathcal{L}(\mathbb{X}_{\mathrm{p}}(\Omega))} \le \frac{\|(T_l^{\mathrm{p}})^{-1}\|}{1 - \|(T_l^{\mathrm{p}})^{-1}(T_{\mathbf{h},l}^{\mathrm{p}} - T_l^{\mathrm{p}})\|} \le 2\widetilde{C}_{\mathrm{p}}.
\]

Since both the reference and pulled-back mixed formulations share the same source functional, the corresponding Riesz operators satisfy
\[
T_{\mathbf{h}}^{\mathrm{p}} (\hat{\mathbf{E}}_{\mathbf{h}}, \hat{\mathbf{Q}}_{\mathbf{h}}) = T^{\mathrm{p}}(\mathbf{E}, \mathbf{Q}),
\]
where $T^{\mathrm{p}} := T_l^{\mathrm{p}} + T_n^{\mathrm{p}}$ and $T_{\mathbf{h}}^{\mathrm{p}} := T_{\mathbf{h},l}^{\mathrm{p}} + T_{\mathbf{h},n}^{\mathrm{p}}$, with $T_n^{\mathrm{p}}$ and $T_{\mathbf{h},n}^{\mathrm{p}}$ denoting the operators induced by the nonlinear boundary terms. Rearranging this identity yields
\[
T_{\mathbf{h},l}^{\mathrm{p}}(\Delta_{\mathbf{h}}^E, \Delta_{\mathbf{h}}^Q) + \left[ T_{\mathbf{h},n}^{\mathrm{p}}(\hat{\mathbf{E}}_{\mathbf{h}}) - T_{\mathbf{h},n}^{\mathrm{p}}(\mathbf{E}) \right] = (T_l^{\mathrm{p}} - T_{\mathbf{h},l}^{\mathrm{p}})(\mathbf{E}, \mathbf{Q}) + (T_n^{\mathrm{p}} - T_{\mathbf{h},n}^{\mathrm{p}})(\mathbf{E}).
\]
Analogous to the impedance case, the fixed-field perturbation estimate follows from the expansions of $\mathcal{N}_{\mathbf{h}}$, the transported trace, and the surface Jacobian, giving
\[
\left\| (T_l^{\mathrm{p}} - T_{\mathbf{h},l}^{\mathrm{p}})(\mathbf{E}, \mathbf{Q}) + (T_n^{\mathrm{p}} - T_{\mathbf{h},n}^{\mathrm{p}})(\mathbf{E}) \right\|_{\mathbb{X}_{\mathrm{p}}(\Omega)} \le C_{\mathbf{E}, \mathbf{Q}} \|\mathbf{h}\|_{C^1}.
\]
Applying $(T_{\mathbf{h},l}^{\mathrm{p}})^{-1}$ and utilizing the stability bound, we obtain
\begin{equation} \label{eq:NPEC_stability_intermediate}
\|(\Delta_{\mathbf{h}}^E, \Delta_{\mathbf{h}}^Q)\|_{\mathbb{X}_{\mathrm{p}}(\Omega)} \le 2\widetilde{C}_{\mathrm{p}} C_{\mathbf{E}, \mathbf{Q}} \|\mathbf{h}\|_{C^1} + 2\widetilde{C}_{\mathrm{p}} \left\| T_{\mathbf{h},n}^{\mathrm{p}}(\hat{\mathbf{E}}_{\mathbf{h}}) - T_{\mathbf{h},n}^{\mathrm{p}}(\mathbf{E}) \right\|_{\mathbb{X}_{\mathrm{p}}(\Omega)}.
\end{equation}

To bound the nonlinear difference, we observe that by the definition of the Riesz operator,
\[
\begin{aligned}
\left\| T_{\mathbf{h},n}^{\mathrm{p}}(\hat{\mathbf{E}}_{\mathbf{h}}) - T_{\mathbf{h},n}^{\mathrm{p}}(\mathbf{E}) \right\|_{\mathbb{X}_{\mathrm{p}}(\Omega)} &= \sup_{\|(\boldsymbol{\tau}, \boldsymbol{\nu})\|_{\mathbb{X}_{\mathrm{p}}(\Omega)}=1} \left| \widetilde{\mathcal{A}}_{\mathbf{h},n}^{\mathrm{p}}(\hat{\mathbf{E}}_{\mathbf{h}}; \boldsymbol{\tau}) - \widetilde{\mathcal{A}}_{\mathbf{h},n}^{\mathrm{p}}(\mathbf{E}; \boldsymbol{\tau}) \right|.
\end{aligned}
\]
Utilizing the definition of the pulled-back nonlinear boundary term, we have
\[
\begin{aligned}
&\widetilde{\mathcal{A}}_{\mathbf{h},n}^{\mathrm{p}}(\hat{\mathbf{E}}_{\mathbf{h}}; \boldsymbol{\tau}) - \widetilde{\mathcal{A}}_{\mathbf{h},n}^{\mathrm{p}}(\mathbf{E}; \boldsymbol{\tau}) \\
&\quad =  \int_\Gamma \Big[ \mathbf{g}(\varphi_{\mathbf{h}}(\cdot), \gamma_{T,\mathbf{h}} \hat{\mathbf{E}}_{\mathbf{h}}) - \mathbf{g}(\varphi_{\mathbf{h}}(\cdot), \gamma_{T,\mathbf{h}} \mathbf{E}) \Big] \cdot \gamma_{T,\mathbf{h}} \overline{\boldsymbol{\tau}} \, \omega_{\mathbf{h}} \, \mathrm{d}s.
\end{aligned}
\]
Invoking the uniform $L^2$-Lipschitz continuity of $\mathbf{g}$, the uniform boundedness of $\omega_{\mathbf{h}}$, and the stability of the transported trace operator, it follows that
\[
\left\| T_{\mathbf{h},n}^{\mathrm{p}}(\hat{\mathbf{E}}_{\mathbf{h}}) - T_{\mathbf{h},n}^{\mathrm{p}}(\mathbf{E}) \right\|_{\mathbb{X}_{\mathrm{p}}(\Omega)} \le C_\gamma^2 L_g \|\Delta_{\mathbf{h}}^E\|_{\mathbf{X}(\Omega)}.
\]
Since $\|\Delta_{\mathbf{h}}^E\|_{\mathbf{X}(\Omega)} \le \|(\Delta_{\mathbf{h}}^E, \Delta_{\mathbf{h}}^Q)\|_{\mathbb{X}_{\mathrm{p}}(\Omega)}$, we obtain the inequality
\[
\left( 1 - 2\widetilde{C}_{\mathrm{p}} C_\gamma^2 L_g \right) \|(\Delta_{\mathbf{h}}^E, \Delta_{\mathbf{h}}^Q)\|_{\mathbb{X}_{\mathrm{p}}(\Omega)} \le 2\widetilde{C}_{\mathrm{p}} C_{\mathbf{E}} \|\mathbf{h}\|_{C^1}.
\]
Under the smallness condition $2\widetilde{C}_{\mathrm{p}} C_\gamma^2 L_g < 1$, we deduce the stability estimate
\[
\|\hat{\mathbf{E}}_{\mathbf{h}} - \mathbf{E}\|_{\mathbf{X}(\Omega)} + \|\hat{\mathbf{Q}}_{\mathbf{h}} - \mathbf{Q}\|_{H(\operatorname{curl}; \Omega)} \le C \|\mathbf{h}\|_{C^1}.
\]
Finally, in view of the continuous embedding $\mathbf{X}(\Omega) \hookrightarrow H(\operatorname{curl}; \Omega)$, we conclude that
\[
\lim_{\|\mathbf{h}\|_{C^1} \to 0} \|\hat{\mathbf{E}}_{\mathbf{h}} - \mathbf{E}\|_{H(\operatorname{curl}; \Omega)} = 0,
\]
which completes the proof.

\end{proof}

Having established the continuity of the pulled-back mixed solution in
Theorem~\ref{thm:continuity_p_mixed}, we now turn to its first-order
perturbation analysis.
\begin{theo}\label{th:p_material_differentiability}
Suppose that the hypotheses of Theorem~\ref{thm:continuity_p_mixed} are satisfied. Assume further that the associated $\mathbb{R}$-linearized mixed NPEC problem is well-posed in $\mathbb{X}_{\mathrm{p}}(\Omega) := \mathbf X(\Omega)\times H(\operatorname{curl};\Omega)$, in the sense that the induced operator is a bounded linear isomorphism. 
Then the pulled-back solution $(\hat{\mathbf{E}}_{\mathbf{h}}, \hat{\mathbf{Q}}_{\mathbf{h}})$ is  differentiable with respect to the domain perturbation $\mathbf{h} \in C_c^1(B_R; \mathbb{R}^3)$ at $\mathbf{h} = \mathbf{0}$. Specifically, there exists a unique pair $(\mathbf{W}, \mathbf{P}) \in \mathbb{X}_{\mathrm{p}}(\Omega)$, depending linearly and continuously on $\mathbf{h}$, such that
\begin{equation*}
\lim_{\|\mathbf{h}\|_{C^1} \to 0} \frac{\|\hat{\mathbf{E}}_{\mathbf{h}} - \mathbf{E} - \mathbf{W}\|_{\mathbf{X}(\Omega)} + \|\hat{\mathbf{Q}}_{\mathbf{h}} - \mathbf{Q} - \mathbf{P}\|_{H(\operatorname{curl};\Omega)}}{\|\mathbf{h}\|_{C^1}} = 0,
\end{equation*}
where $(\mathbf{W}, \mathbf{P})$ denotes the material derivative of the mixed field. Under the relation $\mathcal{N}_{\mathbf{h}} \hat{\mathbf{Q}}_{\mathbf{h}} = \nabla \times \hat{\mathbf{E}}_{\mathbf{h}}$, the material derivatives satisfy the linearized identity
\[
\mathbf{P} + \dot{\mathcal{N}}_{\mathbf{h}} \mathbf{Q} = \nabla \times \mathbf{W}
\]
in the weak sense encoded by the mixed variational system. 

The pair $(\mathbf{W}, \mathbf{P})$ is the unique solution to the linearized problem
\[
\widehat{\mathcal{A}}^{\mathrm{p}} \bigl( (\mathbf{W}, \mathbf{P}), (\boldsymbol{\tau}, \boldsymbol{\nu}) \bigr) = \mathcal{L}_{\mathbf{h}}^{\mathrm{p}}(\boldsymbol{\tau}, \boldsymbol{\nu}) \qquad \forall (\boldsymbol{\tau}, \boldsymbol{\nu}) \in \mathbb{X}_{\mathrm{p}}(\Omega),
\]
where the $\mathbb{R}$-linear form $\widehat{\mathcal{A}}^{\mathrm{p}}$ is defined by
\begin{equation*}
\begin{aligned}
\widehat{\mathcal{A}}^{\mathrm{p}} \bigl( (\mathbf{W}, \mathbf{P}), (\boldsymbol{\tau}, \boldsymbol{\nu}) \bigr) &:= \int_{\Omega} \mathbf{P} \cdot \overline{\boldsymbol{\tau}} \,\mathrm{d}\boldsymbol{x} - \int_{\Omega} \mathbf{W} \cdot (\nabla \times \overline{\boldsymbol{\tau}}) \,\mathrm{d}\boldsymbol{x} \\
&\quad + \frac{\mathrm{i}}{k} \left\langle \Lambda^{-1} \gamma_t \mathbf{P}, \gamma_T \boldsymbol{\tau} \right\rangle_{\Gamma_R} + \int_{\Gamma} \mathbf{g}_{\boldsymbol{z}}(\cdot, \gamma_T \mathbf{E}; \gamma_T \mathbf{W}) \cdot \gamma_T \overline{\boldsymbol{\tau}} \,\mathrm{d}s \\
&\quad + \int_{\Omega} (\nabla \times \mathbf{P}) \cdot \overline{\boldsymbol{\nu}} \,\mathrm{d}\boldsymbol{x} - k^2 \int_{\Omega} \mathbf{W} \cdot \overline{\boldsymbol{\nu}} \,\mathrm{d}\boldsymbol{x},
\end{aligned}
\end{equation*}
and the right-hand side functional $\mathcal{L}_{\mathbf{h}}^{\mathrm{p}}$ is given by
\begin{equation*}
\begin{aligned}
\mathcal{L}_{\mathbf{h}}^{\mathrm{p}}(\boldsymbol{\tau}, \boldsymbol{\nu}) &:= -\int_{\Omega} \mathbf{Q}^{\top} \dot{\mathcal{N}}_{\mathbf{h}} \overline{\boldsymbol{\tau}} \,\mathrm{d}\boldsymbol{x} + k^2 \int_{\Omega} \mathbf{E}^{\top} \dot{\mathcal{N}}_{\mathbf{h}} \overline{\boldsymbol{\nu}} \,\mathrm{d}\boldsymbol{x} \\
&\quad - \int_{\Gamma} \left[ \nabla_{\boldsymbol{x}} \mathbf{g}(\cdot, \gamma_T \mathbf{E}) \mathbf{h} + \mathbf{g}_{\boldsymbol{z}}(\cdot, \gamma_T \mathbf{E}; \dot{\gamma}_{T,\mathbf{h}} \mathbf{E}) \right] \cdot \gamma_T \overline{\boldsymbol{\tau}} \,\mathrm{d}s \\
&\quad - \int_{\Gamma} \mathbf{g}(\cdot, \gamma_T \mathbf{E}) \cdot \dot{\gamma}_{T,\mathbf{h}} \overline{\boldsymbol{\tau}} \,\mathrm{d}s - \int_{\Gamma} \mathbf{g}(\cdot, \gamma_T \mathbf{E}) \cdot \gamma_T \overline{\boldsymbol{\tau}} \, \dot{\omega}_{\mathbf{h}} \,\mathrm{d}s.
\end{aligned}
\end{equation*}
\end{theo}

\begin{proof}
Define the increments for the electric and auxiliary fields as
\[
\Delta_{\mathbf{h}}^E := \hat{\mathbf{E}}_{\mathbf{h}} - \mathbf{E}, \qquad \Delta_{\mathbf{h}}^Q := \hat{\mathbf{Q}}_{\mathbf{h}} - \mathbf{Q}.
\]
According to Theorem~\ref{thm:continuity_p_mixed}, these increments satisfy the stability estimate $\|\Delta_{\mathbf{h}}^E\|_{\mathbf{X}(\Omega)} + \|\Delta_{\mathbf{h}}^Q\|_{H(\operatorname{curl};\Omega)} = O(\|\mathbf{h}\|_{C^1})$. Since the source functional $\mathcal{F}^{\mathrm{p}}$ is invariant under the perturbation, the reference and pulled-back perturbed formulations satisfy the consistency identity
\begin{equation}\label{eq:NPEC_consistency}
\widetilde{\mathcal{A}}^{\mathrm{p}}_{\mathbf{h}}((\hat{\mathbf{E}}_{\mathbf{h}}, \hat{\mathbf{Q}}_{\mathbf{h}}), (\boldsymbol{\tau}, \boldsymbol{\nu})) - \mathcal{A}^{\mathrm{p}}((\mathbf{E}, \mathbf{Q}), (\boldsymbol{\tau}, \boldsymbol{\nu})) = 0 \qquad \forall (\boldsymbol{\tau}, \boldsymbol{\nu}) \in \mathbb{X}_{\mathrm{p}}(\Omega).
\end{equation}

By invoking the first-order expansions of the geometric coefficients from Lemma~\ref{le:positive}, namely $\mathcal{N}_{\mathbf{h}} = \mathrm{Id} + \dot{\mathcal{N}}_{\mathbf{h}} + o(\|\mathbf{h}\|_{C^1})$ in $L^\infty(\Omega)$ and $\omega_{\mathbf{h}} = 1 + \dot{\omega}_{\mathbf{h}} + o(\|\mathbf{h}\|_{C^1})$ in $L^\infty(\Gamma)$, together with the trace expansion $\gamma_{T,\mathbf{h}}\mathbf{U} = \gamma_T\mathbf{U} + \dot{\gamma}_{T,\mathbf{h}}\mathbf{U} + o(\|\mathbf{h}\|_{C^1})$ in $L^2(\Gamma)$, we expand the nonlinear boundary term as
\[
\begin{aligned}
\mathbf{g}(\varphi_{\mathbf{h}}(\cdot), \gamma_{T,\mathbf{h}}\hat{\mathbf{E}}_{\mathbf{h}}) &= \mathbf{g}(\cdot, \gamma_T\mathbf{E}) + \mathbf{g}_{\boldsymbol{z}}(\cdot, \gamma_T\mathbf{E}; \gamma_T\Delta_{\mathbf{h}}^E) + \nabla_{\boldsymbol{x}}\mathbf{g}(\cdot, \gamma_T\mathbf{E})\mathbf{h} \\
&\quad + \mathbf{g}_{\boldsymbol{z}}(\cdot, \gamma_T\mathbf{E}; \dot{\gamma}_{T,\mathbf{h}}\mathbf{E}) + o(\|\mathbf{h}\|_{C^1}) \quad \text{in } L^2(\Gamma).
\end{aligned}
\]

Let $(\mathbf{R}_{\mathbf{h}}^E, \mathbf{R}_{\mathbf{h}}^Q) := (\Delta_{\mathbf{h}}^E - \mathbf{W}, \Delta_{\mathbf{h}}^Q - \mathbf{P})$ denote the remainder pair. Testing this pair against the linearized form $\widehat{\mathcal{A}}^{\mathrm{p}}$ and utilizing the identity \eqref{eq:NPEC_consistency}, we obtain the error equation
\[
\begin{aligned}
&\widehat{\mathcal{A}}^{\mathrm{p}}((\mathbf{R}_{\mathbf{h}}^E, \mathbf{R}_{\mathbf{h}}^Q), (\boldsymbol{\tau}, \boldsymbol{\nu})) \\
&\quad = -\int_\Omega \mathbf{Q}^\top (\mathcal{N}_{\mathbf{h}} - \mathrm{Id} - \dot{\mathcal{N}}_{\mathbf{h}}) \overline{\boldsymbol{\tau}} \,\mathrm{d}\boldsymbol{x} - \int_\Omega (\Delta_{\mathbf{h}}^Q)^\top (\mathcal{N}_{\mathbf{h}} - \mathrm{Id}) \overline{\boldsymbol{\tau}} \,\mathrm{d}\boldsymbol{x} \\
&\quad \quad + k^2 \int_\Omega \mathbf{E}^\top (\mathcal{N}_{\mathbf{h}} - \mathrm{Id} - \dot{\mathcal{N}}_{\mathbf{h}}) \overline{\boldsymbol{\nu}} \,\mathrm{d}\boldsymbol{x} + k^2 \int_\Omega (\Delta_{\mathbf{h}}^E)^\top (\mathcal{N}_{\mathbf{h}} - \mathrm{Id}) \overline{\boldsymbol{\nu}} \,\mathrm{d}\boldsymbol{x} + \mathcal{N}_{\mathbf{h}}^g(\boldsymbol{\tau}),
\end{aligned}
\]
where \(\mathcal N_{\mathbf h}^{g}(\boldsymbol{\tau})\) denotes the nonlinear
boundary remainder, namely the part left after subtracting from the perturbed
boundary integral its reference value and all first-order terms generated by
the linearized form \(\widehat{\mathcal A}^{\mathrm p}\) and the functional
\(\mathcal L_{\mathbf h}^{\mathrm p}\). By the \(L^2\)-Nemytskii expansion in
Assumption~\eqref{ass:enhanced_g}, together with the first-order expansions of
\(\gamma_{T,\mathbf h}\) and \(\omega_{\mathbf h}\), it follows that
\[
\|\mathcal N_{\mathbf h}^{g}\|_{\mathbb X_{\mathrm p}(\Omega)^*}
=
o(\|\mathbf h\|_{C^1}).
\]
Similarly, for the volume terms, the geometric expansion $\mathcal{N}_{\mathbf{h}} - \mathrm{Id} - \dot{\mathcal{N}}_{\mathbf{h}} = o(\|\mathbf{h}\|_{C^1})$ and the stability estimates $\Delta_{\mathbf{h}}^E, \Delta_{\mathbf{h}}^Q = O(\|\mathbf{h}\|_{C^1})$ imply that the right-hand side of the error equation is $o(\|\mathbf{h}\|_{C^1})$ in the dual space $\mathbb{X}_{\mathrm{p}}(\Omega)^*$. We thus conclude
\[
\left\| \widehat{\mathcal{A}}^{\mathrm{p}}((\mathbf{R}_{\mathbf{h}}^E, \mathbf{R}_{\mathbf{h}}^Q), \cdot) \right\|_{\mathbb{X}_{\mathrm{p}}(\Omega)^*} = o(\|\mathbf{h}\|_{C^1}).
\]
By the assumed well-posedness of the $\mathbb{R}$-linearized mixed problem, the induced operator is a bounded linear isomorphism, yielding
\[
\|\mathbf{R}_{\mathbf{h}}^E\|_{\mathbf{X}(\Omega)} + \|\mathbf{R}_{\mathbf{h}}^Q\|_{H(\operatorname{curl}; \Omega)} \le C \left\| \widehat{\mathcal{A}}^{\mathrm{p}}((\mathbf{R}_{\mathbf{h}}^E, \mathbf{R}_{\mathbf{h}}^Q), \cdot) \right\|_{\mathbb{X}_{\mathrm{p}}(\Omega)^*} = o(\|\mathbf{h}\|_{C^1}).
\]
we finally obtain
\[
\lim_{\|\mathbf{h}\|_{C^1} \to 0} \frac{\|\hat{\mathbf{E}}_{\mathbf{h}} - \mathbf{E} - \mathbf{W}\|_{\mathbf{X}(\Omega)} + \|\hat{\mathbf{Q}}_{\mathbf{h}} - \mathbf{Q} - \mathbf{P}\|_{H(\operatorname{curl};\Omega)}}{\|\mathbf{h}\|_{C^1}} = 0,
\]
which completes the proof.
\end{proof}

\subsection{The nonlinear transmission condition}

Finally, we address the nonlinear transmission condition. Since the corresponding variational framework was established in Section~\ref{sec:2}, we continue to work in the trace-regular space $\mathbf{X}(\Omega)$.

In the transmission setting, the material parameters \(\mu_r\) and
\(\varepsilon_r\) are assumed to be piecewise constant on the two sides of the
interface \(\Gamma\). The continuity of the tangential trace,
\([\gamma_t\mathbf E]=0\), is encoded by the global
\(H(\operatorname{curl};\Omega)\)-regularity of the field. In the present
trace-regular framework, we work in the subspace \(\mathbf X(\Omega)\), so that
the common tangential component \(\gamma_T\mathbf E\) belongs to
\(L_t^2(\Gamma)\). The nonlinear jump condition across the interface is then
incorporated through the boundary term involving
\(\mathbf g(\cdot,\gamma_T\mathbf E)\).

The weak formulation consists in finding $\mathbf{E} \in \mathbf{X}(\Omega)$ such that
\begin{equation}
\mathcal{A}^{\mathrm{tr}}(\mathbf{E}, \mathbf{V}) = \mathcal{F}^{\mathrm{tr}}(\mathbf{V}) \qquad \forall \mathbf{V} \in \mathbf{X}(\Omega),
\end{equation}
where the nonlinear form $\mathcal{A}^{\mathrm{tr}}$ is defined by
\begin{equation}
\begin{aligned}
\mathcal{A}^{\mathrm{tr}}(\mathbf{E}, \mathbf{V}) &:= \int_{\Omega} \mu_r^{-1} (\nabla \times \mathbf{E}) \cdot (\nabla \times \overline{\mathbf{V}}) \,\mathrm{d}\boldsymbol{x} - k^2 \int_{\Omega} \varepsilon_r \mathbf{E} \cdot \overline{\mathbf{V}} \,\mathrm{d}\boldsymbol{x} \\
&\quad - \int_{\Gamma} \mathbf{g}(\cdot, \gamma_T \mathbf{E}) \cdot \gamma_T \overline{\mathbf{V}} \,\mathrm{d}s + \mathrm{i}k \langle \Lambda \gamma_t \mathbf{E}, \gamma_T \mathbf{V} \rangle_{\Gamma_R}.
\end{aligned}
\end{equation}
For a perturbed interface $\Gamma_{\mathbf{h}}$, let $\hat{\mathbf{E}}_{\mathbf{h}} = J_\varphi^\top(\mathbf{E}_{\mathbf{h}} \circ \varphi_{\mathbf{h}})$ denote the covariant Piola pullback. Pulling the perturbed transmission problem back to the reference domain $\Omega$ yields the variational problem: find $\hat{\mathbf{E}}_{\mathbf{h}} \in \mathbf{X}(\Omega)$ such that
\begin{equation}
\widetilde{\mathcal{A}}^{\mathrm{tr}}_{\mathbf{h}}(\hat{\mathbf{E}}_{\mathbf{h}}, \mathbf{V}) = \mathcal{F}^{\mathrm{tr}}(\mathbf{V}) \qquad \forall \mathbf{V} \in \mathbf{X}(\Omega),
\end{equation}
with the transformed nonlinear form given by
\begin{equation}
\begin{aligned}
\widetilde{\mathcal{A}}^{\mathrm{tr}}_{\mathbf{h}}(\hat{\mathbf{E}}_{\mathbf{h}}, \mathbf{V}) &:= \int_{\Omega} \mu_r^{-1} (\nabla \times \hat{\mathbf{E}}_{\mathbf{h}})^\top \mathcal{M}_{\mathbf{h}} (\nabla \times \overline{\mathbf{V}}) \,\mathrm{d}\boldsymbol{x} - k^2 \int_{\Omega} \varepsilon_r \hat{\mathbf{E}}_{\mathbf{h}}^{\top} \mathcal{N}_{\mathbf{h}} \overline{\mathbf{V}} \,\mathrm{d}\boldsymbol{x} \\
&\quad - \int_{\Gamma} \mathbf{g}(\varphi_{\mathbf{h}}(\cdot), \gamma_{T,\mathbf{h}} \hat{\mathbf{E}}_{\mathbf{h}}) \cdot \gamma_{T,\mathbf{h}} \overline{\mathbf{V}} \, \omega_{\mathbf{h}} \,\mathrm{d}s + \mathrm{i}k \langle \Lambda \gamma_t \hat{\mathbf{E}}_{\mathbf{h}}, \gamma_T \mathbf{V} \rangle_{\Gamma_R}.
\end{aligned}
\end{equation}
Here, the material parameters $\mu_r$ and $\varepsilon_r$ are identified with their piecewise constant values on the reference domain. 

The transmission problem exhibits a variational structure essentially identical to that of the nonlinear impedance case, with the minor modification that the volume integrals are weighted by the bounded piecewise coefficients $\mu_r^{-1}$ and $\varepsilon_r$. Consequently, the linear perturbation estimates follow directly from Lemma~\ref{le:positive}, while the nonlinear interface terms are analyzed using the previously established transported-trace estimates and $L^2$-Lipschitz arguments.

\begin{prop}\label{prop:continuity_ntc}
Suppose that the nonlinear transmission problem is uniquely solvable for all admissible perturbations $\mathbf{h}$ with $\|\mathbf{h}\|_{C^1}$ sufficiently small, and that the associated linear transmission problem is well-posed in $\mathbf{X}(\Omega)$. Assume further that the nonlinear interface response $\mathbf{g}$ satisfies an $L^2$-Lipschitz condition analogous to the impedance case. If the Lipschitz constant $L_g$ is sufficiently small, then the following stability estimate holds:
\[
\|\hat{\mathbf{E}}_{\mathbf{h}} - \mathbf{E}\|_{\mathbf{X}(\Omega)} \le C \|\mathbf{h}\|_{C^1}.
\]
In particular, $\hat{\mathbf{E}}_{\mathbf{h}}$ converges to $\mathbf{E}$ in $H(\operatorname{curl}; \Omega)$ as $\|\mathbf{h}\|_{C^1} \to 0$.
\end{prop}

\begin{prop}\label{prop:material_ntc}
Suppose that the hypotheses of Proposition~\ref{prop:continuity_ntc} are satisfied. Assume further that the associated $\mathbb{R}$-linearized transmission problem is well-posed in $\mathbf{X}(\Omega)$ and that Assumption~\ref{ass:enhanced_g} holds in a tubular neighborhood of $\Gamma$. Then the pulled-back transmission solution is Fr\'{e}chet differentiable with respect to $\mathbf{h}$ at $\mathbf{h} = \mathbf{0}$. Specifically, there exists a unique material derivative $\mathbf{W} \in \mathbf{X}(\Omega)$, depending linearly and continuously on $\mathbf{h}$, such that
\[
\lim_{\|\mathbf{h}\|_{C^1} \to 0} \frac{\|\hat{\mathbf{E}}_{\mathbf{h}} - \mathbf{E} - \mathbf{W}\|_{\mathbf{X}(\Omega)}}{\|\mathbf{h}\|_{C^1}} = 0.
\]
The material derivative $\mathbf{W}$ is the unique solution to the variational problem
\[
\widehat{\mathcal{A}}^{\mathrm{tr}}(\mathbf{W}, \mathbf{V}) = \mathcal{L}_{\mathbf{h}}^{\mathrm{tr}}(\mathbf{V}) \qquad \forall \mathbf{V} \in \mathbf{X}(\Omega),
\]
where the $\mathbb{R}$-linear form $\widehat{\mathcal{A}}^{\mathrm{tr}}$ is defined by
\begin{equation}
\begin{aligned}
\widehat{\mathcal{A}}^{\mathrm{tr}}(\mathbf{W}, \mathbf{V}) &:= \int_{\Omega} \mu_r^{-1} (\nabla \times \mathbf{W}) \cdot (\nabla \times \overline{\mathbf{V}}) \,\mathrm{d}\boldsymbol{x} - k^2 \int_{\Omega} \varepsilon_r \mathbf{W} \cdot \overline{\mathbf{V}} \,\mathrm{d}\boldsymbol{x} \\
&\quad - \int_{\Gamma} \mathbf{g}_{\boldsymbol{z}}(\cdot, \gamma_T \mathbf{E}; \gamma_T \mathbf{W}) \cdot \gamma_T \overline{\mathbf{V}} \,\mathrm{d}s + \mathrm{i}k \langle \Lambda \gamma_t \mathbf{W}, \gamma_T \mathbf{V} \rangle_{\Gamma_R},
\end{aligned}
\end{equation}
and the right-hand side functional $\mathcal{L}_{\mathbf{h}}^{\mathrm{tr}} \in \mathbf{X}(\Omega)^*$ is given by
\begin{equation}
\begin{aligned}
\mathcal{L}_{\mathbf{h}}^{\mathrm{tr}}(\mathbf{V}) &:= -\int_{\Omega} \mu_r^{-1} (\nabla \times \mathbf{E})^\top \dot{\mathcal{M}}_{\mathbf{h}} (\nabla \times \overline{\mathbf{V}}) \,\mathrm{d}\boldsymbol{x} + k^2 \int_{\Omega} \varepsilon_r \mathbf{E}^\top \dot{\mathcal{N}}_{\mathbf{h}} \overline{\mathbf{V}} \,\mathrm{d}\boldsymbol{x} \\
&\quad + \int_{\Gamma} \left[ \nabla_{\boldsymbol{x}} \mathbf{g}(\cdot, \gamma_T \mathbf{E}) \mathbf{h} + \mathbf{g}_{\boldsymbol{z}}(\cdot, \gamma_T \mathbf{E}; \dot{\gamma}_{T,\mathbf{h}} \mathbf{E}) \right] \cdot \gamma_T \overline{\mathbf{V}} \,\mathrm{d}s \\
&\quad + \int_{\Gamma} \mathbf{g}(\cdot, \gamma_T \mathbf{E}) \cdot \dot{\gamma}_{T,\mathbf{h}} \overline{\mathbf{V}} \,\mathrm{d}s + \int_{\Gamma} \mathbf{g}(\cdot, \gamma_T \mathbf{E}) \cdot \gamma_T \overline{\mathbf{V}} \, \dot{\omega}_{\mathbf{h}} \,\mathrm{d}s,
\end{aligned}
\end{equation}
with the geometric variations defined as
\[
\dot{\mathcal{M}}_{\mathbf{h}} = -(\operatorname{div} \mathbf{h})\mathrm{Id} + J_{\mathbf{h}} + J_{\mathbf{h}}^{\top} \quad \text{and} \quad \dot{\mathcal{N}}_{\mathbf{h}} = (\operatorname{div} \mathbf{h})\mathrm{Id} - J_{\mathbf{h}} - J_{\mathbf{h}}^{\top}.
\]
\end{prop}

The proofs of Propositions~\ref{prop:continuity_ntc} and \ref{prop:material_ntc} 
proceed identically to the impedance case. Indeed, the bounded piecewise material 
parameters $\mu_r^{-1}$ and $\varepsilon_r$ only affect the specific constants 
in the volume residuals, while the nonlinear interface contributions are treated 
via the previously established $L^2$-Lipschitz and Nemytskii expansion arguments.

\section{Characterization of the Shape Derivative}
\label{sec:4}

The material derivative characterizes the first-order variation of the pulled-back field on the fixed reference domain. The shape derivative is subsequently recovered by accounting for the convective and geometric contributions induced by the domain transformation. For Maxwell fields, where the pullback is defined via the covariant Piola transform to preserve the $H(\mathrm{curl})$-structure, the first-order expansion for sufficiently smooth fields is given by
\begin{equation}
\label{eq:piola_expansion}
\mathbf{W} = J_{\mathbf{h}}^{\top} \mathbf{E} + J_{\mathbf{E}}\mathbf{h} + \delta \mathbf{E},
\end{equation}
where $J_{\mathbf{E}}\mathbf{h} := (\mathbf{h} \cdot \nabla) \mathbf{E}$. Here, $\mathbf{W}$ denotes the material derivative associated with the Piola pullback, while $\delta \mathbf{E}$ represents the shape derivative.

The boundary characterizations derived in the sequel are initially formulated for smooth fields and subsequently interpreted in the sense of weak trace duality. These expressions exhibit the Hadamard structure: the final boundary data depend exclusively on the normal component $h_n := \mathbf{h} \cdot \mathbf{n}$ of the deformation field. Tangential perturbations, which correspond to local reparametrizations of the boundary, are shown to vanish upon surface integration by parts.

Throughout this section, we impose additional regularity assumptions to provide strong characterizations of the shape derivative. Specifically, we assume the boundary $\Gamma$ is of class $C^2$. Let $U_\Gamma$ denote a local neighborhood of $\Gamma$ in the domain $\Omega$. For the impenetrable scattering problems (NIBC and NPEC), we assume that the reference electric field is locally $H^2$-regular near the boundary, i.e., $\mathbf{E} \in H^2(U_\Gamma; \mathbb{C}^3)$.

For the transmission problem, we assume piecewise regularity on each side of the interface. Specifically, let $U_\Gamma^+$ and $U_\Gamma^-$ denote one-sided tubular neighborhoods of $\Gamma$ in the exterior and interior subdomains, respectively; we then assume that
\[
\mathbf{E}^+ \in H^2(U_\Gamma^+; \mathbb{C}^3) \quad \text{and} \quad \mathbf{E}^- \in H^2(U_\Gamma^-; \mathbb{C}^3).
\]
These regularity assumptions ensure that the normal derivatives, surface differential operators, and nonlinear linearizations appearing in the sequel are well-defined. In the impedance case, this local $H^2$-regularity is consistent with established results for Maxwell impedance problems on $C^2$ domains provided the effective impedance data is sufficiently regular; see, e.g., \cite{chen2025regularity}. For the NPEC and NTC cases, the required regularity is imposed \textit{a priori}.

Under the additional regularity assumptions imposed in this section, the weak
first-order trace variation introduced in Section~\ref{sec:3} admits the
following pointwise representation for sufficiently smooth vector fields:
\begin{equation}\label{eq:trace variation}
\dot{\gamma}_{T,\mathbf h}\mathbf U
=
\delta\mathbf n\times(\mathbf U\times\mathbf n)
+
\mathbf n\times(\mathbf U\times\delta\mathbf n)
-
\mathbf n\times(J_{\mathbf h}^{\top}\mathbf U\times\mathbf n).
\end{equation}

We now use this pointwise trace representation, together with the material
derivative equation derived in Section~\ref{sec:3}, to extract the strong
boundary condition satisfied by the shape derivative. We first consider the
nonlinear impedance boundary condition.
\begin{theo}\label{th:shape_derivative_nibc}
The material derivative $\mathbf{W}$ obtained in Theorem~\ref{th:symmetrisation1} is related to the shape derivative $\delta\mathbf{E}$ via the decomposition
\begin{equation}
\label{eq:material_derivative_decomposition}
\mathbf{W} = J_{\mathbf{h}}^{\top}\mathbf{E} + (\mathbf{h} \cdot \nabla) \mathbf{E} + \delta\mathbf{E}.
\end{equation}
The shape derivative $\delta\mathbf{E}$ is characterized as the unique weak solution to the $\mathbb{R}$-linear boundary value problem
\begin{equation}
\nabla \times (\nabla \times \delta\mathbf{E}) - k^2 \delta\mathbf{E} = 0 \quad \text{in } \Omega,
\end{equation}
subject to the following boundary condition on $\Gamma$:
\begin{equation}
\label{eq:shape1}
\begin{aligned}
&\mathbf{n} \times (\nabla \times \delta\mathbf{E}) + \mathrm{i}k\lambda\,\gamma_T \delta\mathbf{E} - \mathbf{g}_{\boldsymbol{z}}\bigl(\boldsymbol{x}, \gamma_T \mathbf{E}; \gamma_T \delta\mathbf{E}\bigr) \\
&\quad = -\operatorname{Curl}_\Gamma \bigl( h_n \operatorname{curl}_\Gamma(\gamma_T \mathbf{E}) \bigr) - h_n k^2 \gamma_T \mathbf{E} \\
&\qquad - \mathrm{i}k\lambda \Bigl[ (\mathbf{E} \cdot \mathbf{n}) \nabla_\Gamma h_n + h_n \Bigl( \frac{\partial}{\partial \mathbf{n}}(\gamma_T \mathbf{E}) - \mathcal{S}(\gamma_T \mathbf{E}) \Bigr) \Bigr] \\
&\qquad - \mathrm{i}k\lambda h_n \mathfrak{H} (\gamma_T \mathbf{E}) + h_n \Bigl( \frac{\partial}{\partial \mathbf{n}} \mathbf{g}(\boldsymbol{x}, \gamma_T \mathbf{E}) - \mathcal{S}\bigl( \mathbf{g}(\boldsymbol{x}, \gamma_T \mathbf{E}) \bigr) \Bigr) \\
&\qquad + h_n \mathfrak{H} \mathbf{g}(\boldsymbol{x}, \gamma_T \mathbf{E}) + \mathbf{g}_{\boldsymbol{z}} \bigl( \boldsymbol{x}, \gamma_T \mathbf{E}; (\mathbf{E} \cdot \mathbf{n}) \nabla_\Gamma h_n \bigr).
\end{aligned}
\end{equation}
Here, $\mathcal{S}$ denotes the shape operator (or Weingarten map), defined by $\mathcal{S} \boldsymbol{\tau} := (J_{\mathbf{n}}) \boldsymbol{\tau}$ for any tangential vector field $\boldsymbol{\tau}$ on $\Gamma$. Under this convention, the additive  curvature is given by $\mathfrak{H} := \operatorname{tr} \mathcal{S} = \operatorname{div}_\Gamma \mathbf{n} = 2\kappa$. 

On the artificial boundary $\Gamma_R$, the field $\delta\mathbf{E}$ satisfies the radiation condition
\begin{equation}
\label{eq:silver_muller_radiation}
\mathrm{i}k \Lambda(\gamma_t \delta\mathbf{E}) = \gamma_t(\nabla \times \delta\mathbf{E}) \quad \text{on } \Gamma_R.
\end{equation}
Since the incident field $\mathbf{E}^i$ is domain-independent, its shape derivative vanishes, and thus $\delta\mathbf{E} = \delta(\mathbf{E} - \mathbf{E}^i)$ represents the shape derivative of the scattered field. This condition ensures that $\delta\mathbf{E}$ uniquely extends to a radiating weak solution in the exterior domain $\mathbb{R}^3 \setminus \overline{D}$.
\end{theo}

\begin{proof}
Under the regularity assumptions introduced at the beginning of this section, the following differential operations are well-defined. We define the shape derivative $\delta\mathbf{E}$ in $\Omega$ by the relation
\[
\delta\mathbf{E} := \mathbf{W} - J_{\mathbf{h}}^{\top}\mathbf{E} - (\mathbf{h} \cdot \nabla) \mathbf{E}.
\]

Using the linearized variational formulation \eqref{eq:linearized_weak_form}
and the material derivative equation \eqref{eq:material_derivative}, and
substituting the pointwise trace variation formula \eqref{eq:trace variation},
we obtain the following identity for the shape derivative:
\begin{equation}
\label{eq:shape_derivative}
\begin{aligned}
\widehat{\mathcal{A}}(\delta \mathbf{E}, \mathbf{V}) &= \widehat{\mathcal{A}}(\mathbf{W}, \mathbf{V}) - \widehat{\mathcal{A}}(J_{\mathbf{h}}^{\top}\mathbf{E} + (\mathbf{h} \cdot \nabla) \mathbf{E}, \mathbf{V}) \\
&= \int_\Omega \Bigl[ (\nabla \times \mathbf{E})^{\top} \bigl( (\operatorname{div} \mathbf{h})\mathrm{Id} - J_{\mathbf{h}} - J_{\mathbf{h}}^{\top} \bigr) \nabla \times \overline{\mathbf{V}} \\
&\qquad \quad + k^2 \mathbf{E}^{\top} \bigl( (\operatorname{div} \mathbf{h})\mathrm{Id} - J_{\mathbf{h}} - J_{\mathbf{h}}^{\top} \bigr) \overline{\mathbf{V}} \Bigr] \, \mathrm{d}\boldsymbol{x} \\
&\quad + \mathrm{i}k\lambda \int_\Gamma (\gamma_T\mathbf{E})^{\top} \bigl( -\operatorname{div}_\Gamma \mathbf{h}_T - 2\kappa h_n \bigr) \gamma_T\overline{\mathbf{V}} \, \mathrm{d}s \\
&\quad + \mathrm{i}k\lambda \int_\Gamma \Bigl( -\delta\mathbf{n} \times (\mathbf{E} \times \mathbf{n}) - \mathbf{n} \times (\mathbf{E} \times \delta \mathbf{n}) + \mathbf{n} \times \bigl( (J_{\mathbf{h}}^{\top}\mathbf{E}) \times \mathbf{n} \bigr) \Bigr)^{\top} \gamma_T\overline{\mathbf{V}} \, \mathrm{d}s \\
&\quad + \mathrm{i}k\lambda \int_\Gamma \Bigl[ (J_{\mathbf{h}}\gamma_T\mathbf{E})^{\top} \gamma_T\overline{\mathbf{V}} - (\gamma_T\mathbf{E})^{\top} \bigl( \mathbf{n} \times (\gamma_T\overline{\mathbf{V}} \times \delta \mathbf{n}) \bigr) \Bigr] \, \mathrm{d}s \\
&\quad - \int_\Gamma \Bigl[ \mathbf{g}(\cdot, \gamma_T \mathbf{E})^{\top} \bigl( -\operatorname{div}_\Gamma \mathbf{h}_T - 2\kappa h_n \bigr) - \nabla_{\boldsymbol{x}}\mathbf{g}(\cdot, \gamma_T \mathbf{E})\mathbf{h} \\
&\qquad \quad - \mathbf{g}_{\boldsymbol{z}} \Bigl( \cdot, \gamma_T \mathbf{E}; \delta\mathbf{n} \times (\mathbf{E} \times \mathbf{n}) + \mathbf{n} \times (\mathbf{E} \times \delta \mathbf{n}) - \mathbf{n} \times \bigl( (J_{\mathbf{h}}^{\top}\mathbf{E}) \times \mathbf{n} \bigr) \Bigr) \Bigr]^{\top} \gamma_T\overline{\mathbf{V}} \, \mathrm{d}s \\
&\quad - \int_\Gamma \Bigl[ (J_{\mathbf{h}}\mathbf{g}(\cdot, \gamma_T \mathbf{E}))^{\top} \gamma_T\overline{\mathbf{V}} - \mathbf{g}(\cdot, \gamma_T \mathbf{E})^{\top} \bigl( \mathbf{n} \times (\gamma_T\overline{\mathbf{V}} \times \delta \mathbf{n}) \bigr) \Bigr] \, \mathrm{d}s \\
&\quad - \int_\Omega \Bigl[ (\nabla \times (J_{\mathbf{h}}^{\top}\mathbf{E} + (\mathbf{h} \cdot \nabla) \mathbf{E}))^{\top} \nabla \times \overline{\mathbf{V}} - k^2 (J_{\mathbf{h}}^{\top}\mathbf{E} + (\mathbf{h} \cdot \nabla) \mathbf{E})^{\top} \overline{\mathbf{V}} \Bigr] \, \mathrm{d}\boldsymbol{x} \\
&\quad - \mathrm{i}k\lambda \int_\Gamma \bigl[ \mathbf{n} \times ((J_{\mathbf{h}}^{\top}\mathbf{E} + (\mathbf{h} \cdot \nabla) \mathbf{E}) \times \mathbf{n}) \bigr] \cdot \gamma_T\overline{\mathbf{V}} \, \mathrm{d}s \\
&\quad + \int_\Gamma \mathbf{g}_{\boldsymbol{z}} \bigl( \cdot, \gamma_T \mathbf{E}; \gamma_T(J_{\mathbf{h}}^{\top}\mathbf{E} + (\mathbf{h} \cdot \nabla) \mathbf{E}) \bigr)^{\top} \gamma_T\overline{\mathbf{V}} \, \mathrm{d}s \\
&\quad - \mathrm{i}k \langle \Lambda \gamma_t (J_{\mathbf{h}}^{\top}\mathbf{E} + (\mathbf{h} \cdot \nabla) \mathbf{E}), \gamma_T \mathbf{V} \rangle_{\Gamma_R}.
\end{aligned}
\end{equation}

Since $\operatorname{supp} \mathbf{h} \Subset B_R$, both the deformation field $\mathbf{h}$ and its Jacobian $J_{\mathbf{h}}$ vanish identically in a neighborhood of the artificial boundary $\Gamma_R$. It follows that
\[
\gamma_t(J_{\mathbf{h}}^{\top}\mathbf{E} + (\mathbf{h} \cdot \nabla) \mathbf{E}) = 0 \qquad \text{on } \Gamma_R,
\]
which implies the vanishing of the associated Calder\'on boundary term in \eqref{eq:shape_derivative}.

We partition the variational residual into a volume contribution $\mathscr V$, an impedance boundary term $I_1$, and a nonlinear boundary term $I_2$, such that
\[
\widehat{\mathcal{A}}(\delta\mathbf{E}, \mathbf{V}) = \mathscr V + I_1 + I_2.
\]
The boundary contributions are characterized by the following integrals:
\begin{align*}
I_1 &:= \mathrm{i}k\lambda \int_\Gamma \Bigl[ \bigl( -\operatorname{div}_\Gamma \mathbf{h}_T - 2\kappa h_n \bigr) \gamma_T \mathbf{E} - \delta\mathbf{n} \times (\mathbf{E} \times \mathbf{n}) \\
&\qquad \qquad - \mathbf{n} \times (\mathbf{E} \times \delta\mathbf{n}) - \mathbf{n} \times \bigl( ((\mathbf{h} \cdot \nabla)\mathbf{E}) \times \mathbf{n} \bigr) \Bigr]^{\top} \gamma_T \overline{\mathbf{V}} \,\mathrm{d}s \\
&\quad + \mathrm{i}k\lambda \int_\Gamma \Bigl[ (J_{\mathbf{h}}\gamma_T \mathbf{E})^{\top} \gamma_T\overline{\mathbf{V}} - (\gamma_T \mathbf{E})^{\top} \bigl( \mathbf{n} \times (\gamma_T\overline{\mathbf{V}} \times \delta\mathbf{n}) \bigr) \Bigr] \,\mathrm{d}s, \\[10pt]
I_2 &:= -\int_\Gamma \Bigl[ \bigl( -\operatorname{div}_\Gamma \mathbf{h}_T - 2\kappa h_n \bigr) \mathbf{g}(\boldsymbol{x}, \gamma_T \mathbf{E}) - \nabla_{\boldsymbol{x}}\mathbf{g}(\cdot, \gamma_T \mathbf{E}) \mathbf{h} \\
&\qquad \qquad - \mathbf{g}_{\boldsymbol{z}} \Bigl( \boldsymbol{x}, \gamma_T \mathbf{E}; \, \delta\mathbf{n} \times (\mathbf{E} \times \mathbf{n}) + \mathbf{n} \times (\mathbf{E} \times \delta \mathbf{n}) - \mathbf{n} \times \bigl( ((\mathbf{h} \cdot \nabla)\mathbf{E}) \times \mathbf{n} \bigr) \Bigr) \Bigr]^{\top} \gamma_T \overline{\mathbf{V}} \,\mathrm{d}s \\
&\quad - \int_\Gamma \Bigl[ (J_{\mathbf{h}}\mathbf{g}(\cdot, \gamma_T \mathbf{E}))^{\top} \gamma_T\overline{\mathbf{V}} - \mathbf{g}(\cdot, \gamma_T \mathbf{E})^{\top} \bigl( \mathbf{n} \times (\gamma_T\overline{\mathbf{V}} \times \delta \mathbf{n}) \bigr) \Bigr] \,\mathrm{d}s.
\end{align*}

Next, we simplify the volume contribution $\mathscr V$. Let $\mathbf{C} := \nabla \times \mathbf{E}$. By recasting the skew-symmetric component of the displacement gradient as the cross product $(\nabla \times \mathbf{E}) \times \mathbf{h}$, we obtain the following identity for the curl of the convective perturbation:
\begin{equation}
\label{eq:curl_material_identity}
\begin{aligned}
\nabla \times \bigl( J_{\mathbf{h}}^{\top}\mathbf{E} + (\mathbf{h} \cdot \nabla)\mathbf{E} \bigr) 
&= \nabla \times \bigl( (J_{\mathbf{E}} - J_{\mathbf{E}}^{\top})\mathbf{h} + \nabla(\mathbf{h} \cdot \mathbf{E}) \bigr) \\
&= \nabla \times \bigl( (\nabla \times \mathbf{E}) \times \mathbf{h} \bigr) \\
&= (\operatorname{div} \mathbf{h})(\nabla \times \mathbf{E}) + (\mathbf{h} \cdot \nabla)(\nabla \times \mathbf{E}) - J_{\mathbf{h}}(\nabla \times \mathbf{E}).
\end{aligned}
\end{equation}
In deriving \eqref{eq:curl_material_identity}, we have invoked the identity $\nabla(\mathbf{h} \cdot \mathbf{E}) = J_{\mathbf{h}}^{\top}\mathbf{E} + J_{\mathbf{E}}^{\top}\mathbf{h}$ and the vector identity for the curl of a cross product,
\[
\nabla \times (\mathbf{C} \times \mathbf{h}) = (\operatorname{div} \mathbf{h})\mathbf{C} + (\mathbf{h} \cdot \nabla)\mathbf{C} - ( \mathbf{C} \cdot \nabla ) \mathbf{h},
\]
noting that $\operatorname{div} \mathbf{C} = \operatorname{div}(\nabla \times \mathbf{E}) = 0$. Substituting \eqref{eq:curl_material_identity} into the volume part of \eqref{eq:shape_derivative}, the term $\mathscr V$ becomes
\begin{align*}
\mathscr V &= -\int_\Omega \bigl( (\mathbf{h} \cdot \nabla)(\nabla \times \mathbf{E}) + J_{\mathbf{h}}^{\top} (\nabla \times \mathbf{E}) \bigr)^{\top} \nabla \times \overline{\mathbf{V}} \,\mathrm{d}\boldsymbol{x} \\
&\quad + k^2 \int_\Omega \bigl( (\mathbf{h} \cdot \nabla)\mathbf{E} + (\operatorname{div} \mathbf{h})\mathbf{E} - J_{\mathbf{h}}\mathbf{E} \bigr)^{\top} \overline{\mathbf{V}} \,\mathrm{d}\boldsymbol{x} \\
&= -\int_\Omega \Bigl[ J_{\nabla \times \mathbf{E}}\mathbf{h} + J_{\mathbf{h}}^{\top} (\nabla \times \mathbf{E}) \Bigr]^{\top} \nabla \times \overline{\mathbf{V}} \,\mathrm{d}\boldsymbol{x} \\
&\quad + k^2 \int_\Omega \bigl( J_{\mathbf{E}}\mathbf{h} + (\operatorname{div} \mathbf{h})\mathbf{E} - J_{\mathbf{h}}\mathbf{E} \bigr)^{\top} \overline{\mathbf{V}} \,\mathrm{d}\boldsymbol{x}.
\end{align*}

By the governing Maxwell equation $\nabla \times (\nabla \times \mathbf{E}) = k^2 \mathbf{E}$, the skew-symmetric component of the Jacobian $J_{\nabla \times \mathbf{E}}$ satisfies the identity
\[
(J_{\nabla \times \mathbf{E}} - J_{\nabla \times \mathbf{E}}^{\top}) \mathbf{h} = \bigl( \nabla \times (\nabla \times \mathbf{E}) \bigr) \times \mathbf{h} = k^2 \mathbf{E} \times \mathbf{h}.
\]
Furthermore, characterizing the gradient of the scalar product as 
\[
\nabla (\mathbf{h} \cdot \nabla \times \mathbf{E}) = J_{\mathbf{h}}^{\top} \nabla \times \mathbf{E} + J_{\nabla \times \mathbf{E}}^{\top} \mathbf{h},
\]
and substituting these relations into the volume contribution $\mathscr V$, we obtain
\begin{equation*}
\begin{aligned}
\mathscr V &= -\int_\Omega \nabla (\mathbf{h} \cdot \nabla \times \mathbf{E}) \cdot \nabla \times \overline{\mathbf{V}} \,\mathrm{d}\boldsymbol{x} \\
&\quad + k^2 \int_\Omega \nabla \times (\mathbf{E} \times \mathbf{h}) \cdot \overline{\mathbf{V}} \,\mathrm{d}\boldsymbol{x} - k^2 \int_\Omega (\mathbf{E} \times \mathbf{h}) \cdot \nabla \times \overline{\mathbf{V}} \,\mathrm{d}\boldsymbol{x}.
\end{aligned}
\end{equation*}
Next, we invoke the vector identity
\[
\nabla \times (\mathbf{E} \times \mathbf{h}) \cdot \overline{\mathbf{V}} = \nabla \cdot \bigl( (\mathbf{E} \times \mathbf{h}) \times \overline{\mathbf{V}} \bigr) + (\mathbf{E} \times \mathbf{h}) \cdot \nabla \times \overline{\mathbf{V}}
\]
and observe that $\nabla \cdot (\nabla \times \overline{\mathbf{V}}) = 0$, which implies
\[
\nabla (\mathbf{h} \cdot \nabla \times \mathbf{E}) \cdot \nabla \times \overline{\mathbf{V}} = \nabla \cdot \bigl( (\mathbf{h} \cdot \nabla \times \mathbf{E}) \nabla \times \overline{\mathbf{V}} \bigr).
\]
Applying the divergence theorem and recalling that the perturbation $\mathbf{h}$ vanishes in a neighborhood of the artificial boundary $\Gamma_R$, the volume contribution $\mathscr V$ reduces to the following boundary integral over $\Gamma$:
\begin{equation}
\label{eq:A_boundary_initial}
\mathscr V = \int_\Gamma \left[ (\mathbf{h} \cdot \nabla \times \mathbf{E}) (\mathbf{n} \cdot \nabla \times \overline{\mathbf{V}}) - k^2 \bigl( (\mathbf{E} \times \mathbf{h}) \times \overline{\mathbf{V}} \bigr) \cdot \mathbf{n} \right] \mathrm{d}s,
\end{equation}
where $\mathbf{n}$ denotes the unit normal on $\Gamma$ directed into the domain $\Omega$.

To further recast the boundary integral \eqref{eq:A_boundary_initial}, we utilize the surface identity $\mathbf{n} \cdot (\nabla \times \overline{\mathbf{V}}) = -\operatorname{div}_\Gamma(\mathbf{n} \times \overline{\mathbf{V}})$. Integration by parts on the surface $\Gamma$ then implies that for any scalar field $\psi \in C^1(\Gamma)$, 
\[
\int_\Gamma \psi \, \mathbf{n} \cdot (\nabla \times \overline{\mathbf{V}}) \, \mathrm{d}s = \int_\Gamma \nabla_\Gamma \psi \cdot (\mathbf{n} \times \overline{\mathbf{V}}) \, \mathrm{d}s.
\]
Setting $\psi = \mathbf{h} \cdot \nabla \times \mathbf{E}$ and invoking the triple product property $-\bigl( (\mathbf{E} \times \mathbf{h}) \times \overline{\mathbf{V}} \bigr) \cdot \mathbf{n} = (\mathbf{E} \times \mathbf{h}) \cdot (\mathbf{n} \times \overline{\mathbf{V}})$, we obtain the following representation for the volume term $\mathscr V$:
\begin{equation}
\label{eq:A_final}
\mathscr V = \int_\Gamma \nabla_\Gamma (\mathbf{h} \cdot \nabla \times \mathbf{E}) \cdot (\mathbf{n} \times \overline{\mathbf{V}}) \, \mathrm{d}s + k^2 \int_\Gamma (\mathbf{E} \times \mathbf{h}) \cdot (\mathbf{n} \times \overline{\mathbf{V}}) \, \mathrm{d}s.
\end{equation}

We now introduce the orthogonal decompositions of the deformation $\mathbf{h}$ and the reference field $\mathbf{E}$ into their tangential and normal components:
\[
\mathbf{h} = \mathbf{h}_T + h_n \mathbf{n} \quad \text{and} \quad \mathbf{E} = \gamma_T \mathbf{E} + E_n \mathbf{n},
\]
where $h_n := \mathbf{h} \cdot \mathbf{n}$ and $E_n := \mathbf{E} \cdot \mathbf{n}$. In the first integrand of \eqref{eq:A_final}, we use the identity $(\nabla_\Gamma \psi) \cdot (\mathbf{n} \times \overline{\mathbf{V}}) = -(\mathbf{n} \times \nabla_\Gamma \psi) \cdot \overline{\mathbf{V}}$ to find
\[
\begin{aligned}
\nabla_\Gamma (\mathbf{h} \cdot \nabla \times \mathbf{E}) \cdot (\mathbf{n} \times \overline{\mathbf{V}}) 
&= \left[ \nabla_\Gamma (\mathbf{h}_T \cdot \nabla \times \mathbf{E}) + \nabla_\Gamma \bigl( h_n (\nabla \times \mathbf{E}) \cdot \mathbf{n} \bigr) \right] \cdot (\mathbf{n} \times \overline{\mathbf{V}}) \\
&= - \left[ \operatorname{Curl}_\Gamma (\mathbf{h}_T \cdot \nabla \times \mathbf{E}) + \operatorname{Curl}_\Gamma \bigl( h_n \operatorname{curl}_\Gamma(\gamma_T \mathbf{E}) \bigr) \right] \cdot \overline{\mathbf{V}},
\end{aligned}
\]
where $\operatorname{Curl}_\Gamma \psi := \mathbf{n} \times \nabla_\Gamma \psi$ is the surface vector curl. For the second integrand, the triple product expansion $(\mathbf{E} \times \mathbf{h}) \times \mathbf{n} = (\mathbf{E} \cdot \mathbf{n})\mathbf{h} - (\mathbf{h} \cdot \mathbf{n})\mathbf{E}$ yields
\[
k^2 (\mathbf{E} \times \mathbf{h}) \cdot (\mathbf{n} \times \overline{\mathbf{V}}) 
= k^2 \bigl[ (\mathbf{E} \times \mathbf{h}) \times \mathbf{n} \bigr] \cdot \overline{\mathbf{V}} 
= k^2 \bigl[ E_n \mathbf{h}_T - h_n \gamma_T \mathbf{E} \bigr] \cdot \overline{\mathbf{V}}.
\]
Consequently, the volume contribution $\mathscr V$ is expressed as the boundary integral
\begin{equation}
\label{eq:A_final_decomposed}
\begin{aligned}
\mathscr V &= -\int_\Gamma \left[ \operatorname{Curl}_\Gamma (\mathbf{h}_T \cdot \nabla \times \mathbf{E}) + \operatorname{Curl}_\Gamma \bigl( h_n \operatorname{curl}_\Gamma(\gamma_T \mathbf{E}) \bigr) \right] \cdot \overline{\mathbf{V}} \, \mathrm{d}s \\
&\quad + k^2 \int_\Gamma \left[ E_n \mathbf{h}_T - h_n \gamma_T \mathbf{E} \right] \cdot \overline{\mathbf{V}} \, \mathrm{d}s.
\end{aligned}
\end{equation}

The normal component $E_n$ on $\Gamma$ is characterized by projecting the Maxwell system $\nabla \times (\nabla \times \mathbf{E}) = k^2 \mathbf{E}$ onto the unit normal $\mathbf{n}$, which yields
\[
k^2 E_n = \mathbf{n} \cdot \nabla \times (\nabla \times \mathbf{E}).
\]
Utilizing the identity $\mathbf{n} \cdot (\nabla \times \mathbf{U}) = -\operatorname{div}_\Gamma (\mathbf{n} \times \mathbf{U})$ for the field $\mathbf{U} = \nabla \times \mathbf{E}$ and incorporating the nonlinear impedance boundary condition
\[
\mathbf{n} \times (\nabla \times \mathbf{E}) = \mathbf{g}(\boldsymbol{x}, \gamma_T \mathbf{E}) - \mathrm{i} k \lambda \, \gamma_T \mathbf{E},
\]
we obtain a representation of the normal component solely in terms of surface differential operators:
\[
k^2 E_n = \operatorname{div}_\Gamma \left( \mathrm{i} k \lambda \, \gamma_T \mathbf{E} - \mathbf{g}(\boldsymbol{x}, \gamma_T \mathbf{E}) \right).
\]

We proceed to characterize the boundary functionals $I_1$ and $I_2$. While we initially retain the tangential perturbation $\mathbf{h}_T$, we shall subsequently demonstrate that all contributions associated with $\mathbf{h}_T$ vanish, in accordance with Hadamard's boundary symmetry principle. For notational convenience, let $\mathbf{g}_E := \mathbf{g}(\cdot, \gamma_T \mathbf{E})$ denote the nonlinear response evaluated at the reference tangential trace on $\Gamma$.

We utilize the orthogonal decompositions
\[
\mathbf{h} = \mathbf{h}_T + h_n \mathbf{n}, \quad \mathbf{E} = \gamma_T \mathbf{E} + E_n \mathbf{n}, \quad \overline{\mathbf{V}} = \gamma_T \overline{\mathbf{V}} + (\mathbf{n} \cdot \overline{\mathbf{V}}) \mathbf{n},
\]
together with the first-order variation of the unit normal $\delta \mathbf{n} = -\nabla_\Gamma h_n + J_{\mathbf{n}} \mathbf{h}_T$, where $J_{\mathbf{n}}$ is the shape operator (Weingarten map) on $\Gamma$. We first simplify the terms in $I_1$ arising from the domain transformation.

Regarding the contribution involving the Jacobian $J_{\mathbf{E}}$, we expand $J_{\mathbf{E}} \mathbf{h} = J_{\mathbf{E}} \mathbf{h}_T + h_n \frac{\partial \mathbf{E}}{\partial \mathbf{n}}$. Differentiating the decomposition of $\mathbf{E}$ along the tangential field $\mathbf{h}_T$ yields
\[
J_{\mathbf{E}} \mathbf{h}_T = (\nabla_\Gamma \gamma_T \mathbf{E}) \mathbf{h}_T + (\nabla_\Gamma E_n \cdot \mathbf{h}_T) \mathbf{n} + E_n (J_{\mathbf{n}} \mathbf{h}_T).
\]
Since the test field $\gamma_T \overline{\mathbf{V}}$ is strictly tangential, its inner product with any normal component vanishes identically. Consequently, the coordinate transformation term, when projected onto the test space, admits the representation
\[
\begin{aligned}
\mathbf{n} \times \bigl( (J_{\mathbf{E}} \mathbf{h}) \times \mathbf{n} \bigr) \cdot \gamma_T \overline{\mathbf{V}} &= (J_{\mathbf{E}} \mathbf{h}) \cdot \gamma_T \overline{\mathbf{V}} \\
&= \bigl( (\nabla_\Gamma \gamma_T \mathbf{E}) \mathbf{h}_T \bigr) \cdot \gamma_T \overline{\mathbf{V}} + E_n (J_{\mathbf{n}} \mathbf{h}_T) \cdot \gamma_T \overline{\mathbf{V}} + h_n \frac{\partial \mathbf{E}}{\partial \mathbf{n}} \cdot \gamma_T \overline{\mathbf{V}}.
\end{aligned}
\]

Regarding the term involving $J_{\mathbf{h}} \gamma_T \mathbf{E}$, we decompose the Jacobian as $J_{\mathbf{h}} \gamma_T \mathbf{E} = J_{\mathbf{h}_T} \gamma_T \mathbf{E} + J_{h_n \mathbf{n}} \gamma_T \mathbf{E}$. Differentiating the tangential field $\gamma_T \mathbf{E}$ as a section of the tangent bundle yields
\[
J_{\mathbf{h}_T} \gamma_T \mathbf{E} = (\nabla_\Gamma \mathbf{h}_T) \gamma_T \mathbf{E} - \bigl( \gamma_T \mathbf{E} \cdot (J_{\mathbf{n}} \mathbf{h}_T) \bigr) \mathbf{n},
\]
while the normal variation contributes
\[
J_{h_n \mathbf{n}} \gamma_T \mathbf{E} = h_n J_{\mathbf{n}} \gamma_T \mathbf{E} + (\gamma_T \mathbf{E} \cdot \nabla_\Gamma h_n) \mathbf{n}.
\]
Consequently, the resulting scalar product admits the representation
\[
\begin{aligned}
(J_{\mathbf{h}} \gamma_T \mathbf{E}) \cdot \overline{\mathbf{V}} &= \bigl( (\nabla_\Gamma \mathbf{h}_T) \gamma_T \mathbf{E} \bigr) \cdot \gamma_T \overline{\mathbf{V}} - \bigl( \gamma_T \mathbf{E} \cdot (J_{\mathbf{n}} \mathbf{h}_T) \bigr) (\mathbf{n} \cdot \overline{\mathbf{V}}) \\
&\quad + h_n (J_{\mathbf{n}} \gamma_T \mathbf{E}) \cdot \gamma_T \overline{\mathbf{V}} + (\gamma_T \mathbf{E} \cdot \nabla_\Gamma h_n) (\mathbf{n} \cdot \overline{\mathbf{V}}).
\end{aligned}
\]

Substituting these identities into the expression for $I_1$ and aggregating the tangential and normal components, we obtain
\[
\begin{aligned}
I_1 &= -\mathrm{i}k\lambda \int_\Gamma \Bigl[ \gamma_T \mathbf{E} \operatorname{div}_\Gamma \mathbf{h}_T - (\nabla_\Gamma \mathbf{h}_T) \gamma_T \mathbf{E} + (\nabla_\Gamma \gamma_T \mathbf{E}) \mathbf{h}_T \\
&\qquad \qquad \quad + 2\kappa h_n \gamma_T \mathbf{E} + E_n \nabla_\Gamma h_n + h_n \Bigl( \frac{\partial \gamma_T \mathbf{E}}{\partial \mathbf{n}} - J_{\mathbf{n}} \gamma_T \mathbf{E} \Bigr) \Bigr]^{\top} \gamma_T \overline{\mathbf{V}} \,\mathrm{d}s.
\end{aligned}
\]
By invoking the Weingarten map $\mathcal{S}(\boldsymbol \tau) := J_{\mathbf{n}} (\boldsymbol \tau)$ and the additive  curvature $\mathfrak{H} := \operatorname{tr} \mathcal{S} = 2\kappa$, we consolidate the curvature-related terms using the identity $2\kappa h_n \gamma_T \mathbf{E} - h_n J_{\mathbf{n}} \gamma_T \mathbf{E} = h_n \mathfrak{H} \gamma_T \mathbf{E} - h_n \mathcal{S}(\gamma_T \mathbf{E})$. This yields the following refined characterization of the impedance boundary contribution:
\begin{equation}
\label{eq:I1_refined}
\begin{aligned}
I_1 &= -\mathrm{i}k\lambda \int_\Gamma \Bigl[ \gamma_T \mathbf{E} \operatorname{div}_\Gamma \mathbf{h}_T - (\nabla_\Gamma \mathbf{h}_T) \gamma_T \mathbf{E} + (\nabla_\Gamma \gamma_T \mathbf{E}) \mathbf{h}_T \\
&\qquad \qquad \quad + h_n \mathfrak{H} \gamma_T \mathbf{E} - h_n \mathcal{S}(\gamma_T \mathbf{E}) + E_n \nabla_\Gamma h_n + h_n \frac{\partial \gamma_T \mathbf{E}}{\partial \mathbf{n}} \Bigr]^{\top} \gamma_T \overline{\mathbf{V}} \,\mathrm{d}s.
\end{aligned}
\end{equation}

We now address the reduction of the nonlinear boundary functional $I_2$. As the relevant geometric decompositions were established in the analysis of $I_1$, we focus here on the identities specific to the nonlinear response. First, we expand the Jacobian of the deformed nonlinearity as
\[
J_{\mathbf{h}}\mathbf{g}_E = J_{\mathbf{h}_T}\mathbf{g}_E + \mathbf{n}(\nabla h_n)^{\top}\mathbf{g}_E + h_n \mathcal{S}\mathbf{g}_E.
\]
The second term on the right-hand side is strictly normal; consequently, its contribution vanishes identically upon pairing with the tangential test field $\gamma_T\overline{\mathbf{V}}$. Furthermore, the chain rule for the composition yields
\[
(\nabla_\Gamma \mathbf{g}_E) \mathbf{h}_T = \nabla_{\boldsymbol{x}}\mathbf{g}(\cdot,\gamma_T\mathbf{E})\mathbf{h}_T + \mathbf{g}_{\boldsymbol{z}} \bigl(\cdot,\gamma_T\mathbf{E};\, (\nabla_\Gamma\gamma_T\mathbf{E})\mathbf{h}_T\bigr).
\]
Since $\mathbf{g}$ depends on the electric field solely through its tangential trace, the partial variation $\mathbf{g}_{\boldsymbol{z}}$ vanishes for normal arguments; that is, $\mathbf{g}_{\boldsymbol{z}} (\cdot,\gamma_T\mathbf{E}; \eta\mathbf{n}) = 0$ for any scalar field $\eta$.

Substituting these identities into $I_2$ and incorporating the simplification of the coordinate transformation term $\mathbf{n}\times((J_{\mathbf{E}}\mathbf{h})\times\mathbf{n})$ established in the previous step, we obtain
\[
\begin{aligned}
I_2 &= \int_\Gamma \Bigl[ \mathbf{g}_E \operatorname{div}_\Gamma\mathbf{h}_T - (\nabla_\Gamma\mathbf{h}_T)\mathbf{g}_E + (\nabla_\Gamma\mathbf{g}_E)\mathbf{h}_T \\
&\qquad + h_n \Bigl( \partial_{\mathbf{n}}\mathbf{g}_E - \mathcal{S}\mathbf{g}_E \Bigr) + h_n \mathfrak{H}\,\mathbf{g}_E \\
&\qquad + \mathbf{g}_{\boldsymbol{z}} \bigl( \cdot, \gamma_T\mathbf{E};\, -(\delta\mathbf{n}\cdot\mathbf{E})\mathbf{n} - E_n\delta\mathbf{n} \bigr) \Bigr]^{\top} \gamma_T\overline{\mathbf{V}} \, \mathrm{d}s.
\end{aligned}
\]
The term $(\delta\mathbf{n}\cdot\mathbf{E})\mathbf{n}$ in the argument of $\mathbf{g}_{\boldsymbol{z}}$ is normal to the boundary and thus, by the aforementioned property of $\mathbf{g}_{\boldsymbol{z}}$, makes no contribution to the integral. It follows that $I_2$ reduces to
\begin{equation}
\label{eq:I2_refined}
\begin{aligned}
I_2 &= \int_\Gamma \Bigl[ \mathbf{g}_E \operatorname{div}_\Gamma\mathbf{h}_T - (\nabla_\Gamma\mathbf{h}_T)\mathbf{g}_E + (\nabla_\Gamma\mathbf{g}_E)\mathbf{h}_T \\
&\qquad + h_n \Bigl( \partial_{\mathbf{n}}\mathbf{g}_E - \mathcal{S}\mathbf{g}_E \Bigr) + h_n \mathfrak{H}\,\mathbf{g}_E + \mathbf{g}_{\boldsymbol{z}} \bigl( \cdot, \gamma_T\mathbf{E};\, -E_n\delta\mathbf{n} \bigr) \Bigr]^{\top} \gamma_T\overline{\mathbf{V}} \, \mathrm{d}s.
\end{aligned}
\end{equation}

To establish the cancellation of terms involving the tangential variation $\mathbf{h}_T$, we exploit the surface divergence theorem on the manifold $\Gamma$. For any tangential vector fields $\mathbf{p}, \mathbf{q}, \mathbf{r}$ on $\Gamma$, we recall the identity
\[
\operatorname{div}_\Gamma\bigl(\mathbf{p}(\mathbf{q}\cdot\mathbf{r})\bigr) = (\operatorname{div}_\Gamma\mathbf{p})(\mathbf{q}\cdot\mathbf{r}) + \bigl((\nabla_\Gamma\mathbf{q})\mathbf{p}\bigr)\cdot\mathbf{r} + \bigl((\nabla_\Gamma\mathbf{r})\mathbf{p}\bigr)\cdot\mathbf{q}.
\]
We apply this relation to the triples $(\mathbf{h}_T, \gamma_T\mathbf{E}, \gamma_T\overline{\mathbf{V}})$ and $(\gamma_T\mathbf{E}, \mathbf{h}_T, \gamma_T\overline{\mathbf{V}})$, and likewise to $(\mathbf{h}_T, \mathbf{g}_E, \gamma_T\overline{\mathbf{V}})$ and $(\mathbf{g}_E, \mathbf{h}_T, \gamma_T\overline{\mathbf{V}})$. Since $\Gamma$ is a closed surface, the integrals of these surface divergences vanish by Stokes' theorem. Substituting the resulting identities into the expressions for $I_1$ and $I_2$ yields the following intermediate variational formulation:
\begin{equation}\label{eg:shape}
\begin{aligned}
\widehat{\mathcal{A}}(\delta\mathbf{E}, \mathbf{V}) &= \int_\Gamma \Bigl[ (\mathbf{h}_T \cdot \nabla \times \mathbf{E}) (\mathbf{n} \cdot \nabla \times \overline{\mathbf{V}}) - \Bigl( \operatorname{Curl}_\Gamma \bigl( h_n \operatorname{curl}_\Gamma(\gamma_T\mathbf{E}) \bigr) + k^2 h_n \gamma_T\mathbf{E} \Bigr) \cdot \gamma_T\overline{\mathbf{V}} \\
&\quad -\mathrm{i}k\lambda \Bigl( E_n \nabla_\Gamma h_n + h_n \bigl( \tfrac{\partial\gamma_T\mathbf{E}}{\partial\mathbf{n}} - \mathcal{S}(\gamma_T\mathbf{E}) \bigr) \Bigr) \cdot \gamma_T\overline{\mathbf{V}} \\
&\quad - \mathrm{i}k\lambda h_n \mathfrak{H} (\gamma_T\mathbf{E}) \cdot \overline{\mathbf{V}} + h_n \Bigl( \tfrac{\partial\mathbf{g}_E}{\partial\mathbf{n}} - \mathcal{S}(\mathbf{g}_E) \Bigr) \cdot \gamma_T\overline{\mathbf{V}} \\
&\quad + h_n \mathfrak{H} \mathbf{g}_E \cdot \overline{\mathbf{V}} + \mathbf{g}_{\boldsymbol{z}} \bigl(\cdot, \gamma_T\mathbf{E}; -E_n \delta\mathbf{n}\bigr) \cdot \gamma_T\overline{\mathbf{V}} \\
&\quad - \mathrm{i}k\lambda \bigl( \nabla_\Gamma(\gamma_T\overline{\mathbf{V}}) \gamma_T\mathbf{E} \bigr) \cdot \mathbf{h}_T + \mathrm{i}k\lambda \bigl( \nabla_\Gamma(\gamma_T\overline{\mathbf{V}}) \mathbf{h}_T \bigr) \cdot \gamma_T\mathbf{E} \\
&\quad +\bigl( \nabla_\Gamma(\gamma_T\overline{\mathbf{V}}) \mathbf{g}_E \bigr) \cdot \mathbf{h}_T - \bigl( \nabla_\Gamma(\gamma_T\overline{\mathbf{V}}) \mathbf{h}_T \bigr) \cdot \mathbf{g}_E \Bigr] \, \mathrm{d}s.
\end{aligned}
\end{equation}

It remains to simplify the final four terms in \eqref{eg:shape}. By virtue of the boundary condition $\mathbf{n} \times (\nabla \times \mathbf{E}) = -\mathrm{i}k\lambda \gamma_T\mathbf{E} + \mathbf{g}_E$, we define the tangential field $\mathbf{b} := \mathbf{n} \times (\nabla \times \mathbf{E})$ and observe that
\[
\begin{aligned}
&-\mathrm{i}k\lambda \bigl( \nabla_\Gamma(\gamma_T\overline{\mathbf{V}}) \gamma_T\mathbf{E} \bigr) \cdot \mathbf{h}_T + \mathrm{i}k\lambda \bigl( \nabla_\Gamma(\gamma_T\overline{\mathbf{V}}) \mathbf{h}_T \bigr) \cdot \gamma_T\mathbf{E} \\
&\quad + \bigl( \nabla_\Gamma(\gamma_T\overline{\mathbf{V}}) \mathbf{g}_E \bigr) \cdot \mathbf{h}_T - \bigl( \nabla_\Gamma(\gamma_T\overline{\mathbf{V}}) \mathbf{h}_T \bigr) \cdot \mathbf{g}_E \\
&\qquad = \bigl( \nabla_\Gamma(\gamma_T\overline{\mathbf{V}}) \mathbf{b} \bigr) \cdot \mathbf{h}_T - \bigl( \nabla_\Gamma(\gamma_T\overline{\mathbf{V}}) \mathbf{h}_T \bigr) \cdot \mathbf{b} \\
&\qquad = \Bigl[ \bigl( -\nabla_\Gamma(\gamma_T\overline{\mathbf{V}}) + \nabla_\Gamma(\gamma_T\overline{\mathbf{V}})^{\top} \bigr) \mathbf{h}_T \Bigr] \cdot \mathbf{b}.
\end{aligned}
\]
The skew-symmetric part of the surface gradient is related to the normal component of the curl via the identity
\[
\Bigl( -\nabla_\Gamma(\gamma_T\overline{\mathbf{V}}) + \nabla_\Gamma(\gamma_T\overline{\mathbf{V}})^{\top} \Bigr) \mathbf{h}_T = -(\mathbf{n} \cdot \nabla \times \overline{\mathbf{V}}) (\mathbf{n} \times \mathbf{h}_T).
\]
Consequently, the expression reduces to
\[
\Bigl[ \bigl( -\nabla_\Gamma(\gamma_T\overline{\mathbf{V}}) + \nabla_\Gamma(\gamma_T\overline{\mathbf{V}})^{\top} \bigr) \mathbf{h}_T \Bigr] \cdot \mathbf{b} = -(\mathbf{n} \cdot \nabla \times \overline{\mathbf{V}}) (\mathbf{n} \times \mathbf{h}_T) \cdot \bigl( \mathbf{n} \times (\nabla \times \mathbf{E}) \bigr).
\]
Substituting this result into \eqref{eg:shape}, we find that the tangential contribution $(\mathbf{h}_T \cdot \nabla \times \mathbf{E}) (\mathbf{n} \cdot \nabla \times \overline{\mathbf{V}})$ is precisely canceled by the four terms derived above. 

Since all terms involving $\mathbf{h}_T$ vanish, the resulting variational form depends exclusively on the normal component $h_n = \mathbf{h} \cdot \mathbf{n}$. This is consistent with the Hadamard structure theorem for shape derivatives. This recovers the boundary expression \eqref{eq:shape1} and completes the proof.

\end{proof}

\begin{theo}\label{th:shape_derivative_npec}
Suppose that the nonlinear perfectly conducting boundary condition
\[
\mathbf{n} \times \mathbf{E} = \mathbf{g}(\boldsymbol{x}, \gamma_T \mathbf{E}) \qquad \text{on } \Gamma
\]
is satisfied, where \(\mathbf{g}\) satisfies Assumption~\ref{ass:enhanced_g}. 
Then the shape derivative \(\delta \mathbf{E}\) is the unique weak solution to the following \(\mathbb{R}\)-linear boundary value problem:
\begin{equation*}
\nabla \times (\nabla \times \delta \mathbf{E}) - k^2 \delta \mathbf{E} = 0 \qquad \text{in } \Omega,
\end{equation*}
subject to the boundary condition on \(\Gamma\)
\begin{equation}
\label{eq:shape_npec}
\mathbf{n} \times \delta \mathbf{E} - \mathbf{g}_{\boldsymbol{z}} \bigl( \boldsymbol{x}, \gamma_T \mathbf{E}; \gamma_T \delta \mathbf{E} \bigr) = \mathcal{F}_{\mathbf{h}},
\end{equation}
where the right-hand side \(\mathcal{F}_{\mathbf{h}}\) is defined by
\begin{equation}
\label{eq:npec_shape_rhs}
\begin{aligned}
\mathcal{F}_{\mathbf{h}} &:= -\operatorname{Curl}_\Gamma (h_n E_n) - h_n \gamma_T(\nabla \times \mathbf{E}) \\
&\quad + h_n \left( \frac{\partial}{\partial \mathbf{n}} \mathbf{g}(\boldsymbol{x}, \gamma_T \mathbf{E}) - \mathcal{S} \bigl( \mathbf{g}(\boldsymbol{x}, \gamma_T \mathbf{E}) \bigr) \right) + h_n \mathfrak{H} \mathbf{g}(\boldsymbol{x}, \gamma_T \mathbf{E}) \\
&\quad + \mathbf{g}_{\boldsymbol{z}} \bigl( \boldsymbol{x}, \gamma_T \mathbf{E}; E_n \nabla_\Gamma h_n \bigr).
\end{aligned}
\end{equation}
On the artificial boundary \(\Gamma_R\), the shape derivative satisfies the transparent boundary condition
\begin{equation}
\label{eq:npec_shape_radiation}
-\frac{\mathrm{i}}{k} \Lambda^{-1} \bigl( \gamma_t(\nabla \times \delta \mathbf{E}) \bigr) = \gamma_t \delta \mathbf{E} \qquad \text{on } \Gamma_R,
\end{equation}
which ensures that \(\delta \mathbf{E}\) admits a unique extension to a radiating weak solution on \(\mathbb{R}^3 \setminus \overline{D}\). 
\end{theo}

\begin{proof}
Having established the existence of the material derivative $(\mathbf{W}, \mathbf{P}) \in \mathbb{X}_p(\Omega)$ in Theorem~\ref{th:p_material_differentiability}, we now characterize the boundary value problem satisfied by the shape derivative. To this end, we define
\[
\delta\mathbf{E} := \mathbf{W} - \left(J_{\mathbf{h}}^{\top}\mathbf{E} + J_{\mathbf{E}}\mathbf{h}\right), \qquad 
\delta\mathbf{Q} := \mathbf{P} - \left(J_{\mathbf{h}}^{\top}\mathbf{Q} + J_{\mathbf{Q}}\mathbf{h}\right).
\]
By the $\mathbb{R}$-linearity of the linearized mixed form, it follows that
\[
\begin{aligned}
&\widehat{\mathcal{A}}^{\mathrm{p}} ((\delta\mathbf{E}, \delta\mathbf{Q}), (\boldsymbol{\tau}, \boldsymbol{\nu})) \\
&\quad = \mathcal{L}_{\mathbf{h}}^{\mathrm{p}}(\boldsymbol{\tau}, \boldsymbol{\nu}) - \widehat{\mathcal{A}}^{\mathrm{p}} \left( \left( J_{\mathbf{h}}^{\top}\mathbf{E} + J_{\mathbf{E}}\mathbf{h}, J_{\mathbf{h}}^{\top}\mathbf{Q} + J_{\mathbf{Q}}\mathbf{h} \right), (\boldsymbol{\tau}, \boldsymbol{\nu}) \right).
\end{aligned}
\]
To identify the governing equations in the interior, let $(\boldsymbol{\tau}, \boldsymbol{\nu})$ be test functions with compact support in $\Omega$. For such functions, all boundary integrals vanish. From the first mixed equation for $(\mathbf{W}, \mathbf{P})$, we obtain
\[
\int_\Omega \mathbf{P} \cdot \overline{\boldsymbol{\tau}} \,\mathrm{d}x - \int_\Omega \mathbf{W} \cdot (\nabla \times \overline{\boldsymbol{\tau}}) \,\mathrm{d}x = -\int_\Omega \mathbf{Q}^{\top} \dot{\mathcal{N}}_{\mathbf{h}} \overline{\boldsymbol{\tau}} \,\mathrm{d}x.
\]

We next derive the corresponding identity for the geometric transport terms. For any sufficiently smooth vector field $\mathbf{U}$, we have
\[
\nabla\times
\left(
J_{\mathbf{h}}^{\top}\mathbf{U}+J_{\mathbf{U}}\mathbf{h}
\right)
=
\left((\operatorname{div}\mathbf{h})I-J_{\mathbf{h}}\right)
(\nabla\times\mathbf{U})
+
J_{\nabla\times\mathbf{U}}\mathbf{h} .
\]
This follows from the decomposition 
\[
J_{\mathbf{h}}^{\top}\mathbf{U}+J_{\mathbf{U}}\mathbf{h}
=
(\nabla\times\mathbf{U})\times\mathbf{h}
+
\nabla(\mathbf{h}\cdot\mathbf{U})
\]
in conjunction with the standard vector identity for $\nabla\times(\mathbf{a}\times\mathbf{b})$. Setting $\mathbf{U}=\mathbf{E}$ and noting that $\nabla\times\mathbf{E}=\mathbf{Q}$, we obtain
\[
\nabla\times
\left(
J_{\mathbf{h}}^{\top}\mathbf{E}+J_{\mathbf{E}}\mathbf{h}
\right)
=
\left((\operatorname{div}\mathbf{h})I-J_{\mathbf{h}}\right)\mathbf{Q}
+
J_{\mathbf{Q}}\mathbf{h} .
\]
Recalling the definition $\dot{\mathcal{N}}_{\mathbf{h}} = (\operatorname{div}\mathbf{h})I-J_{\mathbf{h}}-J_{\mathbf{h}}^{\top}$, we arrive at the relation
\[
\nabla\times
\left(
J_{\mathbf{h}}^{\top}\mathbf{E}+J_{\mathbf{E}}\mathbf{h}
\right)
=
J_{\mathbf{h}}^{\top}\mathbf{Q}+J_{\mathbf{Q}}\mathbf{h}
+
\dot{\mathcal{N}}_{\mathbf{h}}\mathbf{Q} .
\]
Consequently, by integration by parts and the fact that $\boldsymbol{\tau}$ is compactly supported in $\Omega$, it follows that
\[
\begin{aligned}
&\int_\Omega
\left(J_{\mathbf{h}}^{\top}\mathbf{Q}+J_{\mathbf{Q}}\mathbf{h}\right)
\cdot\overline{\boldsymbol{\tau}}\,\mathrm{d}x
-\int_\Omega
\left(J_{\mathbf{h}}^{\top}\mathbf{E}+J_{\mathbf{E}}\mathbf{h}\right)
\cdot(\nabla\times\overline{\boldsymbol{\tau}})\,\mathrm{d}x \\
&\quad = -\int_\Omega \mathbf{Q}^\top\dot{\mathcal{N}}_{\mathbf{h}} \overline{\boldsymbol{\tau}}\,\mathrm{d}x .
\end{aligned}
\]
Subtracting this identity from the first mixed equation for $(\mathbf{W}, \mathbf{P})$ yields
\[
\int_\Omega \delta\mathbf{Q}\cdot\overline{\boldsymbol{\tau}}\,\mathrm{d}x - \int_\Omega \delta\mathbf{E}\cdot(\nabla\times\overline{\boldsymbol{\tau}})\,\mathrm{d}x = 0,
\]
which implies that the shape derivatives satisfy
\[
\delta\mathbf{Q}=\nabla\times\delta\mathbf{E} \quad \text{in } \Omega.
\]

Analogously, the second mixed equation for $(\mathbf{W}, \mathbf{P})$ provides
\[
\int_\Omega
(\nabla\times\mathbf{P})\cdot\overline{\boldsymbol{\nu}}\,\mathrm{d}x
-
k^2\int_\Omega
\mathbf{W}\cdot\overline{\boldsymbol{\nu}}\,\mathrm{d}x
=
k^2\int_\Omega
\mathbf{E}^\top\dot{\mathcal{N}}_{\mathbf{h}}
\overline{\boldsymbol{\nu}}\,\mathrm{d}x.
\]
By setting $\mathbf{U}=\mathbf{Q}$ in the previously established vector identity and invoking the relation $\nabla\times\mathbf{Q}=k^2\mathbf{E}$, we find
\[
\begin{aligned}
\nabla\times
\left(
J_{\mathbf{h}}^{\top}\mathbf{Q}+J_{\mathbf{Q}}\mathbf{h}
\right)
&=
\left((\operatorname{div}\mathbf{h})I-J_{\mathbf{h}}\right)
(\nabla\times\mathbf{Q})
+
J_{\nabla\times\mathbf{Q}}\mathbf{h}
\\
&=
k^2
\left((\operatorname{div}\mathbf{h})I-J_{\mathbf{h}}\right)\mathbf{E}
+
k^2J_{\mathbf{E}}\mathbf{h}.
\end{aligned}
\]
It follows that
\[
\begin{aligned}
&\nabla\times
\left(
J_{\mathbf{h}}^{\top}\mathbf{Q}+J_{\mathbf{Q}}\mathbf{h}
\right)
-
k^2
\left(
J_{\mathbf{h}}^{\top}\mathbf{E}+J_{\mathbf{E}}\mathbf{h}
\right)
\\
&\quad=
k^2
\left[
(\operatorname{div}\mathbf{h})I-J_{\mathbf{h}}-J_{\mathbf{h}}^{\top}
\right]\mathbf{E}
=
k^2\dot{\mathcal{N}}_{\mathbf{h}}\mathbf{E}.
\end{aligned}
\]
Consequently, we have
\[
\int_\Omega
\nabla\times
\left(J_{\mathbf{h}}^{\top}\mathbf{Q}+J_{\mathbf{Q}}\mathbf{h}\right)
\cdot\overline{\boldsymbol{\nu}}\,\mathrm{d}x
-
k^2\int_\Omega
\left(J_{\mathbf{h}}^{\top}\mathbf{E}+J_{\mathbf{E}}\mathbf{h}\right)
\cdot\overline{\boldsymbol{\nu}}\,\mathrm{d}x
=
k^2\int_\Omega
\mathbf{E}^\top
\dot{\mathcal{N}}_{\mathbf{h}}
\overline{\boldsymbol{\nu}}\,\mathrm{d}x.
\]
Subtracting this identity from the second mixed equation for $(\mathbf{W}, \mathbf{P})$, we obtain
\[
\int_\Omega
(\nabla\times\delta\mathbf{Q})\cdot\overline{\boldsymbol{\nu}}\,\mathrm{d}x
-
k^2\int_\Omega
\delta\mathbf{E}\cdot\overline{\boldsymbol{\nu}}\,\mathrm{d}x
=0,
\]
which implies that $\nabla\times\delta\mathbf{Q}-k^2\delta\mathbf{E}=0$ in $\Omega$. Combining this result with the relation $\delta\mathbf{Q}=\nabla\times\delta\mathbf{E}$ established previously, we conclude that $\delta \mathbf{E}$ satisfies the  Maxwell equation
\[
\nabla\times(\nabla\times\delta\mathbf{E})-k^2\delta\mathbf{E}=0 \quad \text{in } \Omega.
\]

It remains to characterize the boundary condition on \(\Gamma\). We now use
only the first equation in the mixed variational identity and, for notational
simplicity, denote its test function \(\boldsymbol{\tau}\) by \(\mathbf V\). Since $\delta\mathbf{Q} = \nabla \times \delta\mathbf{E}$, the mixed variational identity may be recast as a single-field boundary variational identity for the shape derivative $\delta\mathbf{E}$. Let $\widehat{\mathcal{A}}^{\mathrm{p}}$ denote the resulting linearized single-field form, defined such that
\[
\widehat{\mathcal{A}}^{\mathrm{p}}(\delta\mathbf{E}, \mathbf{V}) = \widehat{\mathcal{A}}^{\mathrm{p}}(\mathbf{W}, \mathbf{V}) - \widehat{\mathcal{A}}^{\mathrm{p}}\left(J_{\mathbf{h}}^{\top}\mathbf{E} + J_{\mathbf{E}}\mathbf{h}, \mathbf{V}\right).
\]
Let $\mathbf{g}_E := \mathbf{g}(\cdot, \gamma_T \mathbf{E})$ on $\Gamma$. Under the nonlinear perfectly conducting condition $\mathbf{n} \times \mathbf{E} = \mathbf{g}_E$, the material derivative satisfies the linearized boundary identity
\begin{equation*}
\begin{aligned}
\widehat{\mathcal{A}}^{\mathrm{p}}(\mathbf{W}, \mathbf{V}) &= -\int_\Gamma \Bigl[ (\operatorname{div}_\Gamma \mathbf{h}_T + 2\kappa h_n) \mathbf{g}_E + \mathbf{g}_{\boldsymbol{x}}(\cdot, \gamma_T \mathbf{E}) \mathbf{h} \\
&\qquad + \mathbf{g}_{\boldsymbol{z}} \Bigl( \cdot, \gamma_T \mathbf{E}; \delta \mathbf{n} \times (\mathbf{E} \times \mathbf{n}) + \mathbf{n} \times (\mathbf{E} \times \delta \mathbf{n}) - \mathbf{n} \times \bigl( (J_{\mathbf{h}}^{\top} \mathbf{E}) \times \mathbf{n} \bigr) \Bigr) \Bigr] \cdot \gamma_T\overline{\mathbf{V}} \, \mathrm{d}s \\
&\quad +\int_\Gamma \Bigl[ (J_{\mathbf{h}} \mathbf{g}_E) \cdot \overline{\mathbf{V}} - \mathbf{g}_E \cdot \bigl( \mathbf{n} \times (\overline{\mathbf{V}} \times \delta \mathbf{n}) \bigr) \Bigr] \, \mathrm{d}s,
\end{aligned}
\end{equation*}
where the second integral accounts for the first-order expansion of the transported test function and the pulled-back boundary pairing. 
Combining the volume term in \(\mathcal L_{\mathbf h}^{\mathrm p}\) with the
transport contribution and using
\[
\nabla\times
(J_{\mathbf h}^{\top}\mathbf E+J_{\mathbf E}\mathbf h)
=
J_{\mathbf h}^{\top}\mathbf Q
+
J_{\mathbf Q}\mathbf h
+
\dot{\mathcal N}_{\mathbf h}\mathbf Q,
\]
we obtain the following single-field transport contribution:
\begin{equation*}
\begin{aligned}
\widehat{\mathcal{A}}^{\mathrm{p}}(J_{\mathbf{h}}^{\top} \mathbf{E} + J_{\mathbf{E}} \mathbf{h}, \mathbf{V}) &= \int_\Omega \Bigl[ \nabla \times (J_{\mathbf{h}}^{\top} \mathbf{E} + J_{\mathbf{E}} \mathbf{h}) \cdot \overline{\mathbf{V}} - (J_{\mathbf{h}}^{\top} \mathbf{E} + J_{\mathbf{E}} \mathbf{h}) \cdot \nabla \times \overline{\mathbf{V}} \Bigr] \, \mathrm{d}x \\
&\quad + \int_\Gamma \mathbf{g}_{\boldsymbol{z}} \bigl( \cdot, \gamma_T \mathbf{E}; \gamma_T (J_{\mathbf{h}}^{\top} \mathbf{E} + J_{\mathbf{E}} \mathbf{h}) \bigr) \cdot \gamma_T\overline{\mathbf{V}} \, \mathrm{d}s.
\end{aligned}
\end{equation*}
Given that $\operatorname{supp} \mathbf{h} \Subset B_R$, the contribution from the artificial boundary $\Gamma_R$ vanishes identically. Consequently, we decompose the variational form for the shape derivative as
\[
\widehat{\mathcal{A}}^{\mathrm{p}}(\delta\mathbf{E}, \mathbf{V}) = A + B,
\]
where the volume contribution $A$ and boundary contribution $B$ are given by
\begin{equation*}
A := -\int_\Omega \Bigl[ \nabla \times (J_{\mathbf{h}}^{\top}\mathbf{E} + J_{\mathbf{E}}\mathbf{h}) \cdot \overline{\mathbf{V}} - (J_{\mathbf{h}}^{\top}\mathbf{E} + J_{\mathbf{E}}\mathbf{h}) \cdot \nabla \times \overline{\mathbf{V}} \Bigr] \, \mathrm{d}x 
\end{equation*}
and
\begin{equation*}
\begin{aligned}
B &:= -\int_\Gamma \Bigl[ (\operatorname{div}_\Gamma \mathbf{h}_T + 2\kappa h_n) \mathbf{g}_E + \mathbf{g}_{\boldsymbol{x}}(\cdot, \gamma_T \mathbf{E}) \mathbf{h} \\
&\qquad + \mathbf{g}_{\boldsymbol{z}} \Bigl( \cdot, \gamma_T \mathbf{E}; \, \delta \mathbf{n} \times (\mathbf{E} \times \mathbf{n}) + \mathbf{n} \times (\mathbf{E} \times \delta \mathbf{n}) - \mathbf{n} \times \bigl( (J_{\mathbf{h}}^{\top} \mathbf{E}) \times \mathbf{n} \bigr) \Bigr) \Bigr] \cdot \overline{\mathbf{V}} \, \mathrm{d}s \\
&\quad + \int_\Gamma \Bigl[ (J_{\mathbf{h}} \mathbf{g}_E) \cdot \overline{\mathbf{V}} - \mathbf{g}_E \cdot \bigl( \mathbf{n} \times (\overline{\mathbf{V}} \times \delta \mathbf{n}) \bigr) \Bigr] \, \mathrm{d}s \\
&\quad - \int_\Gamma \mathbf{g}_{\boldsymbol{z}} \Bigl( \cdot, \gamma_T \mathbf{E}; \, \gamma_T(J_{\mathbf{h}}^{\top} \mathbf{E} + J_{\mathbf{E}} \mathbf{h}) \Bigr) \cdot \overline{\mathbf{V}} \, \mathrm{d}s.
\end{aligned}
\end{equation*}

We begin by simplifying the volume contribution $A$. Recalling the identity
\[
J_{\mathbf{h}}^{\top}\mathbf{E} + J_{\mathbf{E}}\mathbf{h} = (J_{\mathbf{E}} - J_{\mathbf{E}}^{\top})\mathbf{h} + \nabla(\mathbf{h} \cdot \mathbf{E}),
\]
and observing that the skew-symmetric part of the Jacobian $(J_{\mathbf{E}} - J_{\mathbf{E}}^{\top})\mathbf{h}$ coincides with the cross product $(\nabla \times \mathbf{E}) \times \mathbf{h}$, we have
\[
J_{\mathbf{h}}^{\top}\mathbf{E} + J_{\mathbf{E}}\mathbf{h} = (\nabla \times \mathbf{E}) \times \mathbf{h} + \nabla(\mathbf{h} \cdot \mathbf{E}).
\]
An application of the generalized Stokes theorem, combined with the fact that $\mathbf{h}$ vanishes in a neighborhood of $\Gamma_R$, yields the boundary representation
\begin{equation*}
\begin{aligned}
A &= \int_\Gamma \mathbf{n} \times (J_{\mathbf{h}}^{\top}\mathbf{E} + J_{\mathbf{E}}\mathbf{h}) \cdot \overline{\mathbf{V}} \, \mathrm{d}s \\
&= \int_\Gamma \Bigl[ \mathbf{n} \times \bigl( (\nabla \times \mathbf{E}) \times \mathbf{h} \bigr) + \mathbf{n} \times \nabla(\mathbf{h} \cdot \mathbf{E}) \Bigr] \cdot \overline{\mathbf{V}} \, \mathrm{d}s.
\end{aligned}
\end{equation*}
Invoking the orthogonal decompositions $\mathbf{h} = \mathbf{h}_T + h_n \mathbf{n}$ and $\mathbf{E} = \gamma_T \mathbf{E} + E_n \mathbf{n}$, and utilizing the identities
\[
\mathbf{n} \times \bigl( (\nabla \times \mathbf{E}) \times \mathbf{h} \bigr) = h_n \gamma_T(\nabla \times \mathbf{E}) - (\mathbf{n} \cdot \nabla \times \mathbf{E}) \mathbf{h}_T
\]
and
\[
\mathbf{n} \times \nabla(\mathbf{h} \cdot \mathbf{E}) = \operatorname{Curl}_\Gamma(h_n E_n) + \operatorname{Curl}_\Gamma(\mathbf{h}_T \cdot \gamma_T \mathbf{E}),
\]
the expression for $A$ reduces to
\begin{equation*}
A = \int_\Gamma \Bigl[ h_n \gamma_T(\nabla \times \mathbf{E}) + \operatorname{Curl}_\Gamma(h_n E_n) + \operatorname{Curl}_\Gamma(\mathbf{h}_T \cdot \gamma_T \mathbf{E}) - (\mathbf{n} \cdot \nabla \times \mathbf{E}) \mathbf{h}_T \Bigr] \cdot \gamma_T\overline{\mathbf{V}} \, \mathrm{d}s.
\end{equation*}
Finally, noting that $\mathbf{n} \cdot (\nabla \times \mathbf{U}) = -\operatorname{div}_\Gamma(\mathbf{n} \times \mathbf{U})$, the NPEC condition $\mathbf{n} \times \mathbf{E} = \mathbf{g}_E$ implies $\mathbf{n} \cdot \nabla \times \mathbf{E} = -\operatorname{div}_\Gamma \mathbf{g}_E$. Consequently, the volume contribution becomes the boundary integral
\begin{equation*}
A = \int_\Gamma \Bigl[ h_n \gamma_T(\nabla \times \mathbf{E}) + \operatorname{Curl}_\Gamma(h_n E_n) + \operatorname{Curl}_\Gamma(\mathbf{h}_T \cdot \gamma_T \mathbf{E}) + (\operatorname{div}_\Gamma \mathbf{g}_E) \mathbf{h}_T \Bigr] \cdot \gamma_T\overline{\mathbf{V}} \, \mathrm{d}s.
\end{equation*}

We next  address the boundary contribution $B$. Since $\gamma_T(J_{\mathbf{h}}^{\top}\mathbf{E}) = \mathbf{n} \times \bigl( (J_{\mathbf{h}}^{\top}\mathbf{E}) \times \mathbf{n} \bigr)$, the terms involving $J_{\mathbf{h}}^{\top}\mathbf{E}$ within the argument of $\mathbf{g}_{\boldsymbol{z}}$ cancel. Consequently, $B$ simplifies to
\begin{equation*}
\begin{aligned}
B &= -\int_\Gamma \Bigl[ (\operatorname{div}_\Gamma\mathbf{h}_T + 2\kappa h_n) \mathbf{g}_E + \mathbf{g}_{\boldsymbol{x}}(\cdot,\gamma_T\mathbf{E})\mathbf{h} \\
&\qquad + \mathbf{g}_{\boldsymbol{z}} \Bigl( \cdot, \gamma_T\mathbf{E}; \delta\mathbf{n}\times(\mathbf{E}\times\mathbf{n}) + \mathbf{n}\times(\mathbf{E}\times\delta\mathbf{n}) + \mathbf{n}\times\bigl((J_{\mathbf{E}}\mathbf{h})\times\mathbf{n}\bigr) \Bigr) \Bigr] \cdot \gamma_T\overline{\mathbf{V}} \, \mathrm{d}s \\
&\quad +\int_\Gamma \Bigl[ (J_{\mathbf{h}}\mathbf{g}_E) \cdot \overline{\mathbf{V}} - \mathbf{g}_E \cdot \bigl( \mathbf{n} \times (\overline{\mathbf{V}} \times \delta\mathbf{n}) \bigr) \Bigr] \, \mathrm{d}s.
\end{aligned}
\end{equation*}

Substituting the variation of the unit normal $\delta\mathbf{n} = -\nabla_\Gamma h_n + J_{\mathbf{n}}\mathbf{h}_T$ and decomposing the terms $J_{\mathbf{h}}\mathbf{g}_E$ and the derivative operators, we resolve $B$ into its normal and tangential components. The normal component, $B_{\mathbf{n}}$, is given by
\begin{equation*}
B_{\mathbf{n}} = -\int_\Gamma \Bigl[ h_n \bigl( \tfrac{\partial \mathbf{g}_E}{\partial\mathbf{n}} - \mathcal{S}\mathbf{g}_E \bigr) + h_n \mathfrak{H} \mathbf{g}_E + \mathbf{g}_{\boldsymbol{z}} \bigl( \cdot, \gamma_T\mathbf{E}; E_n \nabla_\Gamma h_n \bigr) \Bigr] \cdot \gamma_T\overline{\mathbf{V}} \, \mathrm{d}s,
\end{equation*}
where $\mathcal{S} = J_{\mathbf{n}}$ denotes the shape operator (Weingarten map) and $\mathfrak{H} = \operatorname{div}_\Gamma \mathbf{n} = \operatorname{tr}\mathcal{S}$ is the additive curvature. In this derivation, we have exploited the fact that $\mathbf{g}$ depends on $\mathbf{E}$ exclusively through its tangential trace; consequently, purely normal increments in the second argument do not contribute to the partial derivative $\mathbf{g}_{\boldsymbol{z}}$.

The tangential component of $B$ is expressed as
\begin{equation*}
B_T = -\int_\Gamma \Bigl[ \mathbf{g}_E \operatorname{div}_\Gamma\mathbf{h}_T - (\nabla_\Gamma\mathbf{h}_T)\mathbf{g}_E + (\nabla_\Gamma\mathbf{g}_E)\mathbf{h}_T \Bigr] \cdot \gamma_T\overline{\mathbf{V}} \, \mathrm{d}s.
\end{equation*}
This representation follows from the surface divergence identities
\begin{equation*}
\begin{aligned}
\operatorname{div}_\Gamma \bigl(\mathbf{g}_E(\mathbf{h}_T \cdot \gamma_T\overline{\mathbf{V}})\bigr) &= (\operatorname{div}_\Gamma\mathbf{g}_E) (\mathbf{h}_T \cdot \gamma_T\overline{\mathbf{V}}) + (\nabla_\Gamma{\mathbf{h}_T}\mathbf{g}_E) \cdot \gamma_T\overline{\mathbf{V}} + (\nabla_\Gamma{{\gamma_T\overline{\mathbf{V}}}}\mathbf{g}_E) \cdot \mathbf{h}_T, \\
\operatorname{div}_\Gamma \bigl(\mathbf{h}_T(\mathbf{g}_E \cdot \gamma_T\overline{\mathbf{V}})\bigr) &= (\operatorname{div}_\Gamma\mathbf{h}_T) (\mathbf{g}_E \cdot \gamma_T\overline{\mathbf{V}}) + (\nabla_\Gamma{\mathbf{g}_E}\mathbf{h}_T) \cdot \gamma_T\overline{\mathbf{V}} + (\nabla_\Gamma{\gamma_T\overline{\mathbf{V}}}\mathbf{h}_T) \cdot \mathbf{g}_E,
\end{aligned}
\end{equation*}
together with the relation $\operatorname{Curl}_\Gamma \psi = \mathbf{n} \times \nabla_\Gamma \psi$. A direct calculation reveals that $B_T$ precisely cancels the tangential contributions arising from the volume term $A$.

Consequently, all terms involving the tangential variation $\mathbf{h}_T$ vanish, a result consistent with the Hadamard structure theorem for shape derivatives,leaving the expression dependent solely on the normal component $h_n$. Combining the results for $A$ and $B$ yields the variational identity
\begin{equation*}
\widehat{\mathcal{A}}^{\mathrm{p}}(\delta\mathbf{E}, \mathbf{V}) = -\int_\Gamma \mathcal{F}_{\mathbf{h}} \cdot \gamma_T\overline{\mathbf{V}} \, \mathrm{d}s,
\end{equation*}
where $\mathcal{F}_{\mathbf{h}}$ is the boundary data defined in \eqref{eq:npec_shape_rhs}. This serves as the weak formulation of the boundary condition 
\begin{equation*}
\mathbf{n} \times \delta\mathbf{E} - \mathbf{g}_{\boldsymbol{z}} \bigl( \boldsymbol{x}, \gamma_T\mathbf{E}; \gamma_T\delta\mathbf{E} \bigr) = \mathcal{F}_{\mathbf{h}} \quad \text{on } \Gamma.
\end{equation*}

Finally, since $\mathbf{h}$ vanishes in a neighborhood of the artificial boundary $\Gamma_R$, the condition \eqref{eq:npec_shape_radiation} is inherited directly from the material derivative. The uniqueness of $\delta\mathbf{E}$ follows from the well-posedness of the $\mathbb{R}$-linearized BVP, which concludes the proof.

\end{proof}

The analysis of the nonlinear transmission condition is obtained by applying
the preceding impedance calculation on the two sides of the interface and
then taking the corresponding jumps. We state the resulting boundary value
problem below and comment on the derivation in the following remark.

\begin{coro}\label{th:shape_derivative_transmission}
Assume that the nonlinear transmission condition \eqref{eq:transmission} holds on the interface \(\Gamma\), the associated linearized transmission problem is well-posed, and $\mathbf{g}$ satisfies the regularity requirements of Assumption~\ref{ass:enhanced_g}. Then the shape derivative $\delta\mathbf{E}$ is the unique weak solution to the system
\begin{equation}
\label{eq:shape_transmission_volume}
\nabla \times \left( \frac{1}{\mu_r} \nabla \times \delta\mathbf{E} \right) - k^2 \varepsilon_r \delta\mathbf{E} = 0 \quad \text{in } \Omega \setminus \Gamma.
\end{equation}
On the interface $\Gamma$, the field $\delta\mathbf{E}$ satisfies the transmission conditions
\begin{equation}
\label{eq:shape_transmission_conditions}
\left\{
\begin{aligned}
\bigl[ \mathbf{n} \times \delta\mathbf{E} \bigr] &= -\operatorname{Curl}_\Gamma \bigl( h_n [ \mathbf{n} \cdot \mathbf{E} ] \bigr) - h_n [ \gamma_T (\nabla \times \mathbf{E}) ], \\[5pt]
\left[ \frac{1}{\mu_r} \mathbf{n} \times (\nabla \times \delta\mathbf{E}) \right] &- \mathbf{g}_{\boldsymbol{z}}(\boldsymbol{x}, \gamma_T \mathbf{E}; \gamma_T \delta\mathbf{E}) = \mathcal{G}_{\mathbf{h}},
\end{aligned}
\right.
\end{equation}
where the boundary data $\mathcal{G}_{\mathbf{h}}$ is defined by
\begin{equation}
\label{eq:shape_transmission_rhs}
\begin{aligned}
\mathcal{G}_{\mathbf{h}} &:= -\operatorname{Curl}_\Gamma \left( h_n \left[ \frac{1}{\mu_r} \operatorname{curl}_\Gamma(\gamma_T\mathbf{E}) \right] \right) - h_n k^2 [\varepsilon_r] \gamma_T \mathbf{E} \\
&\quad + h_n \left( \frac{\partial}{\partial \mathbf{n}} \mathbf{g}(\boldsymbol{x}, \gamma_T \mathbf{E}) - \mathcal{S} \bigl( \mathbf{g}(\boldsymbol{x}, \gamma_T \mathbf{E}) \bigr) \right) + h_n \mathfrak{H} \, \mathbf{g}(\boldsymbol{x}, \gamma_T \mathbf{E}) \\
&\quad + \mathbf{g}_{\boldsymbol{z}} \bigl( \boldsymbol{x}, \gamma_T \mathbf{E}; [ \mathbf{n} \cdot \mathbf{E} ] \nabla_\Gamma h_n \bigr).
\end{aligned}
\end{equation}
On the artificial boundary $\Gamma_R$, the shape derivative satisfies the transparent boundary condition
\begin{equation}
\label{eq:shape_transmission_radiation}
\mathrm{i}k \, \Lambda(\gamma_t \delta\mathbf{E}) = \gamma_t (\nabla \times \delta\mathbf{E}) \quad \text{on } \Gamma_R.
\end{equation}
This condition ensures that $\delta\mathbf{E}$ admits a unique extension as a radiating weak solution to the exterior domain $\mathbb{R}^3 \setminus \overline{D}$.
\end{coro}
\begin{rema}
The proof of the transmission result is omitted only to avoid repeating the
arguments used for the nonlinear impedance condition. After the covariant
Piola pullback, the transmission problem has the same primal curl--curl
variational structure as the impedance problem, with the only differences
being the piecewise constant coefficients \(\mu_r^{-1}\) and \(\varepsilon_r\),
the use of one-sided traces on the interface, and the appearance of jump
terms in the final boundary data. The transported tangential trace, the
surface Jacobian, and the nonlinear Nemytskii expansion are treated in exactly
the same way as before. Applying the impedance calculation on both sides of
\(\Gamma\) and then taking the jump yields the interface conditions stated in
Corollary~\ref{th:shape_derivative_transmission}. In particular, the tangential
deformation terms cancel by the same surface-divergence identities, and the
resulting shape derivative has the Hadamard structure, depending only on the
normal velocity \(h_n=\mathbf h\cdot\mathbf n\). 
\end{rema}


\end{document}